%% file: rwt7.tex
\documentclass[11pt]{amsart}

\input{rwt7-not.tex}

\usepackage[style=rwtmath, dashed=true]{biblatex}
\begin{document}
\author{Roelof Bruggeman}
\address{Mathematisch Instituut, Universiteit Utrecht, Postbus 80010,
3508 TA Utrecht, Nederland}
\email{r.w.bruggeman@uu.nl}

\author{YoungJu Choie}
\address{Department of Mathematics and PMI, Postech, Pohang, Korea
790--784}
\email{yjc@postech.ac.kr}

\author{Anke Pohl}
\address{University of Bremen, Department 3 -- Mathematics, Institute
for Dynamical Systems, Bibliothekstr.~5, 28359 Bremen, Germany}
\email{apohl@uni-bremen.de}

\title[Period functions, Maass forms and Jacobi Maass forms]{Period
functions for vector-valued Maass cusp forms of real weight, with an
application to Jacobi Maass cusp forms}

\begin{abstract}
For vector-valued Maass cusp forms for~$\SL_2(\ZZ)$ with real
weight~$k\in\RR$ and spectral parameter $s\in\CC$, $\re s\in (0,1)$,
$s\not\equiv \pm k/2$ mod~$1$, we propose a notion of vector-valued
period functions, and we establish a linear isomorphism between the
spaces of Maass cusp forms and period functions by means of a
cohomological approach. The period functions are a generalization of
those for the classical Maass cusp forms, being solutions of a
finite-term functional equation or, equivalently, eigenfunctions with
eigenvalue~$1$ of a transfer operator deduced from the geodesic flow on
the modular surface. We apply this result to deduce a notion of period
functions and related linear isomorphism for Jacobi Maass forms of
weight $k+1/2$ for the semi-direct product of $\SL_2(\ZZ)$ with the
integer points $\hei(\ZZ)$ of the Heisenberg group.
\end{abstract}



\maketitle

\input{rwt7-intro}

\input{rwt7-ack}

\input{rwt7-Mf}

\input{rwt7-ps}

\input{rwt7-pi}

\input{rwt7-tro}

\input{rwt7-abg}

\input{rwt7-pfcf}

\input{rwt7-JMf}

\bigskip
\raggedright \printbibliography
\bigskip

\input{rwt7-ind}

\tableofcontents

\end{document}

%% file: rwt7-intro.tex
\section{Introduction}

For several hyperbolic orbisurfaces~$\Gm\backslash\uhp$, with $\uhp$ denoting the hyperbolic plane and $\Gm$ being a discrete subgroup of~$\SL_2(\RR)$ acting by fractional linear transformation on~$\uhp$, notions of period functions for Maass forms and associated isomorphisms have been established in the course of the last years. For $\Gm=\SL_2(\ZZ)$, which is the seminal example, this has been achieved in combination of work by E.~Artin~\cite{Artin}, Series~\cite{Series}, Mayer~\cite{May90, May91}, Lewis~\cite{L97}, Bruggeman~\cite{Bruggeman_lewiseq}, Chang--Mayer~\cite{CM99}, and Lewis--Zagier~\cite{LZ_survey, LZ01}. Alternative proofs are given in~\cite{BM09,MMS12}, and most recently, by combination of~\cite{Pohl_Symdyn2d, Moeller_Pohl, Pohl_mcf_Gamma0p, Pohl_mcf_general}.

The variant of these proofs most relevant for our work proceeds roughly as follows, applying to Maass cusp forms. See also the survey~\cite{PZ20}. The space of Maass cusp forms for~$\SL_2(\ZZ)$ with spectral parameter~$s$ is shown to be linear isomorphic to the space of parabolic $1$-cohomology of~$\SL_2(\ZZ)$ with module being the vector space of smooth, semi-analytic vectors of the principal series representation with spectral parameter~$s$.
The cocycle classes can be characterized by real-analytic, rapidly decaying solutions of a rather simple functional equation on~$(0,\infty)$ that depends on~$s$. In this way, Maass cusp forms with spectral parameter~$s$ are seen to be linear isomorphic to real-analytic functions on~$(0,\infty)$ or, equivalently, holomorphic functions on~$\CC\setminus(-\infty,0]$ that satisfy the $s$-dependent functional equation and certain decay properties at boundaries. This isomorphism is given by an integral transform relation, and the solutions of this functional equation are the \emph{period functions}.

The functional equation can be deduced from the dynamics of the modular surface~$\SL_2(\ZZ)\backslash\uhp$: A well-chosen discretization of the geodesic flow on~$\SL_2(\ZZ)\backslash\uhp$ gives rise to a discrete dynamical system on~$(0,\infty)$, more precisely to a finitely-branched self-map on~$(0,\infty)\setminus\QQ$,  which is closely related to the Farey map. The associated transfer operator with parameter~$s$, called \emph{slow transfer operator}, is finite-term. The defining equation of its eigenfunctions with eigenvalue~$1$ is just the function equation from above.

Furthermore, an induction on parabolic elements in the discretization gives rise to a companion discrete dynamical system, an infinitely-branched self-map on $(0,\infty)$, closely related to the Gauss map and continued fraction expansions. The family of  \emph{fast transfer operators} associated to this map represents the Selberg zeta function via its Fredholm determinant and can be used to characterize necessary decay properties of period functions.

Turning around the order of this presentation (as it is done for generalizations), the geodesic flow by means of discretizations and transfer operator techniques gives rise to a functional equation suitable for the notion of period functions. And indeed the regularity and decay properties to be requested from period functions as well as the construction of the cohomology theory can partly be motivated by geometric-dynamical considerations. We refer to~\cite{PZ20} and~\cite[Section~8]{tahA} for more explanations.

These results have been generalized to certain classes of hyperbolic orbisurfaces of finite and infinite area. See in particular~\cite{CM99, Moeller_Pohl,
Pohl_mcf_Gamma0p, Pohl_spectral_hecke,Pohl_mcf_general, Pohl_Symdyn2d,
tahA}. With this paper we reach out to establish first instances of analogous results beyond Maass forms of weight~$0$ as well as beyond hyperbolic orbisurfaces. We provide such results for 
\begin{enumerate}[label=(\alph*)]
\item Jacobi Maass cusp forms of any real weight for the discrete (integral) Jacobi group of level~$1$, which is the semi-direct product of~$\SL_2(\ZZ)$ and the integer points~$\hei(\ZZ)$ of the Heisenberg group, and
\item vector-valued Maass cusp forms for~$\SL_2(\ZZ)$ of any real weight and any unitary representation.
\end{enumerate}

To survey our results in more detail, we start with a few preparatory comments and explanations. We set throughout $G \coloneqq \SL_2(\RR)$ and $\Gm\coloneqq \SL_2(\ZZ)$. We let $\hei$ denote the $3$-dimensional continuous Heisenberg group and $\hei(\ZZ)$ the discrete Heisenberg group, i.e., the subgroup of~$\hei$ given by restricting to the ring of integers. See Section~\ref{sect-JMf} for precise definitions. 

The space on which Jacobi Maass forms (and Jacobi Maass cusp forms) are defined is the product space~$\uhp\times\CC$ of the hyperbolic plane~$\uhp$ and the complex plane~$\CC$. On this space, the (continuous) Jacobi group $G^J \coloneqq \hei\rtimes G$ (of level~$1$) acts by fractional linear transformations in the $\hei$-component and by a certain skew product in the $\CC$-component. See Section~\ref{sect-JMf}. Endowing $\uhp\times\CC$ with a Riemannian metric such that $G^J$ acts by Riemannian isometries is not unique. Indeed, there is at least a two-parameter family of such Riemannian metrics on~$\uhp\times\CC$ (see \cite[Remark~2.5]{Yang}). Thus, if we wanted to proceed as for hyperbolic orbisurfaces starting with a discretization of the geodesic flow, then we would face the difficulty of the non-uniqueness of the choice of this flow. In addition, even if we settled on one choice of the Riemannian metric, then we would need to handle the seven-dimensional sphere bundle of~$\uhp\times\CC$ in combination with a mixture of hyperbolic and euclidean action behavior.

To circumvent this obstacle and to simultaneous stay close to the approach for hyperbolic orbisurfaces, we use here another approach based on the theta decomposition. Pitale~\cite[Theorem~4.6]{Pit09} showed that Jacobi Maass forms for $\Gm^J \coloneqq \hei(\ZZ)\rtimes\Gm$ of integral weight and positive integral index are linear isomorphic to certain spaces of vector-valued Maass forms on $\Gm\backslash\uhp$. Thus, this isomorphism allows us to transfer the quest for a notion of period functions for Jacobi Maass cusp forms to a question about period functions for vector-valued Maass cusp forms in the more well-known realm of hyperbolic surfaces. This way, the request for a discretization of the non-unique geodesic flow on $\Gm^J\backslash(\uhp\times\CC)$ is solved implicitly and essentially avoided. However, via this isomorphism the \emph{integral weight} of Jacobi Maass forms gets converted into a \emph{half-integral weight} for the vector-valued Maass forms. Up to date, only \emph{weight-zero situations} have been considered in the literature in this realm of research, and hence we are required to find a notion of period functions for Maass cusp form for half-integral weight and establish the necessary linear isomorphism. Indeed, we provide these results for \emph{arbitrary real weight} as it is no more difficult than half-integral weights. In turn, the generality of our results for vector-valued Maass cusp forms then allows us to consider also arbitrary real weight for Jacobi Maass cusp forms.

For the definition of vector-valued Maass forms of weight~$k\in\RR$ we fix a one-dimensional multiplier system $v_k$, given in~\eqref{vk}, and a unitary representation~$\rho$ of~$\Gm$ on a finite-dimensional vector space~$X_\rho$. \emph{Maass forms} of weight~$k$, spectral parameter~$s$ and multiplier system~$\rho v_k$ are smooth eigenfunctions $\uhp\to X_\rho$ with eigenvalue $s(1-s)$ of the generalized Laplacian
\[
\Dt_k \coloneqq - y^2\,\partial_x^2 - y^2\,\partial_y^2 + i k y\qquad  (z=x+iy\in\uhp)
\]
with growing at most polynomial towards $\infty$ and being invariant under the action $|_{\rho v_k,k}$ on all of~$\Gm$, where
\[
u|_{\rho v_k,k}\gm(z) \coloneqq \rho(\gm)^{-1}v_k(\gm)^{-1} e^{-ik\arg(cz+d)}u(\gm z)
\]
for $u\colon \uhp\to X_\rho$, $\gm\in\Gm$, $z\in\uhp$ and $\gm = \rmatc abcd$. The space of such Maass cusp forms is denoted $\A_k(s,\rho v_k)$. Asking for exponential decay towards $\infty$ instead of polynomial bounds defines the space of \emph{Maass cusp forms} of weight~$k$, spectral parameter~$s$ and multiplier system~$\rho v_k$, which is denoted $\A^0_k(s,\rho v_k)$. See Section~\ref{sect-Mcf} for more details. This section also contains an alternative definition using the universal covering group of~$G$, which is helpful for our considerations.

The space~$\FE^\om_{\rho v_k,s,k}$ of period functions for weight~$k$, spectral parameter~$s$ and multiplier system~$\rho v_k$ consists of real-analytic functions $f\colon (0,\infty)\to X_\rho$ that obey certain extension properties and satisfy the three-term functional equation
\[
f \= f|^\prs_{\rho v_k,s,k} T + f|^\prs_{\rho v_k,s,k}T' \qquad\text{with $T\coloneqq \rmatc 1101\ $ and $T'\coloneqq\rmatc1011$}\,,
\]
where the action~$|^\prs_{\rho v_k,s,k}$ is closely related to the action~$|_{\rho v_k,k}$. We refer to Section~\ref{sect-prs} for precise definitions. We emphasize that this functional equation can be deduced from a transfer operator associated to a discretization of the geodesic flow on the modular surface~$\Gm\backslash\uhp$, and hence the same discretization as for the weight-zero results is lurking in our considerations.

\begin{mainthm}\label{mainthmA}
For $k\in\RR$ and $s\in\CC$ such that $\re s\in (0,1)$ and $s\not\equiv \pm k/2\bmod 1$, the vector spaces $\A^0_k(s,\rho v_k)$ and $\FE^\om_{\rho v_k,s,k}$ are isomorphic. 
\end{mainthm}

The isomorphism in Theorem~\ref{mainthmA} is indeed constructive, at least in the direction $\A^0_k(s,\rho v_k)\to\FE^\om_{\rho v_k,s,k}$. It is given by an integral transform and uses a cohomological setting. The condition $s\not\equiv \pm k/2\bmod 1$ in Theorem~\ref{mainthmA} restricts the spectral parameter~$s$ to the values for which the discrete series representation is irreducible. We refer to the full statement of this isomorphism in Theorem~\ref{thm-al}, Proposition~\ref{prop-pfMcf} and their proofs.

In our present consideration we restrict to~$\Gm=\SL_2(\ZZ)$ for definiteness and simplification of some steps. In particular, we may work with the Farey tesselation of~$\uhp$, which is underlying both the transfer-operator-based deduction of the functional equation above as well as some parts in the cohomological argumentations. However, we expect that the tools we use here for the generalization of the weight-zero results to arbitrary real weights can be adapted to non-cusp forms and other discrete subgroups of~$G$.

\emph{Jacobi Maass forms} for $\Gm^J$ of index $m\in\ZZ$, weight~$k\in\RR$, eigenvalue parameter~$s\in\CC$ and multiplier system~$\varphi_a v_k$ with $\varphi_a$ being a character parameterized by $a\in\ZZ\bmod 12$, are smooth functions $\uhp\times\CC\to\CC$ that are eigenfunctions of certain differential operators related to the Laplacian on~$\uhp$, that are of at most polynomial growth, and that are invariant under the action~$|^J_{\ph_a v_k,k,m}$, which is an extension of the action on Maass forms to include the elliptic variable space $\CC$. We refer to Section~\ref{sect-JMf} for details and an alternative definition using the universal covering group. The space of such Jacobi Maass forms is denoted~$\A^J_{k,m}(s,\ph_a \ch_k)$. Asking for quick decay instead of polynomial growth, defines the subspace of \emph{Jacobi Maass cusp forms}, denoted~$\A^{J,0}_{k,m}(s,\ph_a \ch_k)$. We obtain the following generalization of Pitale's result, where $\rho_{a,m}$ is a certain unitary representation of~$\Gm$ which is associated to $\varphi_a$ and $m$, defined in~\eqref{def:rhoam}.

\begin{mainthm}\label{mainthmB}
Let $m\in\ZZ$, $m\geq 1$, $a\in \ZZ/12$, $k\in\RR$ and $s\in\CC$ with $\re s\geq 0$. Set $s'\coloneqq (s+1)/2$ and $k'\coloneqq k -1/2$. Then the vector spaces $\A^J_{k,m}(s,\ph_a \ch_k)$ and $\A_{k'}(s',\rho_{a,m} v_{k'})$ are isomorphic, as well as the vector spaces $\A^{J,0}_{k,m}(s,\ph_a \ch_k)$ and $\A^0_{k'}(s',\rho_{a,m} v_{k'})$.
\end{mainthm}

Theorem~\ref{mainthmB} provides a theta decomposition, which generalizes Pitale's result. It is based on working with the universal covering group of~$G^J$ and Fourier expansions. Indeed, the proof provides more insights into the isomorphism. We refer to Theorem~\ref{thm-V} and its proof for full details.

The combination of (the full versions of) Theorem~\ref{mainthmA} and Theorem~\ref{mainthmB} yields period functions for Jacobi Maass cusp forms.

\begin{mainthm}\label{mainthmC}
Let $m\in\ZZ$, $m\geq 1$, $s\in\CC$ and $k\in\RR$ such that $\re s\in [0,1)$ and $s\not\equiv \pm k\bmod 2$. Set $s' \coloneqq (s+1)/2$ and $k'\coloneqq k-1/2$. Then the vector spaces $\A^{J,0}_{k,m}(s,\ph_a \ch_k)$ and $\FE^\om_{\rho_{a,m}v_{k'},s',k'}$ are isomorphic.
\end{mainthm}

This result is stated as Corollary~\ref{cor-pfJM} in Section~\ref{sect-JMf}. As for Theorem~\ref{mainthmA}, the isomorphism between Jacobi Maass cusp forms and their period functions can be provided rather explicitly via an integral transform. We refer to Proposition~\ref{prop-isoJMPF} for details.

\medskip

This article is structured as follows. In Section~\ref{sect-Mcf} we introduce Maass forms and Maass cusp forms for arbitrary real weight, first as functions on~$\uhp$ and then as functions on the universal covering group of~$G$. We discuss their Fourier expansions, and weight-increasing and weight-lowering between Maass forms, which yields that only spectral parameters $s$ with $\re s\in (0,1)$ need to be considered (Proposition~\ref{prop-sppMcf}). In Section~\ref{sect-prs} we discuss principal series representations and discrete series representations in the presence of arbitrary real weight. We further provide the definition of period functions and show some first properties. In Section~\ref{sect-pi} we intensify the discussion of period functions, provide the integral transform including the generalization of all necessary ingredients, present the cohomology setting, in particular the parabolic cocycles, and establish an explicit linear map from Maass cusp forms to period functions (Proposition~\ref{prop-pfMcf}). In Section~\ref{sect-tro} we discuss the relation between slow/fast transfer operators and period functions. We obtain that, as in the classical results, slow transfer operators determine the functional equation (and some parts of the regularity conditions) of period functions, and fast transfer operators help to characterize the necessary regularity conditions. In Section~\ref{sect-abg} we start working on showing that the linear map from Maass cusp forms to period functions is indeed bijective by indeed inverting this map, i.e., the integral transform. To that end we provide a kernel function for the inversion, and a boundary germ construction. In Section~\ref{sect-pfcf} we complete these efforts by constructing the inverse map from period functions to Maass cusp forms. In Section~\ref{sect-JMf} we provide a generalization of Jacobi Maass forms to arbitrary real weight, we establish a theta decomposition for them, allowing us to relate Jacobi Maass forms and vector-valued Maass forms for real weight,  and we apply our result on period functions for Maass cusp forms to obtain period functions for Jacobi Maass cusp forms. Throughout we attempt to follow the proofs in the previous results mentioned at the beginning of this introduction as close as possible, and we emphasize the new tools and steps needed for the generalizations.

%% file: rwt7-ack.tex
\begin{ack}
The research of AP is funded by the Deutsche Forschungsgemeinschaft (DFG, German Research Foundation) -- project no.~441868048 (Priority Program~2026 ``Geometry at Infinity''). YC is partially supported by NRF2022R1\-A\-2B5B0100187113.
\end{ack}

%% file: rwt7-Mf.tex


\section{Maass forms}\label{sect-Mcf}
We consider in Subsection~\ref{sect-Mfuhp} Maass forms first as
functions on the complex upper half-plane that are eigenfunctions of a generalized Laplace operator,
satisfy an invariance relation and a
growth condition, and we discuss their Fourier expansion in
Subsection~\ref{sect-Fe}.

For our purposes it is useful to consider Maass forms also as functions on a Lie group
covering $\SL_2(\RR)$. We discuss this in Subsection~\ref{sect-ucg}.
Further, the action of the Lie algebra of $\SL_2(\RR)$, in
Subsection~\ref{sect-wso}, can be used to relate Maass forms in weights
that differ by a multiple of $2$. This leads to
Proposition~\ref{prop-sppMcf}, which reduces the set of eigenvalues of
the Laplace operator that we have to consider.

\subsection{Maass forms on the upper half-plane}\label{sect-Mfuhp}We
first discuss the concepts involved in the definition of Maass forms,
working with functions on the complex upper half-plane
\il{uhp}{$\uhp$}$\uhp=\bigl\{z\in \CC\;:\; \im z>0\bigr\}$.

\rmrk{Differential equation}Maass forms of weight \il{k}{$k\in \RR$
weight}$k\in\RR$ should be eigenfunctions of the differential
operator\ir{Dtk}{\Dt_k}
\be\label{Dtk} \Dt_k \= - y^2\,\partial_x^2 - y^2\,\partial_y^2 + i k y
\,
\partial_x\,.\ee
Here and further on we will tacitly write $z\in \uhp$ as $z=x+iy$ with
$x\in \RR$, $y>0$.\il{x}{$x$}\il{y}{$y$} For $k=0$, the differential
operator~$\Dt_ 0 $ is the hyperbolic Laplace operator on $\uhp$. For
any $k\in \RR$ the operator $\Dt_k$ is elliptic, and all its
eigenfunctions are real-analytic. The operator $\Dt_k$ makes sense on
vector-valued functions by applying it on each coordinate component. We
follow the practice of parametrizing eigenvalues as
\il{s}{$s$}$s(1-s)$ with $s\in\CC$, and call $s$ and $1-s$
\il{spm}{spectral parameter}\emph{spectral parameters}. Maass forms
satisfy the condition
\be\label{eifc} \Dt_k u \= s(1-s)u\ee
for some $s\in \CC$.

\rmrk{Invariance under the modular group}For each
\il{G}{$G=\SL_2(\RR)$}$g=\rmatc abcd\in \SL_2(\RR)=:G$ we use the
operator\ir{slash-k}{|_k}
\be \label{slash-k} u \mapsto u|_k g \,,\quad \bigl( u|_k g\bigr) (z) \=
e^{-ik \arg(cz+d)
} \, u(g z)\,,\ee
where we take \il{ac}{argument convention}$-\pi < \arg(c z+d) \leq \pi$.
By $g z$ we mean $\frac{a z+b}{cz+d}$. The operators can be applied to
vector-valued functions. The operators $|_k g$ commute with the
operator $\Dt_k$.

For $k\in \RR\setminus \ZZ$ the operator $|_k g$ depends on the choice of the
argument.  This  has the consequence that for
\il{km}{$\km(\cdot)$}$g= \km(\th)\coloneqq
\rmatr{\cos\th}{\sin\th}{-\sin\th}{\cos\th}$ and $z=i$ the factor
$e^{-i k \arg(\cos \th - i \sin \th)}$ is right continuous in
$\th=-\pi$, but not continuous if $k\not\in 2\ZZ$.

If $k\in \ZZ$, then $g\mapsto |_k g$ is a
(right) representation of $G$ in the functions on~$\uhp$:
\be u|_k g_1g_2 \= \bigl( u|_k g_1\bigr)\bigm|_k g_2\qquad (g_1,g_2\in
G)\,.\ee
For $k\in \RR\setminus \ZZ$ this relation holds only up to a factor with
absolute value $1$. The operators $|_k g$ with $g\in G=\SL_2(\RR)$
generate a group, which depends on $k$. This group is a homomorphic
image of the universal covering group of $\SL_2(\RR)$, which we will
discuss in Subsection~\ref{sect-ucg}.

It is impossible to add a factor in the definition in~\eqref{slash-k}
such that we arrive at a representation of the group $G$. However we
can turn $\gm\mapsto |_k \gm$ into a representation of the discrete
subgroup \il{Gm}{$\Gm$}$\Gm=\SL_2(\ZZ)$ by writing
\be\label{slash-vk} u|_{v_k,k}\gm (z) \= v_k(\gm)^{-1}\, e^{-i k \arg(c
z+d)} \, u(\gm z)\,, \qquad \text{ with } \gm= \rmatc abcd \ee
and the function $v_k \colon \Gm \rightarrow\CC^\ast$ given by\ir{vk}{v_k}
\badl{vk} v_k(\gm) &\= \frac{ \eta^{2k}(\gm z)}{(c_\gm z+d_\gm)^k \,
\eta^{2k}(z)}\,,\\
\eta^{2k}(z) &\= e^{\pi i k z/6} \prod_{n \geq 1} \Bigl(1-e^{2\pi i n
z}\Bigr)^{2k}\,,
\eadl
where $\eta$ is the Dedekind eta function. For the multiplier system
$v_k$, we have in fact an action of $\PSL_2(\ZZ)= \Gm/\{\pm I_2\}$ with
$I_2=\rmatr{1}00{1}$, since $|_{v_k,k}\gm = |_{v_k,k}(-\gm)$ for all
$\gm \in \Gm$.

We let $X_\rho= \CC^n$ for some
\il{nrho}{$n(\rho)$}$n=n(\rho)\in \ZZ_{\geq 1}$, and consider a unitary
representation
\il{rho}{$\rho$}\il{Xrho}{$X_\rho$}$\rho\colon\Gm \rightarrow \U(X_\rho)$
with respect to the standard inner product
\il{iprho}{$(x,y)_\rho$}$(x,y)_\rho = \sum_l x_l \bar y_l$ of
$\CC^n=X_\rho$. We obtain a representation on the $X_\rho$-valued
functions on $\uhp$ by\ir{|rhockk}{|_{\rho v_k,k}}
\be \label{|rhockk}
u|_{\rho v_k,k}\gm \= \rho(\gm)^{-1} \, u|_{v_k,k}\gm\,.\ee
When dealing with the Jacobi group we will obtain examples of
representations with these properties.

\rmrk{Growth conditions}\il{grc}{growth conditions}On functions $u$ that
are invariant under $|_{\rho v_k,k }\Gm$ we impose growth conditions at
$\infty$. A function $u$ has \il{pg}{polynomial growth}\emph{polynomial
growth} if
\be\label{pg} u(z) \= \oh(y^a) \quad\text{ as }y\uparrow\infty\ee
for some $a\in \RR$ that may depend on~$u$, uniform for $x$ in compact
sets in~$\RR$. A function $u$ has \il{ed}{exponential
decay}\emph{exponential decay} if for some $a>0$
\be\label{ed} u(z) \= \oh \bigl(e^{-a y}\bigr) \quad\text{ as }y\uparrow
\infty, \text{ uniformly for $x$ in compact sets in $\RR$}\,.\ee

\begin{defn}\label{defMf}Let $s\in \CC$, $k\in \RR$, and
$\rho\colon\Gm\rightarrow \U(X_\rho)$ a finite-dimensional unitary
representation of $\Gm$. The space $\A_k(s,\rho v_k)$ of \il{Mf}{Maass
form}\emph{Maass forms} of \il{wt}{weight}weight $k$, with spectral
parameter $s$ and \il{msys}{multiplier system}multiplier system
$\rho v_k$ consists of all smooth functions $u\colon\uhp\rightarrow X_\rho$
that satisfy $u|_{\rho v_k,k}\gm = u$ for all $\gm\in \Gm$, the
eigenfunction condition \eqref{eifc}, and the condition \eqref{pg} of
polynomial growth. The stronger condition \eqref{ed} of exponential
decay determines the subspace $\A^0_k(s,\rho v_k)$ of \emph{Maass cusp
forms}.\il{A0}{$A_k(s,\rho v_k) \supset \A^0_k(s,\rho v_k)$}\il{Mcf}{Maass
cusp form}
\end{defn}

The presence of \il{I2}{$I_2$}$-I_2 = \rmatr{-1}00{-1}$ in $ \Gm$
requires attention. We have
\be v_k(-I_2) e^{-i k \arg(-1)}\=1\ee
from~\eqref{vk}. Hence we get $\rho(-I_2) u(z) = u(z)$. So functions
that are invariant under $|_{\rho v_k,k}\Gm$ have values in the
$1$-eigenspace of $\rho(-I_2)$. We might avoid this by requiring that
$\rho$ is a representation of $\PSL_2(\ZZ)$. However, the
representation $\rho$ might arise naturally, and it might be
inconvenient to tamper with it.

The concept of Maass form in the scalar-valued case for more general
groups than $\SL_2(\ZZ)$ is due to H.~Maass, who called them
\emph{non-analytic modular forms}; see \cite{Ma49} and \cite[p.~185]{Ma83}. For
vector-valued Maass forms we may consult Roelcke
\cite{Roelcke66,Roelcke67}.

\subsection{Fourier expansion}\label{sect-Fe}
Let the function $u\colon\uhp\rightarrow X_\rho$ be equivariant under
$|_{\rho v_k,k}T$ with \il{T}{$\rmatc1101$}$T=\rmatc1101$. The operator
$v_k(T)\rho(T)$ is unitary, and $X_\rho$ has an orthonormal basis of
eigenvectors
\il{el}{$e_l\in X_\rho$}$\bigl\{ e_l: 1\leq l \leq n(\rho)\bigr\}$ of
$v_k(T)\rho(T)$. We take \il{kpl}{$\k_l$}$\k_l\in [0,1) $ such that
$\rho(T) v_k(T) e_l = e^{2\pi i \k_l}\, e_l$. Writing
$u\colon\uhp\rightarrow X_\rho$ in the form
\be \label{ucomp} u(z) \= \sum_l u_l(z)\, e_l\,,\ee
we get $n(\rho)$ component functions $u_l\colon\uhp\rightarrow \CC$.

If $u\in \A^0_k(s,\rho v_k)$, then the functions $u_l$ have a Fourier
expansion\il{Fe}{Fourier expansion}
\be\label{Fe} u_l(z) \= \sum_{\substack{n \equiv \k_l\bmod 1, \\ n\neq 0}} c_l(n)
e^{2\pi i n x}\, W_{\e(n)k/2,s-1/2}(4\pi|n|y)\,,\ee
with $\e(n) = \sign(n)$. We note that $n$ runs through a set of real
numbers, not necessarily integers. Since the $W$-Whittaker functions
and their derivatives have exponential decay, all derivatives of $u$
with respect to $x$ and $y$ satisfy condition~\eqref{ed} of exponential
decay. Under less strict assumptions than the unitarity of $\rho$ there
still is a Fourier expansion in which $W$-Whittaker functions are
involved; see~\cite{FPR}. The
exponential decay of derivatives of Maass cusp forms goes through.

The Fourier terms of components $u_l$ of a function
$u\colon\uhp\rightarrow X_\rho$ satisfying only \eqref{eifc} and
\eqref{|rhockk} are more general. For $\re s>0$ any term with order
$n\neq 0$ is a linear combination of
\be \label{genFt} e^{2\pi i n x}\, W_{\e(n)k/2,s-1/2}(4\pi|n|y)
\quad\text{ and }\quad e^{2\pi i n x}\, M_{\e(n)k/2,s-1/2}(4\pi|n|y)\,.\ee
The $W$-Whittaker function has exponential decay, and the $M$-Whittaker
function has exponential growth. The term of order $0$ is for
$s\neq \frac12$ a linear combination of $y^s$ and $y^{1-s}$, and for
$s=\frac12$ a linear combination of $y^{1/2}$ and $y^{1/2}\log y$.

Fourier terms inherit growth conditions. A consequence is that if we
replace the condition of exponential growth of Maass cusp forms by
$u(z) \= \oh\bigl( y^{-a} \bigr)$ as $y\uparrow \infty$ with
$a>\max(\re s, 1-\re s)$, then we have Fourier expansions of the $u_l$
as indicated in~\eqref{Fe}. This is the condition of \il{qd}{quick
decay}\emph{quick decay}. It is weaker than exponential decay, but for
Maass forms quick decay implies exponential decay.

\subsection{Universal covering group}\label{sect-ucg} The weight $k$ of
a Maass form as defined in Definition~\ref{defMf} is a parameter in the
transformation behavior by elements of~$\Gm$. The concept of Maass
forms on a Lie group separates the weight and the $\Gm$-invariance. In
the context of arbitrary real weights the Lie group to be used is the
universal covering group of $\SL_2(\RR)$.

\rmrk{Description of the universal covering group}The \il{Id}{Iwasawa
decomposition}Iwasawa decomposition of $\SL_2(\RR)$ writes each element
of $\SL_2(\RR)$ uniquely as $\ppm(z)\km(\th)$, $z\in \uhp$ and
$\th\in \RR\bmod 2\pi \ZZ$, with
\bad \ppm(x+i y) &\= \rmatc {y^{1/2}}{xy^{-1/2}}0{y^{-1/2}}&
&x+i y \in \uhp\,,\\
\km(\th)&\=\rmatr {\cos\th}{\sin\th}{-\sin\th}{\cos\th}&&\th\in \RR\,.
\ead
The \il{ucg}{universal covering group}universal covering group
\il{tG}{$\tG$}$\tG$ of $G=\SL_2(\RR)$ is based on the covering
\[ \uhp \times \RR \rightarrow \uhp \times \RR/2\pi \ZZ\,,\]
with the natural map $\RR \rightarrow \RR/2\pi \ZZ$. There exists a
unique Lie group $\tG$ consisting of the elements
\il{tppm}{$\tppm(z)$}\il{tkm}{$\tkm(\th)$}$\tppm(z)\tkm(\th)$ with
$z\in \uhp$ and $\th\in \RR$ with a group structure such that
\ir{pr}{\pr}
\be \label{pr} \pr\colon \tppm(z)\tkm(\th) \mapsto \ppm(z)\km(\th)\ee
is a surjective Lie group homomorphism $\pr\colon\tG\rightarrow G$ with
kernel\ir{Z2}{\tZ_2}
\be \label{Z2} \tZ_2 \= \bigl \{ \tkm(2\pi n)\;:\; n\in \ZZ
\bigr\}\,.\ee
One finds a description of the group operations in~\cite[\S2.2.1]{B94}.

The map $\pr$ in~\eqref{pr} is a group homomorphism. There does not
exist an inverse group homomorphism, but we can choose a section
$g\mapsto \ell(g)$ from $G$ to $\tG$ of the homomorphism
$\pr\colon\tG\rightarrow G$ by\ir{tg}{\tilde g}
\be\label{tg}
\ell\rmatc abcd \= \tppm \Bigl( \frac{ai+b}{ci+d} \Bigr)
\tkm\bigl(-\arg(ci+d) \bigr). \ee
We stress that the map $\pr\circ \ell$ is the identity on $G$, but $\pr$
is not invertible. Further for all
$\rmatc abcd\in G$ we have
\be \ell\rmatc abcd \, \tilde \ppm(z) \tilde\km(\th) \= \tilde\ppm\Bigl(
\frac{az+b}{cz+d}\Bigr)\, \tilde\km\bigl(\th-\arg(c z+d) \bigr)\,,\ee
where we use the argument convention $-\pi < \arg(cz+d)\leq \pi$.

\rmrk{Weights and equivariance}A function $f $ on $\tilde G$ has
\emph{weight} $k\in \RR$ if it satisfies
\be f \bigl( \tilde g \tilde\km(\th) \bigr)
\= e^{ik\th}\, f(\tilde g)\qquad \text{ for all }\tilde g\in \tG, \;
\th\in \RR\,.\ee
The representation of $\tG$ by right translation in the functions
on~$\tG$ is defined by\ir{R}{R}
\be \label{R}
\bigl(R(\tilde g_1 )f\bigr)(\tilde g) \= f(\tilde g\tilde g_1)\,.\ee
Hence the function $f$ on $\tilde G$ has weight $k$ if the subgroup
$\tilde K = \bigl\{ \tilde\km(\th)\;:\; \th\in \RR\bigr\}$ of~$\tG$ acts
according to the character $\tilde \km(\th) \mapsto e^{ik\th}$
of~$\tilde K$. The representation of $\tG$ by left translation of
function on~$\tG$ is given by\ir{L}{L}
\be \label{L}\bigl( L(\tilde g_1) f \bigr) (\tilde g) \=f( \tilde
g_1\tilde g)\,.\ee
As we do not apply $\tilde g_1^{-1}$
in~\eqref{L}, this is a right representation:
\[
L(\tilde g_1 \tilde g_2) \= L(\tilde g_2) L(\tilde g_1)\,.
\]

Since left and right translations commute, left translation does not
change the weight of functions on~$\tilde G$.

\rmrk{Discrete subgroup}The discrete subgroup $\Gm=\SL_2(\ZZ) \subset G$
has an inverse image \il{tGm}{$\tGm$}$\tGm=\pr^{-1}\Gm$ in $\tG$. The
group $\tGm$ is discrete in $\tG$, and it contains the
center\ir{tZ}{\tZ}
\be \label{tZ}\tZ\= \bigl\{\tkm(\pi n)\;:\; n\in \ZZ\bigr\}\ee
of $\tG$. The group $\tGm$ is equal to $\tZ\, \ell(\Gm)$. The group
$\tGm$ is generated by\il{tm}{$\tm\in \tGm$}\il{sm}{$\sm\in \tGm$}
\be \tm \= \tppm(i+1) \quad\text{ and }\quad \sm \= \tkm(-\pi/2)\,,\ee
which implies that\il{S}{$S$}
\bad \tm &\= \ell(T)&\quad \text{ with } \quad T&\=\rmatc 1101\,,\\
\sm &\= \ell(S)&\quad \text{ with } \quad S&\;\coloneqq\; \rmatr0{-1}10\,.
\ead
The relations for $\tGm$ are determined by
\be \sm^2 \tm\= \tm \sm^2 \quad\text{ and }\bigl( \tm \sm)^3
\=\sm^2\,.\ee

Suppose that a function $f\in C^\infty(\tG)$ is left-invariant under
$\tGm$ and has weight~$k$. Since $\sm^2$ is central in~$\tG$ we have
\[ f(\tilde g) \= f\bigl( \sm^2 \tilde g\bigr) \= f\bigl( \tilde g
\sm^2\bigr) \= e^{-\pi i k } \, f(\tilde g)\,.\]
Hence we have $f=0$ if $k\in \RR\setminus \ZZ$.
Thus, for general real weights, functions on $\tG$ cannot be left invariant
under $\tilde \Gm$, and hence we have to be content to work with functions that
are left equivariant for a suitable character of $\tilde \Gm$, for
instance for the character $\ch_k$ determined by \ir{chk}{\ch_k}
\be\label{chk}
\ch_k(\tm) \= e^{\pi i k/6}\,,\qquad \ch_k(\sm) \= e^{-\pi i k/2}\,.\ee
Then we deal with $\ch_k$-equivariant functions of weight $k$ that
satisfy
\be f\bigl( \tilde \gm \tilde g) \= \ch_k(\tilde \gm) \, f(\tilde
g)\qquad \tilde\gm\in \tGm,\; \tilde g \in \tG\,.\ee

\rmrk{Functions on \intitle{\uhp} and on \intitle{\tG}} For any
$u\in C^\infty(\uhp)$ and any weight $k\in \RR$ we define
$\Ps_k u \in C^\infty(\tG)$ by\ir{Psk}{\Ps_k}
\be \label{Psk} (\Ps_k u)\bigl(\tppm(z) \tilde \km(\th) \bigr) \=
e^{ik\th}\, u(z)\,.\ee
The function $\Ps_k u$ has weight $k$. Moreover, the operator $|_k g$
in~\eqref{slash-k} corresponds to left translation on $\tG$:
\be \label{Psig} \Ps_k\bigl( u|_k g\bigr) \= L\bigl( \ell(g) \bigr)
\Ps_k u \qquad \text{ for all }g\in G\,.\ee

The operator $\Ps_k$ works for vector-valued functions as well as for
scalar-valued functions. We have
$v_k(\gm) \= \ch_k\bigl( \ell(\gm) \bigr)$ for all $\gm \in \Gm$. If
$\rho$ is a representation of $\Gm$ in $\CC^n$, then we define
$\rho(\tilde\gm) \coloneqq \rho(\pr\tilde\gm)$ for all $\tilde\gm\in \tGm$.
With \eqref{|rhockk} we get
\be \Ps_k\bigl( u |_{\rho v_k,k}\gm \bigr) \= \rho(\gm)^{-1} L(\gm)
\Ps_k u\,.\ee

\subsection{Lie algebra and differential operators}\label{sect-wso}The
groups $G$ and $\tG$ have isomorphic neighborhoods of the unit element,
and hence they have the same Lie algebra
\il{glie}{$\glie, \;\glie_c$}$\glie$, with complexification
$\glie_c= \CC\otimes_\RR \glie$. We use the notations of Berndt and
Schmidt, in particular for the basis $\{Z, X_+, X_-\}$ of the
complexified Lie algebra. See \cite[p.~12]{BS}.

The Lie algebra $\glie_c$ acts on $C^\infty(\tG)$ by differentiation on
the right. This action commutes with left translation. This holds in
particular for the following operators, in the coordinates given by
$(x,y,\th) \leftrightarrow \tppm(x+i y) \tkm(\th)$.
\il{Zl}{$Z\in \glie_c$}\il{Xpm}{$X_+, X_-\in \glie_c$}
\badl{opZXX} Z &\= - i \partial_\th\,,\\
X_+ &\=e^{2i \th}\Bigl(2iy\,\partial_z - \frac i2\,\partial_\th \Bigr)&&
\= e^{2i \th}\Bigl(iy\,\partial_x+ y \,\partial_y - \frac
i2\,\partial_\th
\Bigr)\\
X_-&= e^{-2i\th}\Bigl( -2iy\, \partial_{\bar z} + \frac i2 \partial_\th
\Bigr)
&&\= e^{-2i\th}\Bigl( -i y\,\partial _x +y \,\partial_y+ \frac i2
\,\partial_\th
\Bigr)\,.
\eadl
The differential operator $Z$ detects the weight of functions. With the
commutator relations in $\glie_c$ (see~\cite[p.~12]{BS}) we check that
$X_+$ \il{wsho}{weight shifting operator}shifts the weight up by $2$,
and $X_-$ shifts down by~$2$. We have the second order
element\ir{DtXpm}{\Dt}
\be \label{DtXpm} \Dt\= - X_-X_+-\frac14 Z^2-\frac12 Z \= -
X_+X_--\frac14 Z^2+\frac12 Z\,,\ee
which is known to commute with all elements of $\glie_c$. The operator
$\Dt$ is called the Casimir operator; it corresponds to the
differential operator
\be - y^2\bigl( \partial_x^2+\partial_y^2\bigr) + y\,
\partial_x\partial_\th\,.\ee
On functions of weight $k$ the operator $\Dt$ acts as
\be \label{Dt-Dtk}- y^2\bigl(\partial_x^2+\partial_y^2\bigr)+ i k y
\,\partial_x\,,\ee
which is the operator $\Dt_k$ in~\eqref{Dtk}.

Now we are ready to use the operator $\Ps_k$ to transform
Definition~\eqref{defMf} into an equivalent definition of Maass forms
as functions on $\tilde G$. We put \il{tam}{$\tilde\am(y)$}
\[
\tilde \am(y) \coloneqq \tppm(iy) =\ell\matc{y^{1/2}}00{y^{-1/2}}\,.
\]
\begin{defn}\label{defMftG}Let $s\in \CC$, $k\in \RR$, and
$\rho\colon\Gm\rightarrow \U(X_\rho)$ a finite-dimensional unitary
representation of $\Gm$. The space
\il{AktG}{$\A_k(s,\rho \ch_k)$}$\A_k(s,\rho \ch_k)$ of \il{MftG}{Maass
form}\emph{Maass forms} on $\tilde G$ of weight $k$, with spectral
parameter $s$ and representation $\ch_k\rho$ consists of all smooth
functions
\[
f \colon \tG\rightarrow X_\rho
\]
that satisfy
\begin{enumerate}[label=$\mathrm{(\alph*)}$, ref=$\mathrm{\alph*}$]
\item\label{eq:MFa} $R\bigl( \tilde\km(\th) \bigr)f \= e^{i k \th}\,f$ for all
$\th\in \RR$,
\item\label{eq:MFb}
$L(\tilde\gm) f = \rho(\tilde\gm)^{-1}\, \ch_k(\tilde\gm) \, f$ for all
$\tilde\gm\in \tGm$,
\item\label{eq:MFc} $\Dt f \= s(1-s) f$,
\item\label{eq:MFd} $f \bigl(\tilde\am(t)\tilde g\bigr) \= \oh(t^a)$ as
$t\uparrow\infty$ uniform for $\tilde \gm$ in compact sets in
$\tilde G$ for some $a\in \RR$.
\end{enumerate}

If we replace condition~\eqref{eq:MFd} by
$f \bigl(\tilde\am(t)\tilde g\bigr) \= \oh(t^{-a})$ as
$t\uparrow\infty$, uniform for $\tilde \gm$ in compact sets in
$\tilde G$ for all $a\in \RR$, then $f$ is a \emph{Maass cusp
form}.\il{McftG}{Maass cusp form}
\end{defn}

\rmrk{Weight shifting operators}The operator $\Dt$ preserves the weight
of functions on $\tG$, and corresponds to the operator $\Dt_k$
on~$\uhp$.

The operator $X_+$ transforms functions of weight $k$ into functions of
weight $k+2$. It corresponds to the operator
$X_{+,k}=2i y \,\partial_z + \frac k2$ on $\uhp$, which we can see with the
following computation:
\begin{align*}
X_+ (\Ps_k u)(z,\th)& \= e^{2i \th}\Bigl( 2i y \,\partial_z-\frac
i2 \,\partial_\th\Bigr) \Bigr( u(z) e^{ik\th}\Bigr)
\displaybreak[0]\\
&\= 2i y \, \frac{\partial u}{\partial z}(z) e^{i(k+2)\th} + u(z) \frac
k2 e^{i(k+2)\th}
\displaybreak[0]\\
&\= \Bigl( 2iy \frac{\partial u}{\partial_z}(z) + \frac k2 \, u(z)
\Bigr)\, e^{i(k+2)\th}\displaybreak[0]\\
&\= \Bigl( \Ps_{k+2} \Bigl( 2iy \partial_z + \frac k2 \Bigr)
u \Bigr)(z,\th) \,.
\end{align*}
\il{Xpmk}{$X_{+,k}, \; X_{-,k}$}We leave to the reader the analogous
computation for $X_-$ to see that $X_-$ corresponds to
$X_{-,k}=-2i y\,\partial_{\bar z} - \frac k2 $. The operators $X_{+,k}$
and $-X_{-,k}$ correspond to the operators in \cite[(3.1), (3.2)]{Roelcke66}.
Working with Maass forms on the group we have the following results.
\begin{lem}The weight shifting operators satisfy
\be X_\pm \colon \A_k^0(s,\ch_k \rho ) \rightarrow \A_{k\pm 2}^0(s,\ch_k \rho
)\,.\ee
\end{lem}
We note that the weight changes, but that the representation
$\ch_k \rho$ of $\tGm$ stays the same.
\begin{proof}We apply the operator $  X_{\pm,k}$ corresponding to
$X_\pm$ to the Fourier expansion in~\eqref{Fe}. This leads to an
expression in the Whittaker function and its derivative. The contiguous
relations for Whittaker functions allow to show that we get a multiple of
the expression for the Fourier expansion in weight $k\pm 2$. The
absolute convergence of the Fourier expansion gives an estimate of the
growth of the coefficients $c_l(n)$ that is strong enough to show
convergence of the sum of derivatives, and it leads to exponential decay
of the resulting sum.
\end{proof}
\begin{lem}
The operators $X_\pm X_\mp$ act in $\A^0_k(s,\ch_k\rho)$ as
multiplication by the factor $(s\mp k/2)(s-1\pm k/2)$.
\end{lem}
\begin{proof}The operator $Z$ acts in $\A^0_{k'}(s,\ch_k\rho)$ as
multiplication by $k'$. The operator $\Dt$ acts as multiplication by
$s(1-s)$. A computation based on \eqref{DtXpm} gives the actions of
$X_\pm X_\mp$.
\end{proof}
If $s\not\equiv \frac k2\bmod 1$ and $s\not\equiv -\frac k2\bmod 1$ the
product $(s\mp k/2)(s-1\pm k/2)$ is non-zero, and hence
\be \label{isowsh}X_{\pm,k}\colon \A^0_k(s,\rho v_k) \rightarrow \A^0_{k\pm
2}(s,\rho v_k)\ee
is a bijection.
(Note that the multiplier system $\rho v_k$ is preserved under
$X_{\pm , k}$.)
This shows that if we know one space $\A^0_k(s,\rho v_k)$ with
$s\not\equiv \pm \frac k2\bmod1$, then we know $\A^0_{k'}(s,\rho v_k)$
for all $k'\equiv k\bmod 2$.

\begin{prop}\label{prop-sppMcf} Let $s\in \CC$,
$s\not\equiv \pm \frac k2\bmod1$. Define $\nu\in [0,2)$ by
$\nu\equiv k\bmod 2$. If the space $\A^0_k(s,\rho v_k)$ is non-zero,
then
\badl{sdom} & s\in \frac12 + \bigl(i \RR\setminus\{0\}\bigr)\,,\\
\text{or }\quad &\begin{cases} \frac \nu2 < s < 1-\frac \nu2 &\text{ if
}\nu\in [0,1)\,,\\
1-\frac \nu2 < s < \frac \nu 2&\text{ if }\nu \in [1,2)\,.
\end{cases}
\eadl
\end{prop}
This follows from \cite[Satz 3.1]{Roelcke66}, which implies that
\be -s(2-s) \;\leq\; - \frac {k'}2\Bigl( 1+ \frac {k'}2\Bigr)\ee
for all $k'\equiv  k\bmod 2$. For $k'=\nu$ this equality is stronger
than for all other $k'\equiv \nu\bmod 2$.

%% file: rwt7-ps.tex


\section{Principal series representation}\label{sect-prs}In
\cite{BLZ15}, the representation of $G$ underlying the modules used in
the cohomology of $\Gm$ is the discrete series representation. Handling
arbitrary real weights requires some care.

\rmrk{Operators on the real projective line}The action of $G$ on $\uhp$
by fractional linear transformations extends to an action on the
projective line $\proj\RR$, which is the boundary of~$\uhp$.

By an open interval in $\proj\RR$ we mean an open connected subset
$I\subset \proj\RR$ that is not equal to $\proj\RR$ and has more than
one point. The set $\RR$ is an open interval in $\proj\RR$, and
intervals $(\al,\bt)$ in $\RR$ with $\al<\bt\in \RR$ are intervals in
$\proj\RR$ as well. If $\al>\bt$, then we have the open interval
\il{ci}{$(\cdot,\cdot)_c$}$(\al,\bt)_c 
\= (\al,\infty) \cup\{\infty\}\cup(-\infty,\bt)$.

Let $s\in \CC$ and $k\in \RR$, and let $g=\rmatc abcd\in G$. For
functions $\ph$ on an open subset $I\subset \proj\RR$ we define
$\ph|^\prs_{s,k} g$ on $g^{-1} I$ by\ir{prs-trf}{|^\prs_{s,k}g}
\badl{prs-trf} \bigl(\ph|^\prs_{s,k} g\bigr)(t)&\= (a-ic)^{-s+k/2}\,
(a+i c)^{-s-k/2}\, \Bigl( \frac{t-i}{t-g^{-1}i}\Bigr)^{s-k/2}\\
&\qquad\hbox{} \cdot
\Bigl( \frac{t+i}{t-g^{-1}(-i)}\Bigr)^{s+k/2}\,
\ph(gt)\,,\eadl
where we use $-\pi \leq \arg(a-ic) < \pi$ and
$-\pi < \arg(a+ic) \leq \pi$. (In this way $\ph|^\prs_{s,k}\km(\th)$ is
right-continuous in $\th=-\pi$, like we have in~\eqref{slash-k}.)

The function $t\mapsto \ph(gt)$ is again a function on $\proj\RR$.
The function
$t\mapsto \Bigl( \frac{t-i}{t-g^{-1} i}\Bigr)^{s-k/2}$
is well-defined on $\proj\RR$. In fact it determines a holomorphic
function on $\proj\CC$ minus a path from $i$ to $g^{-1} i$ in $\uhp$.
Similarly, the other factor is holomorphic on $\proj\CC$ outside a path
in the lower half-plane. These two factors are real-analytic on
$\proj\RR$. So $|^\prs_{s,k} g$ sends real-analytic functions on $I$ to
real-analytic functions. It also sends $C^p$-functions to
$C^p$-functions for $p=0,1,\ldots, \infty$. Tensoring $\CC$ with
$X_\rho$ we get operators\ir{prsrvk}{|^\prs_{\rho v_k,s,k} }
\be \label{prsrvk}\bigl( \ph|^\prs_{\rho v_k,s,k} \gm\bigr)
(t) \= \rho(\gm)^{-1}\ v_k(\gm)^{-1}\,
(\ph|^\prs_{s,k}\gm)(t)\ee
for each $\gm\in \Gm$, and have
\be \ph |^\prs_{\rho v_k,s,k}(\gm_1\gm_2) \= \Bigl( \ph|^\prs_{\rho
v_k,s,k}\gm_1\Bigr)
|^\prs_{\rho v_k,s,k}\gm_2\ee
for $\gm_1,\gm_2\in \Gm$.

In this way we arrive at the $\Gm$-equivariant sheaf
\il{Vomsh}{$\V\om{\rho v_k,s,k}$}$\V\om{\rho v_k,s,k}$ of analytic
functions on $\proj\RR$ with values in $X_\rho$, and on larger
equivariant sheaves $\V p{\rho v_k,s,k}$ of $p$ times continuously
differentiable functions, with $p=0,1,2,\ldots,\infty$. The action of
$\gm \in \Gm$ is given by
\be |^\prs_{\rho v_k,s,k} \gm \colon  \V \om{\rho v_k,s,k}(I)
\rightarrow \V \om{\rho v_k,s,k}(\gm^{-1}I)\,.\ee
The space of global sections $\V\om{\rho v_k,s,k}(\proj\RR)$ is
$\Gm$-invariant for this action. This is the \il{psr}{principal series
representation}\emph{principal series representation} twisted by
$\rho v_k$. Similar remarks hold for the larger sheaves
$\V p {\rho v_k,s,k}$ of $p$ times continuously differentiable
functions.

We note that $|_{s,k}^\prs(-I_2)$ is multiplication by $e^{-\pi i k}$,
independent of $s$. From $v_k(-I_2)= e^{-\pi i k} $ we conclude that
$|^\prs_{\rho v_k,s,k}(-I_2)$ is just application of the operator
$\rho(-I_2)^{-1} = \rho(-I_2)$. (We use that $\rho$ is a representation
of $\Gm$, and $-I_2\in \Gm$ is its own inverse.)
For any set $I\subset \proj\RR$ the $1$-eigenspace of
$|^\prs_{\rho v_k,s,k}(-I_2)$ in $\V \om{\rho v_k,s,k}(I)$ is
independent of $s$ and $k$.

\subsection{Period functions}\label{sect-pf}
The period functions that we want to relate to Maass cusp forms in
$\A^0_k(s,\rho v_k)$ form a subspace of
$\V \om{\rho v_k,s,k}(0,\infty)$ with several additional properties.

\rmrk{Action of \intitle{-I_2}}Like for Maass cusps forms, we want
period functions to have values in the $1$-eigenspace of $\rho(-I_2)$.
This implies that
\be \bigl( f|^\prs_{\rho v_k,s,k} S\bigr)
|^\prs_{\rho v_k,s,k} S \= f\ee
for a period function~$f$, for $S=\rmatr0{-1}10$.

\rmrk{Three term relation} We want the period functions to satisfy the
three term relation\ir{3te}{three term relation}
\be\label{3te} f \= f|^\prs_{\rho v_k,s,k} T
+ f|^\prs_{\rho v_k,s,k}T'\,,\ee
where \il{Ta}{$T'=\rmatc1011$}$T'=\rmatc1011 = ST^{-1}S$. This relation
goes back to the three term relation in Lewis's paper \cite{L97}.

For $f\in \V \om{\rho v_k,s,k}(0,\infty)$, the terms on the right hand
side of \eqref{3te} are elements of $\V\om{\rho v_k,s,k}(-1,\infty)$
and $\V\om{\rho v_k,s,k}\bigl((0,-1)_c\bigr)$. In the relation these
two terms are understood to be restricted to $(0,\infty)$.

\rmrk{Continuous extension} The period functions attached to Maass cusp
forms by Lewis and Zagier \cite{LZ01} determine real-analytic functions
on $(0,\infty)$ that satisfy the three term relation~\eqref{3te} with
$\rho v_k=1$, $\re s\in (0,1)$, and $k=0$, and have estimates at the
boundary points of $(0,\infty)$. By \cite[Theorem B and Proposition
14.2]{BLZ15} these functions have, in the projective model used in this
paper, a smooth extension $f$ to $\proj\RR$ satisfying
$f|_{1,s,0} ^\prs S = - f$. In particular, $\lim_{x\downarrow 0}f(x)$
and $\lim_{x\uparrow0} f(x)$ exist and are equal, and analogously for
the one-sided limits at~$\infty$. \smallskip

Here we require that the limits
\be \label{lim-pf} a_\infty(f) \= \lim_{t\uparrow\infty} f(t)\quad\text{
and }\quad a_0(f) \= \lim_{t\downarrow 0}f(t)
\ee
exist, and satisfy
\be \label{lim-rel} a_0(f) \= - \rho(S) a_\infty(f).\ee
We note that $f|^\prs_{\rho v_k,s,k}S $ is defined on $(\infty,0)_c$. We
have for $t<0$:
\begin{align*}
\bigl(f|^\prs_{\rho v_k,s,k}& S\bigr)(t)
\stackrel{\eqref{prs-trf}}= \rho(S)^{-1} v_k(S)^{-1} (-i)^{k/2-s}\,
i^{-s-k/2} \, f(-1/t)
\displaybreak[0]\\
& \stackrel{\eqref{vk}}= \rho(S)^{-1} (-i)^{-k} \, e^{\pi i
(-k)}\, f(-1/t)
\displaybreak[0]\\
& \stackrel{t\uparrow0}{\longrightarrow} \rho(S)a_\infty(f )
\,.
\end{align*}
So \eqref{lim-rel} ensures that the real-analytic function
$-f|^\prs_{\rho v_k,s,k} S$ on $(\infty,0)_c$ has the same limit for
$t\uparrow 0$ as the real-analytic function $f$ on $(0,\infty)$ for
$t\downarrow 0$. This implies that
\be\label{pf-extcoc} t\mapsto\begin{cases} f(t) &\text{ for }t>0\,,\\
-\bigl(f|^\prs_{\rho v_k,s,k} S\bigr)(t)
&\text{ for }t<0\,,
\end{cases}
\ee
extends as a continuous function on $\RR$.

Applying $|^\prs_{\rho v_k,s,k}S$ gives a similar continuous extension
across $\infty$. One can check that the three term equation on
$(0,\infty)$ implies that it also holds on $(\infty,-1)_c$ and on
$(-1,0)$.

\rmrk{Holomorphic extension}A real-analytic function on $(0,\infty)$ is
locally on $(0,\infty)$ given by power series, and hence extends
holomorphically to a complex neighborhood of $(0,\infty)$. For period
functions we require that the extension is possible to a wedge of the
form\ir{Wdt}{W_\dt}
\be \label{Wdt}W_\dt\= \bigl\{ t \in \CC\setminus \{0\}\;:\; |\arg t| <
\dt \bigr\}\,.\ee
We note that the extensions of the three functions in the three term
relation~\eqref{3te} may extend to different domains. The relation
extends only to a connected neighborhood of~$(0,\infty)$.

\begin{defn}\label{def-FEom}The space
\il{FEom}{$\FE^\om_{\rho v_k,s,k}$}$\FE^\om_{\rho v_k,s,k}$ of
\il{pdfct}{period function}\emph{period functions} is the linear space
of elements $f\in \V \om{ \rho v_k, s,k}(0,\infty)$ that
\begin{enumerate}[label=$\mathrm{(\alph*)}$, ref=$\mathrm{\alph*}$]
\item\label{eq:FEa} $f$ has values in the $1$-eigenspace of
$|^\prs_{\rho v_k,s,k}(-I_2)$.
\item\label{eq:FEb} $f$ satisfies the three-term relation~\eqref{3te}.

\item\label{eq:FEc} $f$ has limits at $0$ and $\infty$ as indicated in
\eqref{lim-pf} and~\eqref{lim-rel}.
\item\label{eq:FEd} $f$ has a holomorphic extension to a wedge $W_\dt$ for some
$\dt\in (0,\pi/2)$.
\end{enumerate}
\end{defn}

The period functions in \cite[p.~85]{BLZ15} are characterized by a
boundary condition, equivalent to $\oh(1)$ at $0$ and $\infty$ in the
projective model used here. It is equivalent to the existence of
asymptotic expansions at $0$ and $\infty$. The existence of limits in
part~\eqref{eq:FEc} is easier to handle, and leads to the same space of period
functions.\smallskip

Our aim is to establish a relation between Maass cusp forms in
$\A^0_k(s,\rho v_k)$ and period functions in
$\FE^\om_{\rho v_k,s,k}$.\smallskip

The properties required in parts~\eqref{eq:FEb} and~\eqref{eq:FEd} in Definition~\ref{def-FEom}
allow us to apply the bootstrap method method in \cite[Chap III.4,
p.~240]{LZ01}.
\begin{prop}\label{prop-extFE}Each period function
$f\in \FE^\om_{\rho v_k,s,k}$ has a holomorphic extension to
\il{C'}{$\CC'$}$\CC'=\CC\setminus (-\infty,0]$.
\end{prop}
\begin{proof}Condition~\eqref{eq:FEd} implies that $f$ is holomorphic on a wedge
\[
W_\dt= \bigl\{  z\in \CC'\;:\; |\arg(z)|<\dt\bigr\}
\]
for some small
$\dt>0$. For each $z\in \CC'$ we can use the three term
relation~\eqref{3te} to express $f(z)$ as a finite sum of translates
$f|_{\rho v_k,s,k} \gm(z)$ with $\gm$ in the semigroup generated by $T$
and $T'$ such that $f(\gm z)$ is in $W_\dt$.
\end{proof}
\smallskip

In the following result we use the orthonormal eigenbasis $e_l$ of
$X_\rho$, for the scalar product $(\cdot,\cdot)_\rho$, and the
parameters $\k_l\in [0,1)$ introduced in~\S\ref{sect-Fe}.
\begin{lem}\label{lem-f1c}Let $f\in \FE^\om_{\rho v_k,s,k}$. If
$\k_l =0$, then
$\Bigl( \rho(T')^{-1} v_k(T')^{-1} f(1), e_l\Bigr)_\rho=0$.
\end{lem}
\begin{proof}
We take the limit as $t\uparrow\infty$ of the three term equation, and
project it to the line in $X_\rho$ spanned by~$e_l$:
\begin{align*} a_\infty(f) &\= \rho(T)^{-1} v_k(T)^{-1} \, 1^{s-k/2}\,
1^{-s-k/2} \, a_\infty(f)\\
&\qquad\hbox{}
+ \rho(T')^{-1}v_k(T')^{-1}\, (1-i)^{-s+k/2}\,(1+i)^{-s-k/2} \,
1^{s-k/2}\, 1^{s+k/2}\, f(1)
\,,
\displaybreak[0]\\
\bigl(a_\infty(f),e_l\bigr)_\rho &\= e^{-2\pi i \k_l} \,
\bigl(a_\infty(f),e_l\bigr)_\rho
+ 2^{-s}\, i^{-k/2} \,\bigl( \rho(T')^{-1} v_k(T')^{-1} f(1),
e_l\bigr)_\rho\,.
\end{align*}
For $\k_l=0$ this gives the assertion in the lemma.\end{proof} We do not
know how $\rho(T')$ acts on the eigenbasis for $\rho(T)$, and further
simplification seems hard.

%% file: rwt7-pi.tex


\section{Period functions}\label{sect-pi} In this section we show that
we can associate a period function to each Maass cusp form. We follow
the approach in \cite{BLZ15} for weight $0$, and adapt it to arbitrary
real weights.

\subsection{Poisson kernel}\il{Pk}{Poisson kernel}The function
$R(t;z)^s$ in \cite[ \S2.2]{BLZ15} can be generalized as the
scalar-valued function on
$\proj\RR \times \uhp$\ir{Poisk}{R_{s,k}(t,z)}
\be \label{Poisk}R_{s,k}(t,z) \= y^s \,
\Bigl(\frac{t-i}{t-z}\Bigr)^{s-k/2}\, \Bigl( \frac{t+i}{t-\bar
z}\Bigr)^{s+k/2}\,.\ee
As a function of $t$ it is real-analytic on~$\proj\RR$, and as a function of
$z$ it is real-analytic on~$\uhp$. It satisfies
\be \label{efR}\Dt_{-k} R_{s,k}(t,\cdot) \= s(1-s)\,
R_{s,k}(t,\cdot)\,.\ee
For $g\in G$
\be\label{diagactR} \bigl(R_{s,k}|^\prs_{s,k} g \bigr)|_{-k}g \=
\bigl( R_{s,k}|_{-k}g\bigr)
|^\prs_{s,k} g = R_{s,k}\,.\ee
To see this we note that the operator $|^\prs_{s,k}g$ acts on the
variable $t$, and the operator $|_{-k}g$ on the variable $z$. Hence
these operators commute. We carry out the computation for $\rmatc abcd$
near the unit element of $G$. Then the handling of powers of complex
quantities is not hard. The relation extends by analyticity. The action
of $\rmatc{-1}00{-1}$ is multiplication by $e^{-\pi i k}$ for
$|^\prs_{s,k}$ and multiplication by $e^{\pi i k }$ for $|_{-k}$; so
there is no difficulty on the region of discontinuity.

We write $j_{k}(g , z) \= (c_g z+d_g)^{-k}$ and $J^\prs_{s,k}(g,t)$ for
the factor before $\ph(gt)$ in~\eqref{prs-trf}. With
\eqref{diagactR} we see for $\gm\in \Gm$:
\begin{align}\nonumber
v_k(\gm)^{-1} & J^\prs_{s,k} (\gm, t)\, v_{-k}(\gm)^{-1}\, j_{-k}(\gm
,z) \, R_{s,k} (\gm t, \gm z) \= R_{s,k}(t,z)\,,\displaybreak[0]\\
\nonumber
v_k(\gm)^{-1} & J^\prs_{s,k} (\gm, t)\, \bigl( R_{s,k}(\gm
t,\cdot)|_{v_{-k},-k}\gm\bigr)(z)
\= R_{s,k}(t,z)\,,
\displaybreak[0]\\
\nonumber
v_k(\gm)^{-1} & J^\prs_{s,k} (\gm, t) \, R_{s,k}(\gm t,z) \=
R_{s,k}(t,\cdot) |_{v_{-k},-k}\gm^{-1}(z)\,,\displaybreak[0]\\
\label{Rd1}
R_{s,k}(\cdot,z)&|^\prs_{v_k,s,k} \gm \; (t) \=R_{s,k}(t,\cdot)
|_{v_{-k},-k}\gm^{-1}\, (z)\,.
\end{align}

\subsection{Green's form}\il{Gsf}{Green's form}The generalization of the
differential form $[u,v]$ in \cite[(1.9)]{BLZ15} is\ir{Gf}{[u_1,u_2]_k}
\badl{Gf} \bigl[ u_1,u_2\bigr]_k &\= \Bigl( \frac{\partial
u_1}{\partial z}\, u_2+\frac k{4iy} u_1 u_2 \Bigr)\, dz
+\Bigl( u_1\,\frac{\partial
u_2}{\partial_{\bar z}}
- \frac k{4iy} u_1u_2\Bigr)\, d\bar z\\
&\=-2i \Bigl( ( X_{+,k} u_1) v_2\, dz + u_1 (X_{-,-k} v_2) \, d\bar
z\Bigr)
\,,\eadl
for $u_1, u_2\in C^\infty(\uhp)$ (or for smooth functions on an open
subset of $\uhp$).

Some properties are
\begin{align}
d\bigl( u_1u_2\bigr) &\= \bigl[ u_1,u_2\bigr]_k + \bigl[u_2,u_1]_{-k}\,,
\displaybreak[0]
\\
\label{dhk}
d\bigl[ u_1,u_2]_k &\= \bigl(u_1\, \Dt_{-k} u_2 - u_2\,\Dt_k u_1
\bigr)\, \frac{dz\,d\bar z}{-4y^2}\,.
\end{align}
If $u_1 $ is an eigenvector of $\Dt_k$ and $u_2$ is an eigenvector of
$\Dt_{-k}$ with the same eigenvalue, then $\bigl[u_1,u_2\bigr]_k$ is a
closed 1-form. For all $g\in G$ we have
\be \label{Gf-inv}
\bigl[ u_1|_k g, u_2|_{-k} g \bigr]_k \= \bigl[ u_1,u_2\bigr]_k \circ
g\,,\ee
where $\circ g$ means the substitution $z\mapsto g z$. We can write it
as $\bigl[ u_1,u_2\bigr]_k |_0 g$.

These properties go through if one of $u_1$ and $u_2$ is vector-valued.
Then the products involved in the formulas make sense, and the
relations hold for each component of the vector-valued function.

Let $u_1$ be vector-valued with values in~$X_\rho$, and let $u_2$ be scalar-valued. Then we have for
$\gm\in \Gm$
\be\label{khtGm}
\bigl[ u_1|_{ \rho v_k,k}\gm, u_2|_{v_{-k},-k}\gm \bigr]_k \=
\rho(\gm)^{-1} \, \bigl[ u_1,u_2\bigr]_k \circ \gm\,.\ee

\rmrk{Disk model} The upper half-plane is isomorphic as a complex
variety with the unit disk by the map $z\mapsto w=\frac{z-i}{z+i}$ with
inverse $w\mapsto z=i\frac{1+w}{1-w}$. In the proof of
Proposition~\ref{prop-psC} it will be convenient to use the formulation
of the Green's form \ on the unit disk:\il{Gfdm}{Green's form, disk
model}
\badl{Gfdm} \bigl[a,b\bigr]_k &\= \Bigl(\frac{\partial a}{\partial w}b+
\frac{k(1- \bar w)}{2(1-w)(1-|w|^2)}ab \Bigr)\, dw
\\
&\qquad\hbox{}
+\Bigl( a \frac{\partial b}{\partial\bar w} + \frac{k(1-w)}{2(1-\bar
w)(1-|w|^2)} ab \Bigr)\, d\bar w\,.
\eadl

\subsection{Differential form}For smooth functions
$u\colon \uhp\rightarrow X_\rho$ we have a differential form of degree 1 with
values in the real-analytic functions
$\proj\RR \rightarrow X_\rho$:\ir{etaks}{\eta_{s,k}(u)}
\be\label{etaks} \eta_{s,k}(u) \= \bigl[ u, R_{s,k}\bigr]_k\,.\ee
This is a differential form on $\uhp$ with values in the functions on
$\proj\RR$. If we want to stress the role of the variables, we write
$\eta_{s,k}(u;z,t) = \bigl[ u(z), R_{s,k}(t,z) \bigr]_k$. If
$\Dt_k u = s(1-s)\, u$, then $\eta_{s,k}(u) $ is a closed form.
(Use \eqref{dhk} and \eqref{efR}.)

For Maass cusp forms $u\in \A_k^0(s,\rho v_k)$,
we have for $z_1,z_2\in \uhp$ and for $\gm\in \Gm$
\be \label{intetatr}\int_{\gm^{-1}
z_1}^{\gm^{-1} z_2} \eta_{s,k}(u)
\= \int_{z_1}^{z_2}\eta_{s,k}(u) |^\prs_{\rho v_k ,s,k}\gm \ee
as can be checked as follows:
\begin{align*}
\int_{z_1}^{z_2}&\eta_{s,k}(u) |^\prs_{\rho v_k,s,k}\gm = \rho(\gm)^{-1}\, v_k(\gm)^{-1}\,
\int_{z_1}^{z_2} \Bigl[ u, R_{s,k} \bigm|^\prs_{s,k}\gm\Bigr]_k
&&\text{by \eqref{prsrvk}}
\displaybreak[0]\\
&\= \rho(\gm)^{-1}  \, v_k(\gm)^{-1}
\int_{z_1}^{z_2} \Bigl[ u,R_{s,k}\bigm|_{-k}
\gm^{-1} \Bigr]_k&&\text{by \eqref{Rd1}}
\displaybreak[0]\\
&\= \rho(\gm)^{-1}  \, v_k(\gm)^{-1}
\int_{z_1}^{z_2} \Bigl[ u|_{\rho
v_k,k}\gm^{-1},R_{s,k}  \bigm|_{-k } \gm^{-1} \Bigr]_k
&& \text{since }u\in\A^0_k(s,\rho v_k)
\displaybreak[0]\\
&\= \int_{z_1}^{z_2} \Bigl[ u\bigm|_{ k}\gm^{-1},
R_{s,k} \bigm|_{-k} \gm^{-1} \Bigr]_k
&& \text{by \eqref{|rhockk} and}\\
&&&(\rho v_k)(\gm^{-1})^{-1} \= (\rho v_k)(\gm)
\displaybreak[0]\\
&\= \int_{z_1}^{z_2} \bigl[u, R_{s,k} \bigr]_k\circ \gm^{-1}&&\text{by
\eqref{khtGm}}
\displaybreak[0]\\
&\= \int_{\gm^{-1}z_1}^{\gm^{-1}z_2} \eta_{s,k}(u)\,.
\end{align*}

\subsection{Cocycles attached to Maass cusp forms}\label{sect-caMcf} For
  $u\in \A^0_k(s,\rho v_k)$ we put
\be c^u(z_1,z_2) \= \int_{z_1}^{z_2} \eta_{s,k}(u)\qquad\text{for
}z_1,z_2\in \uhp\,.\ee
This function on $\uhp\times\uhp$ has values in the $\Gm$-module
$\V \om{\rho v_k,s,k}(\proj\RR)$, and it does not depend on the choice
of the path from $z_1$ to $z_2$. It satisfies the homogeneous cocycle
relations
\badl{coc-prop} c^u(z_1,z_2) + c^u(z_2,z_3)
&\= c^u(z_1,z_3)&\quad &\text{for }z_1,z_2,z_3\in \uhp\,,\\
c^u(\gm^{-1}z_1,\gm^{-1}z_2) &\= c^u(z_1,z_2)|^\prs_{\rho v_k,s,k}\gm&&
\text{for }z_1,z_2\in \uhp,\;\gm\in \Gm\,.
\eadl
So $c^u$ is a cocycle in
$Z^1\bigl( \Gm;\V\om{\rho v_k,s,k}(\proj\RR)\bigr)$.
(See the discussion in \cite[\S6.1]{BLZ15}.) This definition does not
need a growth condition, and it works for more automorphic forms than
cusp forms.
\rmrk{Parabolic cocycles}If $u$ is a cusp form, then it has exponential
decay as $y\uparrow \infty$, and the same holds for its derivatives.
This implies that $\int_{z_1}^\infty \eta_{s,k}(u)$ converges 
absolutely  and does not depend on the path from
$z_1\in \uhp$ to $\infty$.

The cusps of $\Gm=\SL_2(\ZZ)$ form the set $\proj\QQ \subset\proj\RR$.
Each cusp $\xi$ is of the form $\xi=\gm\infty$ for some (non-unique)
$\gm\in \Gm$. The invariance of $u$ under $|_{v_k,k}\Gm$ implies that
$\eta_{s,k}(u)$ has fast decay when approaching any cusp of $\Gm$. So
we can form integrals $\int_{z_1}^\xi \eta_{s,k}(u)$ for any cusp
$\xi$, and also integrals between two cusps. In this way we get
\ir{cparb}{c^u_\pb}
\be\label{cparb} c^u_\pb(\xi_1,\xi_2) \= \int_{\xi_1}^{\xi_2}
\eta_{s,k}(u)\qquad\text{for }\xi_1,\xi_2\in \proj\QQ\,.\ee
The function $c^u_\pb$ on $\proj\QQ\times\proj\QQ$ has properties
analogous to~\eqref{coc-prop}. It requires some work to determine its
regularity properties. Working out $\bigl[ u, R_{s,k}(t,\cdot)\bigr]_k$
we see that $c^u_\pb(\xi_1,\xi_2)$ is real-analytic at all points of
$\proj\RR\setminus\{\xi_1,\xi_2\}$.

The behavior at $t=\xi_1$ and $t=\xi_2$ has to be considered. By the
transformation behavior under $\Gm$ we can reduce the consideration to
integrals $\int_z^\infty \eta_{s,k}(u)$. Proceeding in the same way as
in \cite[Proposition 9.7]{BLZ15} we can show that it is $C^\infty$ in a
neighborhood of $\infty$ in~$\proj\RR$. Here we are content to have
continuity.

In a notation analogous to the notations in \cite{BLZ15}, we define the
$\Gm$-module \il{Vom0}{$\V{\om,0}{\rho v_k,s,k}(\proj\RR)$
}$\V{\om^0,0}{\rho v_k,s,k}(\proj\RR)$ as the space of functions in the
space $\V0{\rho v_k,s,k}(\proj\RR)$ of continuous functions that
restrict to an element of
$\V\om{\rho v_k,s,k}\bigl( \proj\RR \setminus E\bigr)$ for a finite set
$E \subset \proj\QQ$. For $c^u_\pb(\xi_1,\xi_2)$ the set $E$ can be
taken as $\{\xi_1,\xi_2\}$.

The index \emph{par} in
\be c^u_\pb \in Z_\pb^1\bigl( \Gm;\V{\om^0,0}{\rho
v_k,s,k}(\proj\RR)\bigr)\ee
indicates \emph{parabolic}. For each $\xi\in \proj\QQ$ there is an
infinite subgroup of $\Gm$ fixing $\xi$. For $\xi=\infty$ this is the
subgroup generated by $T$ and $-I_2$. This has the consequence that
cocycles on $\proj\QQ\times\proj\QQ$ do not compute the usual
cohomology groups $H^1(\Gm;\cdot)$, but the \emph{parabolic cohomology
groups} $H^1_\pb
(\Gm;\cdot)$.

\subsection{The period function of a Maass cusp form}
\begin{prop}\label{prop-P}Let $u\in \A^0_k(s,\rho v_k)$. We put
\il{Pu}{$P(u)$}$P(u)= c^u_\pb(0,\infty)$.
\begin{enumerate}[label=$\mathrm{(\alph*)}$, ref=$\mathrm{\alph*}$]
\item\label{eq:Pua} $P(u)$ has values in the $1$-eigenspace of $|^\prs_{\rho
v_k,s,k}(-I_2)$.
\item\label{eq:Pub} $P(u) \= - P(u)|^\prs_{\rho v_k,s,k}S \= P(u) |^\prs_{\rho
v_k,s,k}T + P(u) |^\prs_{\rho v_k,s,k} T'$.
\item\label{eq:Puc} $P(u)$ is real-analytic on $(\infty,0)_c \cup
(0,\infty)$, with a continuous extension across $0$ and $\infty$.
\item\label{eq:Pud} $P(u)$ has a holomorphic extension to $\CC\setminus i\RR$.
\item\label{eq:Pue} If $\xi_1,\xi_2\in \proj\QQ$, then there is a finite number of
elements $\gm_j\in \Gm$ such that
\be c^u_\pb(\xi_1,\xi_2) \= \sum_j \, P(u) |^\prs_{\rho v_k,s,k}
\gm_j\,.\ee
\end{enumerate}
\end{prop}
\begin{proof}Statements~\eqref{eq:Pua} and \eqref{eq:Puc} are specializations of properties
already observed for integrals $c^u_\pb(\xi_1,\xi_2)$ with general
$\xi_1$ and $\xi_1$ in $\proj\QQ$.

We can take the path of integration from $0$ to $\infty$ for $P(u)$ as
the positive imaginary axis. Then we obtain part~\eqref{eq:Pud}.

By \eqref{intetatr}
\[ \int_0^\infty \eta_{s,k}(u) |^\prs_{\rho v_k,s,k} S \= \int_\infty^0
\eta_{s,k}(u) \= - \int_0^\infty \eta_{s,k}(u)\,.\]
This gives the first relation in~\eqref{eq:Pub}. For the other relation we use
\[ \int_0^\infty \eta_{s,k}(u) |^\prs_{\rho v_k,s,k}\bigl( T+ T'\bigr)
\= \int_{-1}^\infty \eta_{s,k}(u) + \int_0^{-1}\eta_{s,k}(u) \=
\int_0^\infty \eta_{s,k}(u)\,.\]

We use the well-known Farey tesselation (sketched in
Figure~\ref{fig-Farey}). The endpoints of the edges run through
$\proj\QQ$. Each edge is the translate $\gm e_{0,\infty}$ for some
$\gm\in \Gm$, where $e_{0,\infty}$ denotes the path from $0$ to
$\infty$. We note that $e_{\infty,0}= S^{-1} e_{0,\infty}$. Each vertex
is connected to $\infty$ by a path along finitely many edges of the
tesselation.
(Use a Farey sequence with bounded denominators to see this.)
In this way we get for each pair $(\xi_1,\xi_2) \in \proj\QQ$ a finite
path $\sum_{j} \gm_j\, e_{0,\infty}$ from $\xi_1$ to $\xi_2$. We use
this to obtain part~\eqref{eq:Pue}.
\end{proof}

\begin{figure}
\begin{center}
\includegraphics[width=.8\textwidth]{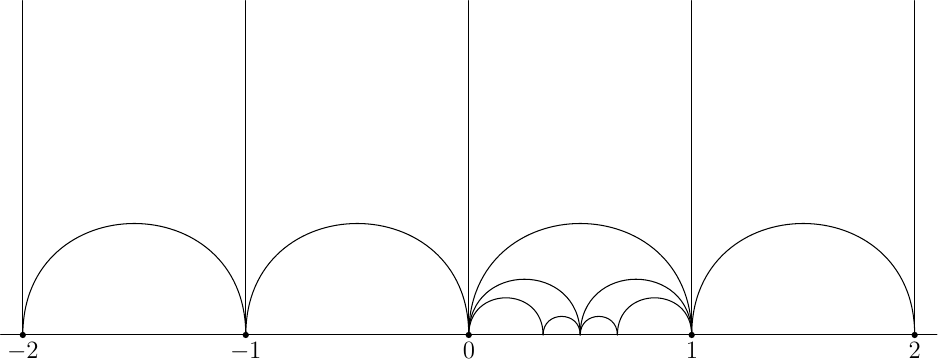}
\end{center}
\caption[]{The Farey tesselation. \\
This tesselation of $\uhp$ consists of all $\Gm$ translates of the
hyperbolic triangle with corner points $0$, $\infty$, and $1$. }
\label{fig-Farey}
\end{figure}
\begin{prop}[Period function of Maass cusp form]\label{prop-pfMcf}\mbox{ }
\begin{enumerate}[label=$\mathrm{(\roman*)}$, ref=$\mathrm{\roman*}$]
\item For $P(u) $ as in Proposition~\ref{prop-P}, its restriction to
$(0,\infty)$ defines $\pf(u)$ by\ir{pf}{\pf}
\be\label{pf} \pf(u) \coloneqq P(u)|_{(0,\infty)}\ee
This determines a linear map
$\pf\colon  \A^0_k(s,\rho v_k) \rightarrow \FE^\om_{\rho v_k,s,k}$. We call
$\pf(u)$ the \il{pfau}{period function associated to Maass cusp
form}\emph{period function} associated to~$u$.
\item For each $f\in \FE^\om_{\rho v_k,s,k}$ there is a unique
element $p\in \V{\om^0,0}{\rho v_k,s,k}(\proj\RR)$ for which properties~\eqref{eq:Pua}, \eqref{eq:Pub} and~\eqref{eq:Puc} in Proposition~\ref{prop-P} hold, with restriction
$p|_{(0,\infty)} $ equal to $f$.
\end{enumerate}
\end{prop}
\begin{proof}Definition~\ref{def-FEom} of the space
$\FE^\om_{\rho v_k,s,k}$ of period functions has been arranged in such
a way that the restriction of $P(u)$ to $(0,\infty)$ is a period
function.

Conversely, any period function $f$ has a unique extension
$p\in \V\om{s,k}\bigl( \RR \setminus \{0\}\bigr)$ satisfying
$p|_{\rho v_k,s,k}S = -p$. The limits in condition~\eqref{eq:FEc} in
Definition~\ref{def-FEom} imply that $p$ extends as a continuous
function on $\proj\RR$, hence
$p\in \V{\om^0,0}{\rho v_k,s,k}(\proj\RR)$. Separate computations on
$(\infty,-1)_c$ and $(-1,0)$ show that $p$ satisfies the three term
relation. The extension to $\CC'$ given in Proposition~\ref{prop-extFE}
shows that $f$ extends holomorphically to the right half-plane. Then
the action of $S$ gives the holomorphy on the left half-plane.
\end{proof}

The cocycle $c^u_\pb$ on $\proj\QQ\times\proj\QQ$ has the property that
$c^u_\pb(0,\infty)$ cannot be changed by adding a coboundary $db$ where
$b\colon  \proj\QQ \rightarrow \V{\om^0,0}{\rho v_k,s,k}$ is equivariant
under $\Gm$, as the following lemma indicates:
\begin{lem}\label{lem-np}Let $s\neq 0$. Each
$f\in \V {\om^0,\infty}{\rho v_k,s,k}$ is analytic on $(\al,\infty)$
for some $\al\geq 0$. If $f|^\prs_{\rho v_k,s,k}T = f$ on
$(\al,\infty)$, then $f=0$.
\end{lem}
\begin{proof}
The first statement follows from the fact that $f$ is real-analytic on
$\proj\RR$ except at finitely many cusps of~$\Gm$.

Let $f_l$ be a component function of $f$ in the decomposition
in~\eqref{ucomp}. For all $t>\al$
\bad f_l(t)&\= e^{2\pi i m \k_l}\,\Bigl(
\frac{t-i}{t-i+m}\Bigr)^{s-k/2}\, \Bigl(
\frac{t+i}{t+i+m}\Bigr)^{s+k/2}\, f_l (t+m) &&\text{ for }m\in
\ZZ_{\geq0}\\
&\;\sim\; e^{2\pi i m \k_l}\, m^{-2s} \, (t^2+1)^{2s}\,
f_l(\infty)&\quad&\text{ as } m\rightarrow\infty\,.
\ead
If $s\neq 0$, then this implies that $f$ is the zero function.
\end{proof}

\rmrk{Role of $k$}At least for $s\not\equiv \frac k2\bmod 1$ we know
that all spaces $\A^0_{k'}(s,\rho v_k)$ with $k'\equiv k\bmod 2$ are
related by the weight shifting operators; see~\eqref{isowsh}. We do not
know the effect of the weight shifting operators on the associated
period functions.

%% file: rwt7-tro.tex


\section{Transfer operators}\label{sect-tro}

Proposition~\ref{prop-pfMcf} shows we can associate period function to
Maass cusp forms. In the introduction we indicated that discretization
of the geodesic flow on the sphere bundle of $\Gm\backslash \uhp$ leads
to transfer operators. Here we discuss such transfer operators, and
show that their eigenfunctions with eigenvalue $1$ lead to period
functions.

\subsection{Slow transfer operator} We denote by
\il{Comn}{$C^\om_\rho(I)$}$C^\om_\rho(I)$ the space of real-analytic
functions on the interval $I\subset \proj\RR$ with values in the
$1$-eigenspace of $|^\prs_{\rho v_k,s,k}(-I_2)$.

The transfer operators that we will discuss act on functions in
$C^\infty_\rho(0,\infty)$. Let \il{Gm'}{$\Gm'$}$\Gm'$ be the semigroup
in $\Gm$ generated by $T$ and $T'$. Since
$\dt^{-1} (0,\infty) \supset (0,\infty)$ for each $\dt\in \Gm'$, the
operator $|^\prs_{\rho v_k,s,k}\dt$ followed by restriction to
$(0,\infty)$ is well defined on $C^\om_\rho(0,\infty)$; it is not a
bijection. The restriction is understood in the formulas.

\begin{defn}\label{def-sltro}The \il{sltro}{transfer operator,
slow}\emph{slow transfer operator} is
\il{tro}{$\tro_{\rho v_k,s,r}$}
\[
\tro_{\rho v_k,s,k}\colon
C^\om_\rho(0,\infty) \rightarrow C^\om_\rho(0,\infty)\,,\quad
f \mapsto f|^\prs_{\rho v_k,s,k}\bigl( T+ T'\bigr)\,.
\]
\end{defn}
The period functions in Definition~\ref{def-FEom} are $1$-eigenfunctions
of the slow transfer operator. Since in the definition of
$C^\om_\rho(I)$ there are no conditions on the behavior near the
boundary points, there may be many more $1$-eigenfunctions of the slow
transfer operator than period functions.

\subsection{One-sided averages} We need results concerning the
\emph{Lerch transcendent} in \eqref{Lerch} below. Proposition~\ref{prop-Lerch} below was
shown in \cite{BCD}, starting from results of Kanemitsu, Katsurada
and Oshimoto in \cite{KKY} and Katsurada \cite{Kat}. See also~\cite[Proposition~A.1]{Fedosova_Pohl}. Lagarias and Li \cite{LL.II} give further going
information on the Lerch transcendent.

\begin{prop}\label{prop-Lerch}The \il{Ltrd}{Lerch transcendent}Lerch
transcendent\ir{Lerch}{H(s,\z,z)}
\be\label{Lerch} H(s,\z,z) \= \sum_{n\geq 0}\z^n\, \,(z+n)^{-s} \ee
converges absolutely for $\al>0$, $\re s>1$, $|\z|\leq 1$.
\begin{enumerate}[label=$\mathrm{(\roman*)}$, ref=$\mathrm{\roman*}$]
\item\label{eq:Lerch1} {\rm Meromorphic extension in $(s,z)$}
\begin{enumerate}[label=$\mathrm{(\alph*)}$, ref=$\mathrm{\alph*}$]
\item If $\z=1$, then $(s,z)\mapsto H(s,1,z)$ has a first order
singularity along $s=1$.
\item If $|\z|=1$, $\z\neq 1$, then $(s,z)\mapsto H(s,1,z)$ is
holomorphic on $\CC\times \left( \CC\setminus(-\infty,0]\right)$.
\end{enumerate}
\item\label{eq:Lerch2} {\rm Asymptotic behavior. }For $|\z|=1$ and $s\in \CC$ and
$N\in \ZZ_{\geq 0}$ there is an expansion
\be H(s,\z,z+1/2) \= \sum_{n=-1}^{N-1} C_n(\z,s) \,z^{-n-s} + \oh \bigl(
|z|^{-N-\re s}\bigr)\ee
on any region $\dt-\pi \leq \arg(z) \leq \pi-\dt$, $0<\dt\leq \pi$.

Furthermore, if $\z\neq 1$, then
\be C_{-1}(\z,s) \= 0\,.\ee
\end{enumerate}
\end{prop}

The following lemma defines so-called one-sided averages, which we will
use to define the fast transfer operator.
\begin{lem}\label{lem-1sa}Let $\al, \bt\in \RR$ satisfy $\bt <\al$.
\begin{enumerate}[label=$\mathrm{(\alph*)}$, ref=$\mathrm{\alph*}$]
\item\label{eq:1saa} For all $s\in \CC$ with $\re s>\frac{1}2$ the \emph{one-sided
averages}\ir{av}{\av+,\;\av-}\il{osav}{one-sided average}
\badl{av} \bigl( f|^\prs_{\rho v_k,s,k}\av+ \bigr)(t) &\=
\sum_{m=0}^\infty \bigl(f|^\prs_{\rho v_k,s,k} T^m\bigr)(t)\,,\\
\bigl( f|^\prs_{\rho v_k,s,k}\av- \bigr)(t) &\=
- \sum_{m\leq -1} \bigl(f|^\prs_{\rho v_k,s,k} T^m\bigr)(t)
\eadl
converge absolutely for all $f\in C^\om_\rho \bigl( (\al,\bt)_c\bigr)$
and define
\[ f|^{\prs}_{\rho v_k,s,k}\av+ \in C^\om_\rho\bigl(\al,\infty)\quad\text{
and }\quad f|^\prs_{\rho v_k,s,k}\av - \in C^\om_\rho\bigl(
(\infty,\bt+1)_c\bigr)
\,.\]
\item\label{eq:1sab} The operators $|^\prs_{\rho v_k,s,k}\av+$ and
$|^\prs_{\rho v_k,s,k}\av-$ commute with $|^\prs_{\rho v_k,s,k}T$.
\item\label{eq:1sac} For $f\in C^\om_\rho\bigl( (\al,\bt)_c)$ the function
$f|^\prs_{\rho v_k,s,k} \av+ |^\prs_{\rho v_k,s,k}(1-T) $ is equal to
the restriction of $f$ to $(\al,\infty)$, and the function
$f|^\prs_{\rho v_k,s,k} \av- |^\prs_{\rho v_k,s,k}(1-T) $ is equal to
the restriction of $f$ to $(\infty,\bt)_c$.
\item\label{eq:1sad} For each $f\in C^\om\bigl( (\al,\bt)_c\bigr)$ the functions
$(s,t)\mapsto \bigl(f|^\prs_{\rho v_k,s,k}\av\pm\bigr)(t)$ are
real-analytic in $(s,t)$ and holomorphic in $s$ on the region of
$(s,t)$ with $\re s>\frac{1}2$, and $t\in (\al,\infty)$, respectively
$t\in (\infty,\bt+1)_c$.
\end{enumerate}
\end{lem}
\begin{proof}
Similar results are proved earlier, in slightly differing contexts; see
e.g., \cite[\S4]{BCD} and \cite[\S7.6]{tahA}.

We note that $\infty\in (\al,\bt)_c$. The function $f$ is bounded on a
neighborhood of $\infty$ in $\proj\RR$. Absolute convergence follows
directly from
\be\label{avsum} \bigl(f|^\prs_{\rho v_k,s,k} T^m \bigr)(t)\= \Bigl(
\frac{t-i}{t-i+m}\Bigr)^{s-k/2} \, \Bigl(
\frac{t+i}{t+i+m}\Bigr)^{s+k/2}\, e^{-\pi i m k/6}\,\rho(T^{-m})
f(t+m)\ee
and the fact that the eigenvalues of the unitary operator $\rho(T)$ in
$X_\rho$ have absolute value~$1$. The other statements follow by
rearranging the order of the infinite sums, and the observation that
the absolute convergence is uniform for $(s,t)$ in compact sets.
\end{proof}

\begin{prop}\label{prop-1sav}Let $f\in C^\om\bigl( (0,-1)_c\bigr)$.
\begin{enumerate}[label=$\mathrm{(\roman*)}$, ref=$\mathrm{\roman*}$]
\item\label{eq:1sav1} The functions
$(s,t) \mapsto \bigl( f|^\prs_{\rho v_k,s,k}\av\pm\bigr)(t)$ extend as
real-analytic functions on
$\bigl\{ (s,t)\in \CC\times (0,\infty) \bigr\}$, respectively
$\bigl\{ (s,t) \in \CC\times (\infty,0)_c\bigr\}$ that are meromorphic
in $s$ with at most first order singularities in $s=\frac n2$ with
$n\in \ZZ_{\leq 1}$.

A singularity at $s=\frac12$ occurs if and only if there exists an
eigenvector $e_l$ of $\rho(T)$ with $\k_l=0$ and
$\bigl( f(\infty), e_l \bigr)_\rho \neq 0$.
\item\label{eq:1sav2} The assertions in \eqref{eq:1sab}, \eqref{eq:1sac} and \eqref{eq:1sad} in Lemma~\ref{lem-1sa} stay
valid for the extensions.
\item\label{eq:1sav3} We apply to $f$ the decomposition \eqref{ucomp}. There are
asymptotic expansions of the form
\badl{asympt-1sa} \bigl( f_l|^\prs_{\rho v_k,s,k}\av+\bigr)(t) &\sim
\sum_{n= -1}^{N-1} C_{n,l}(s) t^{-n}+\oh(t^{-N})&
\qquad &\text{as }t\uparrow\infty\,,\\
\bigl( f_l|^\prs_{\rho v_k,s,k}\av-\bigr)(t) &\sim \sum_{n=-1}^{N-1}
C_{n,l}(s) t^{-n}+\oh(t^{-N})&
&\text{as }t\downarrow-\infty\,,
\eadl
for each $N \in \ZZ_{\geq 0}$.
\end{enumerate}
\end{prop}
\begin{proof} The general approach in \cite{BLZ15, BCD, tahA} goes through with some adaptations.

We have to work with the components $f_l$ in the
decomposition~\eqref{ucomp}, and if $\k_l\neq0$, the factor
$e^{-2\pi i m \k_l}$ lead to the Lerch transcendent instead of the
Hurwitz zeta function. The factors
$\Bigl(\frac{t-i}{t-i+m}\Bigr)^{s-k/2} $ and
$ \Bigl( \frac{t+i}{t+i+m}\Bigr)^{s+k/2}$ have a more complicated
expansion in terms of powers of $t+m$. Taking this into account, we can
follow the approach in \cite[\S4.2]{BCD} to prove the theorem.
\end{proof}

\subsection{Fast transfer operator}Let $f \in C^\om_\rho(0,\infty)$.
Then the function $f|^\prs_{\rho v_k,s,k}T'$ is a function in
$C^\om\bigl( (0,-1)_c\bigr)$. So this function satisfies the condition
in Lemma~\ref{lem-1sa}, and we can form the \il{ftro}{transfer
operator, fast}\emph{fast transfer operator}\ir{trof}{ \trof{\rho
v_k,s,k}}
\be \label{trof}\trof{\rho v_k,s,k} f\= \bigl(f|^\prs_{\rho
v_k,s,k}T'\bigr)|^\prs_{\rho v_k,s,k} \av+ \= \sum_{n\geq 0}
f|^\prs_{\rho v_k,s,k} T' T^n\,.\ee

\begin{prop}\label{prop-ftro}The series in the definition of $\trof{\rho v_k,s,k} f$
in~\eqref{trof} converges absolutely for $\re s>\frac12$.
\begin{enumerate}[label=$\mathrm{(\roman*)}$, ref=$\mathrm{\roman*}$]
\item\label{eq:ft1} The family $s\mapsto f|^\prs_{\rho v_k,s,k}$ extends
meromorphically to $s\in \CC$ with at most first order singularities at
points of $\frac12\ZZ_{\leq 1}$.

A first order singularity occurs at $s=\frac12$ if and only if there is
an eigenvector $e_l$ of $v_k(T)\,\rho(T)$ with $\k_l=0$ for which
\be \Bigl( \rho(T')^{-1}\, v_k(T')^{-1} f (1) , e_l \Bigr)_\rho \neq
0\,. \ee
\item\label{eq:ft2} For all values of $s$ for which $\trof{\rho v_k,s,k} f$ is
holomorphic it defines a function in $C^\om(0,\infty)$ with an
asymptotic behavior as indicated in~\eqref{asympt-1sa}.
\item\label{eq:ft3} For all $f\in C^\om_\rho$
\be \label{dfT-1}\bigl(\trof{\rho v_k,s,k} f\bigr)|^\prs_{\rho
v_k,s,k}(1-T) \= f|^\prs_{\rho v_k,s,k}T'\= \bigl( \tro_{\rho v_k,s,k}
f \bigr)|^\prs_{\rho v_k,s,k} (1-T) \,. \ee
\item\label{eq:ft4} If $f$ is a $1$-eigenfunction of $\trof{\rho v_k,s,k}$ then
$f$ is a $1$-eigenfunction of $\tro_{\rho v_k,s,k}$.
\end{enumerate}
\end{prop}
\begin{proof}Most of these assertions follow directly from
Proposition~\ref{prop-1sav}. The relations in \eqref{dfT-1} and part~\eqref{eq:ft4} follow from Definition~\ref{def-sltro} and assertion~\eqref{eq:1sac} in
Lemma~\ref{lem-1sa}.
\end{proof}

\begin{prop}\label{prop-FE-eftrof}Let
$s\in \CC\setminus \frac12\ZZ_{\leq 0}$. If
$f\in \FE^\om_{\rho v_k,s,k}$, then $f$ is a $1$-eigenfunction of the
fast transfer operator $\trof{\rho v_k,s,k}$.
\end{prop}
\begin{proof}Let $f\in \FE^\om_{\rho v_k,s,k}$. Lemma~\ref{lem-f1c} and
part~\eqref{eq:ft1} of Proposition~\ref{prop-ftro} show that the fast transfer
operator $\trof {\rho v_k,s,k}f$ is holomorphic at $s=\frac12$.

By part~\eqref{eq:FEb} in Definition~\ref{def-FEom} and \eqref{dfT-1} we have
\[ \Bigl( \trof{\rho v_k,s,k} f - f \Bigr) |^\prs_{\rho v_k,s,k}(1-T) \=
0\,.\]
So the difference $p= \trof{\rho v_k,s,k} f - f $ is invariant under
$|^\prs_{\rho v_k , s,k}T$. Like in the proof of Lemma~\ref{lem-np} we
have, now for $t\in (0,\infty)$,
\be p(t)\= v_k(T)^q\, \rho(T)^q\,\Bigl(
\frac{t-i}{t-i+q}\Bigr)^{s-k/2}\, \Bigl(
\frac{t+i}{t+i+q}\Bigr)^{s+k/2}\, p(t+q)\,,\ee for all $q\in \ZZ$. The
limit as $t\uparrow\infty$ of $f$ exists by condition~\eqref{eq:FEc} in
Definition~\ref{def-FEom}, and the expansion of $\trof{\rho v_k,s,k} f$
can have a term with $t^1$, see~\eqref{asympt-1sa}. Thus, if $p\neq 0$
then it satisfies $p(t) \sim A_m t^m$ as $t\uparrow\infty$ for some
$m\in \ZZ_{\leq 1}$ and some $A_m \neq 0$. We go over to the
eigendecomposition \eqref{ucomp}. Taking $t_0\in (0,\infty)$ such that
$p_l(t_0) \neq 0$ we have
\bad p_l(t_0)&\= e^{2\pi i q \k_l}\Bigl(
\frac{t_0-i}{t_0-i+q}\Bigr)^{s-k/2}\, \Bigl(
\frac{t_0+i}{t_0+i+q}\Bigr)^{s+k/2}\,\,p_l(t_0+m)
\\
&\;\sim\; (t_0-i)^{s-k/2}(t_0+i)^{s+k/2} e^{2\pi i q \k_l}\, q^{-2s}\,
A_m\, (t_0+q)^m
\\
&\;\sim\; (t_0-i)^{s-k/2}(t_0+i)^{s+k/2} \, A_m \, e^{2\pi i
q\k_l}\,q^{m-2s}\qquad\text{ as } q\uparrow\infty\,.
\ead
This is possible only if $s=\frac m 2\in \frac12\ZZ_{\leq 1}$. In the
statement of the proposition these values of $s$ are excluded. Hence
$p=0$ and $\trof{\rho v_k,s,k} f = f$.
\end{proof}

%% file: rwt7-abg.tex


\section{Analytic boundary germs}\label{sect-abg}
The step from parabolic cohomology to Maass forms in \cite[\S12]{BLZ15}
is carried out by going over from principal series modules of analytic
functions of $\proj\RR$ to isomorphic modules of analytic boundary
germs. The latter modules allow us to use the geometry of the upper
half-plane to construct Maass forms from cocycles.

\subsection{Kernel function}In the construction of period functions
associated to Maass cusp forms we used the Poisson kernel $R_{s,k}$
defined on $\proj\RR\times\uhp$. We need to replace it by a kernel
function on $\uhp\times \uhp$ with similar properties.

Here it is useful to work on the universal covering group, discussed
in~\S\ref{sect-ucg}. We first describe a function $Q_{s,k}$ on
$\tG\setminus \tK$, where
\il{tK}{$\tK\subset \tG$}$\tK=\bigl\{\tkm(\th)\;:\;\th\in \RR \bigr\}$.
We have the \il{pd}{polar decomposition}\emph{polar decomposition}
$\tG \setminus \tK \=  \tK \tA_+\tK $, with
\il{tAta+}{$\tA,\; \tA_+$}$\tA=\bigl\{\tam(y)\;:\; y>0\bigr\}$ and
$\tA_+\= \bigl\{ \tam(y)\;:\; y>1\bigr\}$,
\il{ta}{$\tam(y)$}$\tam(y) \= \tppm(iy)$.

\begin{lem}\label{lem-Qsk}Let $s\in \CC$ and $k\in \RR$ satisfy
$s\not \in \frac12\ZZ_{\leq 1}\cup\left(-\frac k2+\ZZ\right)\cup 
\left(\frac k2+\ZZ\right)$. There is a function
$Q_{s,k} \in C^\infty(\tG \setminus \tK)$ satisfying
\begin{enumerate}[label=$\mathrm{(\alph*)}$, ref=$\mathrm{\alph*}$]
\item\label{eq:Qa} $Q_{s,k}\bigl( \tkm(\th_1) \tam(y) \tkm(\th_2) \bigr)
\= e^{ik(\th_1+\th_2)} Q_{s,k}\bigl(\tam(y)\bigr)$.
\item\label{eq:Qb} $\Dt Q_{s,k}=s(1-s)\, Q_{s,k}$.
\item\label{eq:Qc} $Q_{s,k}\bigl( \am(y) \bigr) = \oh\bigl( y^{-s}\bigr)$ as
$y\rightarrow\infty$.
\item\label{eq:Qd}
$Q_{s,k}\bigl( \tam(y) \bigr) = -\log(1-v) \, h_1(v) + h_2 (v)$ for
$v=\frac{4y}{(y+1)^2}$, with $C^\infty$-functions $h_1$ and $h_2$ on a
neighborhood of~$1$ in~$\RR$, and $h_1(1)=1$.
\item\label{eq:Qe} $Q_{s,k}(g^{-1} ) = Q_{s,-k}(g)$ for
$g\in\tilde G \setminus \tilde K$.
\end{enumerate}
\end{lem}
\begin{proof}Functions on $\tG$ satisfying a generalization of condition~\eqref{eq:Qa}
are needed to describe the polar expansion of scalar-valued Maass
forms at the point $i\in \uhp$. They are given in \cite[4.2.6 and
4.2.9]{B94} in terms of $u=\frac{(y-1)^2}{4y}$. Condition~\eqref{eq:Qb} imposes a
hypergeometric differential equation, with a two-dimensional solutions
space. For the expansion of Maass forms we need a solution that is
$C^\infty$ at $y=1$. Here we need a solution with a singularity at
$y=1$ that is small for $\re s\geq \frac12$ as $y\uparrow\infty$ and
$y\downarrow 0$. A multiple of the solution $\mu_k(i,ks+1/2)$ in
\cite[4.2.6]{B94} is the one that we need here.

Using \eqref{eq:Qa} and \eqref{eq:Qb} we obtain
\ir{Qsk}{\Q_{s,k}}
\badl{Qsk} Q_{s,k}&\bigl(\tkm(\th_1) \tam(y)\tkm(\th_2) \bigr)\=
\frac{\Gf(s-k/2) \, \Gf(s+k/2)}{\Gf(2s)} \, e^{ik (\th_1+\th_2)}\\
&\qquad\hbox{} \cdot
v^s\, \hypg21\left[ \atop{s-k/2,s+k/2}{2s} \Bigm| v \right]\,,\qquad v\=
\frac{ 4y}{(y+1)^2}\,.
\eadl
The
singularities of the solution are avoided by the condition on $s$ and
$k$ in the lemma. Since the hypergeometric function is holomorphic at
$v=0$ with value $1$, we get property~\eqref{eq:Qc}. There is a logarithmic
singularity at $v=1$. The gamma factors have been chosen such that the
hypergeometric function is $-\log (1-\nobreak v)$ as $v\uparrow 1$.
This leads to assertion~\eqref{eq:Qd}.

The function $Q_{s,k}\bigl( \tam(y) \bigr)$ is invariant under
$y\mapsto 1/y$. With
\[
\left( \tkm(\th_1)\tam(y)\tkm(\th_2)\right)^{-1}
= \tkm(-\th_2)\tam(1/y)\tkm(-\th_1)
\]
this implies assertion~\eqref{eq:Qe}.
\end{proof}

\begin{prop}\label{prop-qsk}Let $s\in \CC$, $k\in \RR$, and
$s\not \in \frac12\ZZ_{\leq 1}\cup
\left(-\frac k2+\ZZ\right)\cup \left(\frac k2+\ZZ\right)$. There is a
kernel function $q_{s,k}$ with the following properties:
\begin{enumerate}[label=$\mathrm{(\roman*)}$, ref=$\mathrm{\roman*}$]
\item\label{eq:q1} $q_{s,k}\in C^\infty
\Bigl(\Bigl\{(z_1,z_2\in \uhp^2\;;\; z_1\neq z_2\Bigr\}\Bigr)$,
\item\label{eq:q2} $\Dt_k q_{s,k} (z_1,\cdot)\= s(1-s) \, q_{s,k}(z_1,\cdot)$,
and $\Dt_{-k} q_{s,k} (\cdot,z_2)\= s(1-s) \, q_{s,k}(\cdot , z_2)$,
\item\label{eq:q3} $q_{s,k}(z_2,z_1) = q_{s,-k}(z_1,z_2)$,
\item\label{eq:q4} $q_{s,k} \left( \bigm|_{-k}\times \bigm|_k\right)g \=q_{s,k}$
for all $g\in G$. Here $|_{-k}g$ acts on the first variable and $|_k g$
on the second variable.
\end{enumerate}
\end{prop}
\begin{proof}
We take for $z_1\neq z_2 \in \uhp$\ir{qsk}{q_{s,k}(\cdot,\cdot)}
\be \label{qsk}
q_{s,k}(z_1,z_2) \= Q_{s,k} \bigl( \tppm(z_1)^{-1}\tppm(z_2)\bigr)\,,
\ee
with $\tppm(z)\in \tG$ as discussed in~\S\ref{sect-ucg}. This satisfies
assertion~\eqref{eq:q1}. Relation~\eqref{eq:q3} follows from Lemma~\ref{lem-Qsk}\eqref{eq:Qe}.

The differential operator $\Dt$ in \eqref{DtXpm} commutes with left
translation, and corresponds to $\Dt_k$ in~\eqref{Dtk} on functions in
weight~$k$. This implies that
$\Dt_k q_{s,k}(z_1,\cdot)=s(1-s)q_{s,k}(z_1,\cdot)$. With~\eqref{eq:q3} this
implies $\Dt_{-k}q_{s,k}(\cdot,z_2)=s(1-s) q_{s,k}(\cdot,z_2)$ as well.

We have for all $\tilde g\in \tG$
\[ Q_{s,k}\bigl( \tppm(z_1)^{-1} \tppm(z_2) \bigr) \= Q_{s,k} \Bigl(
\bigl(\tilde g\tppm(z_1)\bigr)^{-1} \tilde g\tppm(z_2) \Bigr)\,.\]
Hence assertion~\eqref{eq:q4} follows from~\eqref{Psig}.
\end{proof}

\rmrk{Use of the disk model}Let $z_1\in \uhp$ be near to $i$ and
$z_2=i$. Then $w_1=\frac{z_1-i}{z_2+i}$ is near to $0$. Taking
$\th_1=\frac12 \arg(w_1)$ and $\th_2 = \th_1-\frac \pi 2+\arg(z+i)$ we
can check that
\be \tppm(z_1) = \tkm(\th_1)\tam(t) \tkm(\th_2)\,,\ee
with $t= \frac{1+|w_1|}{1-|w_1|}$. Then $\tppm(z_1) = 
\tkm(\th_1) \tam(t) \tkm(\th_2)$ with
$\th_1+\th_2 = \frac \pi2-\arg(z+i)$.

Hence
\be\label{pc} q_{s,k}(z_1,i) \= Q_{s,-k}\bigl( \tppm(z)1) \bigr) \=
e^{-ik(\th_1+\th_2)}\,Q_{s,-k}\Bigl(\frac{1+|w_1|}{1-|w_1|}\Bigr)\,.\ee

\subsection{Integration with the kernel function} We generalize the
integral formula in \cite[Theorem 1.1]{BLZ15}, proved in \cite[Theorem
3.1]{BLZ13} (quoted in \cite{BLZ15} as Theorem~2.1).

\begin{prop}\label{prop-psC}Let $C$ be a piecewise smooth positively
oriented simple closed curve in $\uhp$ and let $U$ be an open region in
$\uhp$ containing the curve $C$ and its interior. If $u\in C^\infty(U)$
satisfies $\Dt_k u = s(1-s) u$, then for $z_2\in \uhp\setminus C$
\be \int _C \bigl[ u, q_{s,k}(\cdot, z_2) \bigr]_k \= \begin{cases}
2\pi i \, u(z_2)&\text{ if $z_2$ is inside $C$}\,,\\
0&\text{ if $z_2$ is outside $C$}\,.
\end{cases}
\ee
\end{prop}
\begin{proof}With \eqref{Gf} we have a differential form
$\bigl[ u, q_{s,k}(\cdot,z_2)\bigr]_k$ on $U\setminus  \{z_2\}$, which
is closed by \eqref{dhk} and Proposition~\ref{prop-qsk}\eqref{eq:q2}.
The integral
\[ \int_C \bigl[ u, q_{s,k}(\cdot,z_2)\bigr]_k \]
does not change if we deform the path $C$ continuously in
$U \setminus\{z_2\}$. In particular, the integral is zero if $z_2$ is
outside $C$. We proceed under the assumption that $z_2$ is inside $C$.

Let us take $g=\ppm(z_2)\in G$, with
\ir{ppm}{\ppm(z)}$\ppm(x+iy) \= \pr\tppm(x+iy)
\= \rmatc1x01 \rmatc {y^{1/2}}00{y^{-1/2}}$. Then $z_2 = g i$, and
$C_1 = g^{-1}C$ encircles $i$ once, contained in the open set $g^{-1}U$
containing $i$. With \eqref{Gf-inv} we obtain
\be \int_C \bigl[ u, q_{s,k}(\cdot,z_2)\bigr]_k \= \int_{C_1} \bigl[
u|_k g, q_{s,k}(\cdot,i) \bigr]_k\,.\ee
We shrink the curve $C_1$ to a small hyperbolic circle around $i$.

We use the disk model, with coordinate $w=\frac{z-i}{z+i}$. Then we can
take $C_1$ as a circle around $w=0$ with radius $\oh(\e)$ and let
$\e\downarrow 0$. We use the description \eqref{Gfdm} for the Green's
form. The function $a$ corresponds to $u$, and the function $b$ to
$z_1\mapsto q_{s,k}(z_1,i)$. By Lemma~\ref{lem-Qsk}\eqref{eq:Qd} we obtain
\begin{align*}
b &\= (1-w)^{k/2}(1-\bar w)^{-k/2} \Bigl(- h_1(1-w\bar w)\, \log(w \bar
w)
+ h_2(1-w\bar w)\Bigr)\= \oh(\log \e)\,,
\displaybreak[0]\\
\partial_{\bar w}b&\=\frac {-1}{2\bar w}(1-w)^{-k/2}(1-\bar w)^{k/2-1}
\Bigl( k \bar w h_2(1-|w|^2)\\
&\qquad\hbox{}
- h_1(1-|w|^2) \bigl( k\bar w \log|w|^2+2\bar w-2\bigr) \\
&\qquad\hbox{}
-2 |w|^2(1-\bar w)\bigl( h_1'(1-|w|^2) \log|w|^2- h_2'(1-|w|^2) \bigr)
\Bigr)\displaybreak[0]\\
&\= \frac{-1}{2} \, \bar w^{-1} \Bigl( 2
+ \oh(\e\log\e) \Bigr)\,.
\end{align*}

We write $w= \e e^{i \ph}$. The first term in \eqref{Gfdm} is
\be \Bigl( (\partial_w a ) \, b + \frac{k(1-\bar w)}{2(1-w)(1-|w|^2)} a
b \Bigr) \, dw \= \oh\bigl(\log\e \bigr) i \e \, e^{i\ph}\, d\ph =
o(1)\,, \ee
and does not contribute to the integral. The second term is
\begin{align*}
\Bigl( a\, (\partial_{\bar w} b) &+ \frac{k(1- w)}{2(1-\bar w)(1-|w|^2}
a b \Bigr)\, d\bar w\\
& \= \Bigl( -a\,\e^{-1} e^{i\ph}+ a\, \oh(\log\e) + \oh(\log\e)
\Bigr)(-i \e) e^{-i\ph}\, d\ph
\end{align*}
This gives in the limit $\e\downarrow 0$ for the total integral the
value
\[ 2\pi \, a(0)= 2\pi i \, \bigl( u|_l \tppm(z_2)\bigr)(i)
\= 2\pi i \,e ^{k\cdot 0} u(z_2)\,. \qedhere\]
\end{proof}

\subsection{Boundary germs}In Proposition~\ref{prop-pfMcf} we associated
to a Maass cusp form in $\A^0_k(s,\rho v_k)$ a $1$-cocycle $c^u_\pb$ on
$\proj\QQ\times\proj\QQ$ with values in the module
$\V {\om^0,0}{\rho v_k,s,k}(\proj\RR)$, which contains
$\V \om{\rho v_k,s,k}(\proj\RR)$ and is contained in
$\V0{\rho v_k,s,k}(\proj\RR)$. To go back from period functions to
Maass cusp forms we go over to modules of boundary germs that are
isomorphic to the principal series modules
$\V {\om }{\rho v_k,s,k}(\proj\RR)$.

\rmrk{Sheaves related to eigenfunctions of \intitle{\Dt_k}}For each open
set $\Om\subset \uhp$ we put
\be \E_{s,k}(\Om)\= \bigl\{ f\in C^\infty(\Om)\;:\; \Dt_k f =
s(1-s)\,f\bigr\}\,.\ee
This defines a sheaf \il{Esh}{$\E_{s,k}$}$\E_{s,k}$ on~$\uhp$.

We turn to subsets $\Om\subset \proj\CC$ that have a non-empty
intersection with $\uhp$. We put $\B_{s,k}(\Om) \= \E_{s,k}(\Om)$ if
$\Om \subset \uhp$, and if $\Om \cap \proj\RR\neq\emptyset$, then we
put\ir{Bdef}{\B_{s,k}}\il{Phsk}{$\Ph_{s,k}$}
\badl{Bdef} \B_{s,k}(\Om) &\= \Bigl\{ f\in \E_{s,k}(\Om \cap \uhp)\;:\;
\text{ the function }F(z) \= \Ph_{s,k}(z) f(z) \\
&\qquad \text{ on }\Om\cap\uhp \text{ extends to a real-analytic
function on $\Om$}\Bigr\}\,,\\
\Ph_{s,k}(z)
&\= y^{-s} \, (z+i)^{s+k/2}\, (\bar z-i)^{s-k/2}\,.
\eadl
The argument of $z+i$ is in $[0,\pi]$ for $z\in\uhp\cup \RR$, and the
argument of $\bar z-i$ is in $[-\pi,0]$. In the coordinate
$w=\frac{z-i}{z+i}$
\be\label{Phi-w}
\Ph_{s,k}(w) \= 4^2\, e^{\pi i k/2}\,
(1-w)^{-k/2}\,(1-\bar w)^{k/2}\, \bigl(1-|w|^2\bigr)^{-s}\,.\ee

\rmrk{Remarks}
\begin{enumerate}[label=$\mathrm{(\arabic*)}$, ref=$\mathrm{\arabic*}$]
\item Any $f\in \E_{s,k}(\Om \cap \uhp)$ is real-analytic,
since $\Dt_k$ is an elliptic operator. It is far from sure that for
$f\in \E_{s,k}(\Om \cap \uhp)$ the function $\Ph_{s,k} \, f$ has a
real-analytic continuation to $\Om$. The analyticity of the
continuation is an additional requirement. It determines $f$ uniquely
on all open connected subsets of $\Om$ that contain $\Om\cap \uhp$.
\item An example is the function $z\mapsto y^s$, which is in
$\B\Bigl( \bigl\{ z\in \CC\;:\; \im z>-1\bigr\}\Bigr)$, where the
restriction $\im z>-1$ arises from the singularity of $\Ph_{s,k}$ at
$-i$.
\item Another example, defined on the region $|z|>1$, is $f(z) = \im(-1/z)^s$,
which leads to $F(z) = \bigl(1+i/z\bigr)^s\, \bigl( 1-i/\bar z)^s$.
\end{enumerate}

\begin{defn}The space of \il{abg}{analytic boundary germ}\emph{analytic
boundary germs} on an open set $I \subset\proj\RR$ is\ir{Womsh}{\W
\om{s,k}}
\be \label{Womsh}\W\om{s,k}(I) \= \lim_{\stackrel \Om\rightarrow}
\B_{s,k}(\Om)\,,
\ee
where $\Om$ runs over the open sets in $\proj\CC$ that contain $I$.
\end{defn}

For $I=\proj\RR$ the elements of $\W \om{s,k}( \proj\RR)$ are
represented by real-analytic functions on an annulus
$1-\e <\Bigl|\frac{z-i}{z+i} \Bigr| < 1+\e$ such that $\Dt_k f=s(1-s)f$
on $1-\e<\Bigl|\frac{z-i}{z+i}\Bigr|<1$.

The use of the direct limit in~\eqref{Womsh} implies that for
representatives $f_1\in \B_{s,k}(\Om_1)$ and $f_2 \in \B_{s,k}( \Om_2)$
of $\ph $, the functions $F_1=\Ph_{s,k}f_1$ and $F_2 = \Ph_{s,k} f_2$
have real-analytic extensions that coincide on $\Om_1\cap \Om_2$. This
implies that $I\mapsto \W\om{s,k}(I)$ is a sheaf.

\begin{defn}The \il{res}{restriction morphism}\emph{restriction
morphism}
\il{rest}{$\rest_{s,j}$}$\rest_{s,k}\colon \W\om{s,k}\rightarrow \V\om{s,k}$
is induced by assigning to $f\in \B_{s,k}(\Om)$ the restriction of
$F=\Ph_{s,k}f$ to $\Om \cap \proj\RR$.
\end{defn}
For example the function $h(z) = y^s$ in
$\B _{s,k}\bigl(\left\{ z\in \CC\;:\;\im z>-1\right\}\bigr)$ leads to the
restriction $x\mapsto (x+i)^{s+k/2}(x-i)^{s-k/2}$, which is analytic on
$\RR$.
\begin{lem}\label{lem-intertwW}For each $g\in G$ the operators
$|_k g\colon\B_{s,k}(\Om) \rightarrow \B_{s,k}(g^{-1}\Om)$ with
$\Om \supset I$ induce an operator
$|_k g\colon \W\om{s,k}(I) \rightarrow \W\om{s,k}(g^{-1}I)$.
Furthermore,
\be \label{res-intertw}(\rest_{s,k}\ph) |^\prs_{s,k} f  \=
\rest_{s,k}\bigl(f|_{k} g\bigr) \,.\ee
\end{lem}
We note that principal series action $|^\prs_{s,k}g$ on the sections
$\rest_{s,k}\ph$ of $\V \om{s,k}$ is related to the action $|_k g$ on
boundary germs. The latter action does not depend on~$s$.
\begin{proof}
The existence of the operators $|_k g$ follows from the direct limit
definition of $\W\om{s,k}(I)$.

For $g=\rmatc abcd$ near to $I_2\in G$ a check of \eqref{res-intertw} is
a long but straightforward computation.
For the right hand side we know that the quantity $F_g(z)$ defined by
\be\label{r} F_g(z) \= \Ph_{s,k}(z)\, e^{-i k \arg(cz+d)} f(g z)\ee
for $z\in g^{-1} \Om \cap \uhp$ extends to $g^{-1}\Om$.
For the left hand side we have
\be F(z) \= \Ph_{k,s}(z) \, f(z)\ee
on $\Om \cap \uhp$, and we know that it extends to $\Om$. We can eliminate $f$ from the relation,
and end up with a relation in terms of $z$ and $\bar z$.
 Working out this relation takes some care with powers of complex quantities with complex
exponents, but for $g\approx I_2$ this causes no problems. Then we
substitute $z=t$ and $\bar z=t$ with $t\in g^{-1} I$, and observe that we get the factor
in~\eqref{prs-trf}. The resulting relation extends as the equality of
two multi-valued real-analytic functions on~$G$. We have chosen the
branches for $|_kg$ and $|^\prs_{s,k}g$ in the same way.
\end{proof}

The restriction morphism is not a morphism of $G$-equivariant sheaves.
Tensoring with $X_\rho$ we get a morphism of $\Gm$-equivariant sheaves
\il{Womrvsk}{$\W\om{\tro v_k,s,k}$}$\rest_{s,k}\colon
\W\om{\rho v_k,s,k} \rightarrow \V\om{\rho v_k,s,k}$.

\begin{prop}[Kernel functions $R_{s,k}$ and $q_{s,k}$]
\label{prop-qR}
For $2s\not\equiv k\bmod 2$\ir{bsk}{b(s,k)}
\badl{bsk} \bigl(\rest_{s,k} q_{s,k}(z_1,\cdot)\bigr)(t)&\=
b(s,k)
R_{s,k}(t,z_1)\,,\\
b(s,k) &\= e^{\pi i k/2}\,\frac{\Gf(s-k/2)\,\Gf(s+k/2)}{\Gf(2s)}\,.
\eadl
\end{prop}
\begin{proof}In part~\eqref{eq:q4} of Proposition~\ref{prop-qsk} the kernel
function $q_{s,k}$ transforms with weight $-k$ in $z_1$ and with weight
$k$ in $z_2$. The Poisson kernel $R_{s,k}(t,z)$ transforms with weight
$-k$ in $z$, and with a principal series action of weight $k$ in~$t$;
see~\eqref{diagactR}. This shows that it is sensible to compare the
functions $z_2\mapsto q_{s,k}(z_1,z_2)$ and $t\mapsto R_{s,k}(t,z_1)$.

The transformation behavior of both kernels implies that it suffices to
take $z_1=i$. We denote the gamma factors in~\eqref{Qsk} by
\be \mathrm{Gf} \= \frac{\Gf(s-k/2)\,\Gf(s+k/2)}{\Gf(2s)}\,,\ee
and obtain with \eqref{Qsk}
\begin{align*}
q_{s,k}(i,z) &\= (1+i \bar z)^{k/2} \, (1-i z)^{-k/2} \, \mathrm{Gf}\;
\Bigl( \frac y{|z+i|^2}\Bigr)^s\, \hypg21\Bigl(\frac y{|z+i|^2}
\Bigr)\,,\displaybreak[0]\\
\Phi_{s,k}(z) &\= y^{-s} (z+i)^{s+k/2} \,(\bar z-i)^{s-k/2}\,,
\displaybreak[0]\\
F(z)&\= e^{\pi i k/2} \, \mathrm{Gf}\; \hypg21\Bigl(\frac y{|z+i|^2}
\Bigr)\,.
\end{align*}
The hypergeometric factor equals $1$ for $y=0$,
hence we get
$F(x) = e^{\pi i k/2} \, \mathrm{Gf}$ not depending on  $x\in \RR$, and then also for
$x=\infty$ by analytic continuation. Thus,
\be \bigl( \rest_{s,k} q_{s,k}(i,\cdot)\bigr)(t) \= e^{\pi i k/2} \,
\mathrm{Gf}\,.\ee
We observe in~\eqref{Poisk} that $R_{s,k}(t,i)=1$, which completes the proof.
\end{proof}

\begin{thm}\label{thm-resiso}Let $s\in \CC\setminus \ZZ_{\leq 0}$, and
$2s\neq \pm k$. The restriction morphism
\[
\rest_{s,k}\colon\W\om{s,k}\rightarrow \V\om{s,k}
\]
is bijective and
$\bigl( \rest_{s,k} f \bigr)|^\prs_{s,k}g = \rest_{s,k}\bigl( f|_k g\bigr)$
for all $g\in G$ for representatives $f$ of sections of $\W\om{s,k}$.

If $\rest_{s,k}\ph \in \W\om{s,k}(I)$ extends holomorphically to a
convex neighborhood $\Om$ in $\proj\CC$ of the open interval
$I\subset\proj\RR$ not containing $i$ and~$-i$ and symmetric under
complex conjugation, then $\ph$ can be represented by
$f\in \B_{s,k}(\Om)$ for the same neighborhood~$\Om$.
\end{thm}

\begin{proof}Lemma~\eqref{lem-intertwW} gives the intertwining property
of the operators $|_k g$. Hence we can work with sections over an
interval~$I$ contained in $\RR$.
Let $f\in \B_{s,k}(\Om)$ represent a section $\ph \in \W\om{s,k}(I)$.
Near~$I$ we have
\be \Ph_{s,k}(x+i y) \= y^{-s}\, \al(x,y)\,,\ee
with
$\al(x,y) = \bigl( x+iy+i\bigr)^{s+k/2}\,\bigl(x-i y-i\bigr)^{s-k/2}$.
Since the factor $\al$ is real-analytic without zeros on the strip
$|y|<1$ in~$\uhp$, the function
\be H(x,y) = y^{-s}\, f(x,y) \= F(x,y)/\al(x,y)\ee
is also real-analytic on $\Om$, and we can work with $H$ instead of $F$.

For the injectivity we suppose that $\ph=0$ on $I$, and have to show
that then $H=0$ on a neighborhood of $I$ in $\CC$. The differential
equation $\Dt_k f = s(1-s) f$ implies that $H$ satisfies
\be - y^2\Bigl( \partial_x^2 H + \partial_y^2 H\Bigr) - 2 s y
\,\partial_y
H + i k y \,\partial_x H\=0\,.\ee
Since $H$ is real-analytic there is for each $x\in I$ an expansion
$ \sum_{n\geq 0} a_n(x)y^n$ converging to $H(x,y)$ for $y$ in an open
interval containing $0$. (This interval may depend on~$x$.) Inserting
this into the differential equation we get
\be\label{arel} a_n(x) \= \begin{cases} \frac {ik}{2s} a_0'(x)&\text{ if
}n=1\,,\\[2mm]
\dfrac{ik a_{n-1}'(x) - a_{n-2}''(x)}{n(2s+n-1)} &\text{ if }n\geq
2\,.\end{cases}\ee
(We use that $s\not\in \ZZ_{\leq 0}$.)

Since $a_0(x) = \ph(x)$, the function $a_n$ can be written as
\be a_n(x) \= p_n \, \ph^{(n)}(x)\,,\ee
with coefficients depending on $s$ and $k$. If $\ph=0$, then $H$
vanishes on a neighborhood of $I$. Hence the restriction map is
injective.
(In the case $k=0$ there is a nice formula for the $a_n$ in
\cite[(5.15)]{BLZ13}. We did not try to find a similar formula for
general real weights.)
\medskip

In \cite[\S5.2]{BLZ13} the surjectivity of the restriction is shown in
two ways: With a power series expansion (Theorem 5.6 in \cite{BLZ13}) and with an
integral representation (Theorem 5.7 in \cite{BLZ13}). Here we try to generalize the
latter approach.

Let $\Om$ be a neighborhood of $I$ with the properties indicated in the
theorem. For given $\ph \in \V \om{\rho v_k,s,k}(I)$ extending
holomorphically to $\Om$ we put
\be\label{bskf} f(z) \= \frac1{i 2^{2s-1}b(s,k)} \int_{t=\bar z}^z
R_{1-s,-k}(t;z) \, \ph(t)\, \frac{dt}{1+t^2}\,, \ee
initially for $\re s>\frac12|k|$. The corresponding function $F$
in~\eqref{Bdef} is
\begin{align*}
F(z) &\= \Ph_{k,s}(z) f(z)
\=\frac1{i 2^{2s-1} b(s,k)}\, \int_{t=\bar z}^z (z+i)^{s+k/2}\,(\bar
z-i)^{s-k/2}\\
&\qquad\hbox{} \cdot
\Bigl( \frac{t-i}{t-z}\Bigr)^{1-s+k/2}\, \Bigl( \frac{t+i}{t-\bar
z}\Bigr)^{1-s-k/2}\, \ph(t)\, \frac{dt}{t^2+1}
\displaybreak[0]\\
&\= \frac1{i 2^{2s-1}b(s,k)} \int_{t=-i}^i (x+i y+i)^{s+k/2}\,(x-i
y-i)^{s-k/2}\\
&\qquad\hbox{} \cdot
\Bigl( \frac{ty+x-i}{t-i}\Bigr)^{-s+k/2}\, \Bigl( \frac{ty+ x+i}{t+i}
\Bigr)^{-s-k/2} \ph(x+y t)\, \frac{dt}{1+t^2}\,.
\end{align*}
It is clear that the integral converges absolutely for $\re s$
sufficiently large, and that it describes a real-analytic function in
$z=x+i y$. We take the value at $y=0$:
\badl{Fres} F(x) &\= \frac1{i 2^{2s-1}b(s,k)}(x+i)^{s+k/2}\,
(x-i)^{s-k/2} \ph(x)
\\
&\qquad\hbox{} \cdot
\int_{t=-i}^i \Bigl( \frac{x-i}{t-i}\Bigr)^{-s+k/2} \,\Bigl(
\frac{x+i}{t+i}\Bigr)^{-s-k/2}\, \frac{dt}{1+t^2} \= \ph(x)\,.
\eadl
Under the assumption that $\re s$ is large, this shows that $\ph$ occurs
as the restriction of $f$. That is the surjectivity of $\rest_{s,k}$.
Moreover, if $\ph$ is holomorphic on a set $\Om$ as indicated in the
theorem, then $F$ is real-analytic on $\Om$, and furthermore
$f\in \E_{s,k}\bigl( \Om\cap\uhp\bigr)$.\medskip

The integral
\be\label{i0}\int_{t=-i}^i \, \Bigl( \frac{ty+x-i}{t-i}\Bigr)^{-s+k/2}\,
\Bigl( \frac{ty+ x+i}{t+i} \Bigr)^{-s-k/2} \ph(x+y t)\,
\frac{dt}{1+t^2} \ee
is holomorphic in $(s,k)\in \CC^2$ on the region $\re s>|\re k|/2$. We
aim at a meromorphic continuation for $(s,k) \in \CC^2$. As $t$ runs
from $-i$ to $i$ the term $ty+x$ runs from $t_-=\bar z$ to $t_+=z$. The
factor $\Bigl( \frac{ty+x-i}{t-i}\Bigr)^{-s+k/2}$ is holomorphic on
$\proj\CC$ except for a path from $t_+=z$ to $i$, which we choose as
indicated in Figure~\ref{Pchm}. The other factor is well-defined
outside the path from $t_-=\bar z$ to $-i$ in the figure.

We replace the integration over $[-i,i]$ in~\eqref{i0} by integration
over the Pochhammer contour $P$ sketched in Figure~\ref{Pchm}. For
$\Om$ with the properties mentioned in the theorem we can arrange that
the contour $P$ is contained in $\Om$. The paths from $t_\pm$ to
$\pm i$ are important only on the contour. A given choice of $p_\pm$
can be used for $z$ varying over compact sets.
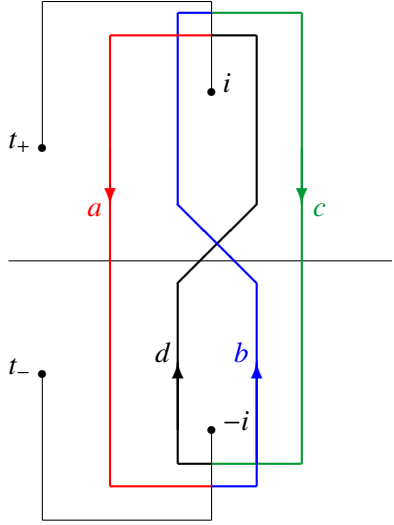
\begin{figure}[ht]
\[
\setlength\unitlength{1.5cm}
\begin{picture}(3.4,5)(-.4,-2.5)
\put(-.3,0){\line(1,0){3.4}}
\put(0,1){\circle*{.08}}
\put(0,-1){\circle*{.08}}
\put(1.5,1.5){\circle*{.08}}
\put(1.5,-1.5){\circle*{.08}}
\put(-.3,1){$t_+$}
\put(-.3,-1){$t_-$}
\put(1.6,1.5){$i$}
\put(1.6,-1.5){$-i$}
\put(0,1){\line(0,1){1.3}}
\put(0,2.3){\line(1,0){1.5}}
\put(1.5,2.3){\line(0,-1){.8}}
\put(0,-1){\line(0,-1){1.3}}
\put(0,-2.3){\line(1,0){1.5}}
\put(1.5,-2.3){\line(0,1){.8}}
\thicklines
\put(1.5,2){\line(1,0){.4}}
\put(1.9,2){\line(0,-1){1.5}}
\put(1.9,.5){\line(-1,-1){.7}}
\put(1.2,-.2){\line(0,-1){1.3}}
\put(1.2,-1.8){\vector(0,1){.9}}
\put(1.5,-1.8){\line(-1,0){.3}}
\put(1,-.9){$d$} \green
\put(1.5,-1.8){\line(1,0){.8}}
\put(2.3,-1.8){\line(0,1){4}}
\put(2.3,1){\vector(0,-1){.5}}
\put(2.3,2.2){\line(-1,0){.8}}
\put(2.4,.4){$c$}
\endgreen \blue
\put(1.5,2.2){\line(-1,0){.3}}
\put(1.2,2.2){\line(0,-1){1.7}}
\put(1.2,.5){\line(1,-1){.7}} 
\put(1.9,-.2){\line(0,-1){1.8}}
\put(1.9,-2){\line(-1,0){.4}}
\put(1.9,-1.8){\vector(0,1){.9}}
\put(1.7,-.9){$b$}
\endblue \red
\put(.6,-2){\line(0,1){4}}
\put(.6,2){\line(1,0){.9}}
\put(.6,-2){\line(1,0){.9}}
\put(.6,1){\vector(0,-1){.5}}
\put(.4,.4){$a$}
\endred
\end{picture}\]
\caption{Pochhammer contour} \label{Pchm}
\end{figure}

We conclude that
\be \label{i1} \int_{t\in P} \Bigl( \frac{ty+x-i}{t-i}\Bigr)^{-s+k/2}\,
\Bigl( \frac{ty+ x+i}{t+i} \Bigr)^{-s-k/2} \ph(x+y t)\,
\frac{dt}{1+t^2} \ee
depends analytically on $(s,k,z) \in \CC^2\times \Om$, holomorphically
depending on $(s,k)$.

To relate the outcome of \eqref{i1} to the outcome of \eqref{i0} we take
$\re s> \frac12|\re k|$. Then we can compute the integral over the
Pochhammer contour as a linear combination of four integrals from $-i$
to $i$. We take the arguments in such a way that at $t=0$ on part $a$
the argument of $\frac{ty+x-i}{t-i} = \arg(1+i x) \in (0,\pi)$, and the
argument of $\frac{ty+x+i}{t+i}$ is equal to $\arg(1-i x)\in (-\pi,0)$.
That is the choice of the arguments that we use in the computation
of~\eqref{Fres}. In this way the transition from \eqref{i0} to
\eqref{i1} amounts to multiplication by
\[ -1+ e^{-2\pi i s-\pi i k} - e^{-2\pi i k} + e^{2\pi i s-\pi i k} \= 4
e^{-\pi i k/2} \,\sin\Bigl(\pi \frac{k}2-\pi s\Bigr)\, \sin\Bigl(\pi
\frac{k}2+\pi s\Bigr)
\,.\]
Hence $f$ and $F$ have a meromorphic extension in $(s,k)$ with
singularities occurring only in $s=\pm \frac k2$.
\end{proof}

\subsection{Restriction and one-sided averages}Let $\bt<\al$. Based on a
fixed function $\ph \in C^\om_\rho\bigl( (\al,\bt)_c\bigr)$ we have the
meromorphic family
\[ s\mapsto f_s(z) =\rest_{s,k}^{-1} \ph\,(z) \= \frac1{b(s,k)}
\int_{t=\bar z}^z R_{1-s,-k}(t,z)\, \ph(t)\frac{dt}{1+t^2}\,.\]
Like in \cite[Lemma 4.6]{BLZ15} we can try to get the lower row of the
following scheme:
\be\xymatrix@C2cm{ \ph \ar[r]^{\av+} \ar[d]_{\rest_{s,k}^{-1}}
& \ph|^\prs_{\rho v_k,s,k} \av+ \ar[d]_{\rest_{s,k}^{-1}}
\\
f_s \ar[r]^{\av+ ?} & \rest_{s,k}^{-1}\bigl( \ph|^\prs_{\rho v_k,s,k}
\av+\bigr)
}\ee

We formulate the result that we will use later on.
\begin{lem}\label{lem-Avbg}
Let $\re s>0$, and denote
\begin{align*} \Om_\pm &\= \proj\CC \setminus \Bigl\{ z\in \CC\;:\;
\Bigl|z\pm \frac12\Bigr|\leq \frac12 \Bigr\}\,,\\
\Ups_\pm &\= \bigl\{ z\in \CC\;:\; \pm\re z>0 \bigr\}\,.
\end{align*}
See Figure~\ref{fig-OU}.

Suppose that $h_\pm\in X_\rho \otimes \B_{s,k}(\Om_\pm)$ and that the
associated real-analytic functions $H_\pm=\Ph_{s,k} h_\pm$ on $\Om_\pm$
satisfy $ \bigl( F_\pm(\infty) ,e_l\bigr)_\rho=0$ for all basis
elements $e_l$ with $\k_l=0$. (See \S\ref{sect-Fe}.)
Then
\be\label{avPdefbg} f_\pm|_{\rho v_k,k}\av\pm \=
\begin{cases} \sum_{m\geq 0} f_+|_{\rho v_k,k} T^m\\
-\sum_{m\leq
-1} f_-|_{\rho v_k,k} T^m
\end{cases}\ee
are well-defined elements of $X_\rho \otimes B_{s,l}(\Ups_\pm)$ which
satisfy
\be\label{resrel} \rest_{s,k} \bigl(f_\pm|_{\rho v_k , k}\av\pm \bigr)
\= \bigl( \rest_{s,k} f_\pm \bigr)\bigm|^\prs_{\rho v_k,s,k}\av\pm\,.\ee
Moreover, for $l=1,\ldots ,n(\rho)$:
\be \label{avasbg}\bigl( f_\pm|_{\rho v_k,k}\av\pm(x+i y),
e_l\bigr)_\rho \= \oh\bigl( y^{-s}\bigr)
\ee
as $y \uparrow\infty$, uniform for $x$ in compact sets in
$\Ups_\pm \cap \RR$.
\end{lem}
\begin{figure}[ht]
\begin{center}
\includegraphics[width=.8\textwidth]{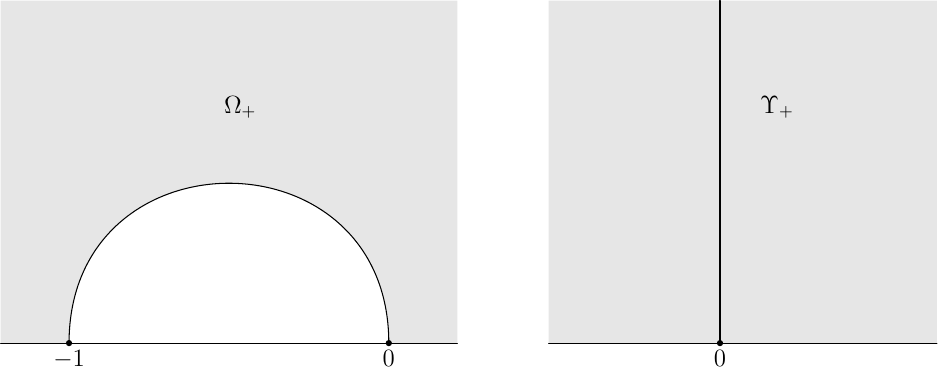}
\end{center}
\caption{Domains $\Om_+$ and $\Ups_+$ in
Lemma~\ref{lem-Avbg}.}\label{fig-OU}
\end{figure}

\begin{proof}
The definition of $|_{\rho v_k,k}\av\pm$ follows the same scheme as the
definition of $|^\prs_{\rho v_k,s,k}$. Only the power series $F_\pm$ at
$\infty$ is in the variables $z$ and $\bar z$. In the region of
convergence the correspondence \eqref{resrel} is clear. It extends by
analytic continuation in $s$. \smallskip

In \eqref{avasbg} we are interested in the asymptotic behavior as
$y=\im z\uparrow\infty$, whereas part~\eqref{eq:1sav3} of
Proposition~\ref{prop-1sav} concerns expansions for $t$ approaching
$\infty$ through $\RR$.

It suffices to consider the component
$f_{\pm ,l}(z) = \bigl( f_\pm(z),e_l\bigr)_\rho$, and the corresponding
component $F_{\pm, l} = \Ph_{s,k} \, f_{\pm ,l}$. The function
$F_{\pm ,l}$ is given by a power series in $(z+i)^{-1}$ and
$(\bar z^{-1}-i)^{-1} $ that converges on a neighborhood of $\infty$ in
$\proj\CC$.

If $F_{\pm ,l}(\infty)=0$ we have
\[ F_{\pm , l}(z) \= \Ph_{s,k}(z) f_{\pm,l}(z) \= \oh \bigl( y^s
((y+1)^2+x^2)^{-s-1/2} \bigr)\,.\]
This gives in the $+$ case for $ x\geq 0$ and $y\uparrow\infty$
\begin{align*} \sum_{m\geq 0} e^{-2\pi i m \k_l}\, f_{\pm,l}(z+m)
&\;\ll\; y^s \sum_{m\geq 0}\bigl( y^2+(x+m)^2\bigr)^{-s-1/2}
\displaybreak[0]\\
&\;\ll\; y^{-s}\int_{t=x}^\infty (1+t^2)^{-s-1/2} \= \oh(y^{-s})\,.
\end{align*}
For the $-$ case, replace $\sum_{m\geq 0}$ by $-\sum_{m\leq -1}$.

We treat the constant term of $F_{\pm,l}$ at $\infty$ separately. This
needs to be done only for $\k_l\neq0$.
\begin{align*}\sum_{m=0}^\infty& e^{2\pi i m \k_l}\, \Ph_{s,k}(z+m)^{-1}
\= \sum _{m\geq 0} e^{-2\pi i m\k_l}\, y^s \,(z+m+i)^{-s-k/2}\,( \bar
z+m-i)^{-s+k/2}
\displaybreak[0]\\
&\= y^s \sum_{m\geq 0} e^{-2\pi i m \k_l}\, (z+m)^{-2s} \Bigl( 1 +
\oh\bigl( (z+m)^{-1}\bigr)
\Bigr)\,.
\end{align*}
The $\oh$-term gives a convergent sum for $\re s>0$ with the estimate
$\oh(y^{-s})$ like above. The main term is $y^s$ times the Lerch
transcendent $H(2s,e^{-2\pi i \k_l}, z)$; see \eqref{Lerch}. We apply
the asymptotic behavior in Proposition~\ref{prop-Lerch}\eqref{eq:Lerch2} to $z=x+i y$
with $x$ in a compact set, and $y\uparrow\infty$. Since
$\z=e^{-2\pi i \k_l}\neq 1$, the expansion starts with $z^{-2s}$. This
gives the desired result.

The individual terms in the sum~\eqref{avPdefbg} correspond to the
individual terms with $\rest_{s,k} f$. In the region of absolute
convergence the relation in \eqref{resrel} is clear. This relation
extends analytically in~$s$.
\end{proof}

%% file: rwt7-pfcf.tex


\section{From period functions to cuspidal Maass forms}\label{sect-pfcf}
In Sections~\ref{sect-Mcf}--\ref{sect-pi} we carried out the following
steps
\[
\begin{array}{ccccccc}
u&\mapsto&\eta_{s,k}(u)&\mapsto&c^u_\pb&\mapsto&
c^u_\pb(0,\infty)|_{(0,\infty)}
\\
\in \A^0_k(s,\rho v_k)&&\text{see
\eqref{etaks}}&&(\proj\QQ)^2\rightarrow \V{\om^0,0}{\rho
v_k,s,k}(\proj\RR)&&\in \FE^\om_{\rho v_k,s,k}
\end{array}\label{usteps}
\]
In this section the aim is to go from a period function to an
automorphic form, using cocycles with values in the boundary germs,
instead of in analytic functions on intervals~in $\proj\RR$.
We state the main theorem.
\begin{thm}\label{thm-al}
Let $k\in \RR$ and $s\in \CC$ such that $\re s\in (0,1)$ and $s\not\equiv \pm k/2\bmod 1$.
Let $\rho$ be a finite-dimensional unitary representation of $\Gm=\SL_2(\ZZ)$.
Then the linear map $\A^0_k(s,\rho v_k) \rightarrow \FE^\om_{\rho v_k,s,k}$
given by $u \mapsto c^u_\pb(0,\infty)|_{(0,\infty)}$ with
$c^u_\pb(\cdot,\cdot)$ as in~\eqref{cparb} is bijective.
\end{thm}

We emphasize that the proof of Theorem~\ref{thm-al} provides the inverse map. However, the statement of the inverse map is rather involved; it involves a transition to boundary germs. The proof is split into a number of steps. At the end of this section we
give a recapitulation. In this section we use $s$, $k$ and $\rho$ as
indicated in the theorem.

\rmrk{Use of the Farey tesselation}We use the \il{Fartess}{Farey
tesselation}Farey tesselation \il{Far}{$\Fa$}$\Fa$ illustrated in
Figure~\ref{fig-Farey}. By
\il{XjF}{$X_j^\Fa$}$X_0^\Fa$ we denote the set of vertices, by
$X^\Fa_1$ the set of edges $e_{\xi,\eta}$ in the tesselation, and by
$X^\Fa_2$ the set of cells. The group $\bar \Gm=\Gm/\{\pm I_2\}$ acts
on these sets, and $X_0^\Fa=\Gm\, \infty$,
$X_1^\Fa = \Gm\, e_{0,\infty}$, and $X_2^\Fa=\Gm C_{0,\infty,-1}$,
where $C_{0,\infty,-1}$ denotes the cell with vertices $0$, $\infty$
and $-1$. For each $e\in X_1^\Fa$ we choose an orientation, and use
$e_{\eta,\xi}=-e_{\xi,\eta}$ to handle the opposite orientation. Since
$S\, e_{0,\infty}= e_{\infty,0}=-e_{0,\infty}$, all oriented edges can
be written as $\gm^{-1} e_{0,\infty}$ with a unique $\gm\in \bar \Gm$.

Like in \cite[\S11.1, 11.3]{BLZ15} the complex
\il{FFar}{$F_\bullet^\Fa$}$F_\bullet^\Fa=\CC[X_\bullet^\Fa]$ forms a
resolution of $\bar\Gm$-modules that leads to the parabolic cohomology
groups
$H^j\bigl( F_\bullet^\Fa; M\bigr)\cong H^j_\pb \bigl( \bar \Gm; M\bigr)  $,
$j=0,1,2$, of $\bar\Gm$-modules $M$. (Here we do not need the mixed
cohomology groups used in~\cite{BLZ15}.) Furthermore,
$H^j_\pb \bigl( \bar \Gm; M\bigr)  = H^j_\pb \bigl(  \Gm; M\bigr) $ for
modules in which the action of $-I_2$ is trivial.)

\begin{lem}\label{lem-bt}
Let $\W{\om^0,0}{\rho v_k,s,k}(\proj\RR)$ correspond to
$\V{\om^0,0}{\rho v_k,s,k}(\proj\RR)$ under the isomorphism
$\rest_{s,k}$ in Theorem~\ref{thm-resiso}. There is an injective linear map
\il{bt}{$\bt_{s,k}$}
\[
\bt_{s,k}\colon  \FE^\om_{\rho v_k,s,k}
\rightarrow Z^1\bigl( F_\bullet^\Fa;\W{\om^0,0}{\rho v_k,s,k}(\proj\RR)\bigr)\,.
\]
\end{lem}
\begin{proof}A period function $f\in \FE^\om_{\rho v_k,s,k}$ is a
real-analytic function on $(0,\infty)$. The definition in
\eqref{def-FEom} implies that
\be \tilde f(t) \= \begin{cases}f(t)&\text{ for }t\in [0,\infty]\,,\\
-f|^\prs_{\rho v_k,s,k}S (t) &\text{ for }t\in [\infty,0]_c
\end{cases}\ee
is a continuous function on $\proj\RR$ that has values in the
$1$-eigenspace of $|^\prs_{\rho v_k,s,k}(-I_2)$ and satisfies
$\tilde f|^\prs_{\rho v_k,s,k}S=-\tilde f$. We can check that
$\tilde f = \tilde f|^\prs_{\rho v_k,s,k}(T+T')$ on $\proj\RR$. We
determine \il{cF}{$c_\Fa$}$c_\Fa 
\in Z^1\bigl( F_\bullet^\Fa; \V{\om ^0,0}{\rho v_k,s,k}(\proj\RR) \bigr)$
by $c_\Fa(e_{0,\infty} ) = \tilde f$, and extending this by
$ c_\Fa(\gm^{-1}e_{0,\infty}) \= \tilde f|^\prs_{\rho v_k,s,k}\gm$. To
see that $c_\Fa$ is a cocycle it suffices to show that
$dc_\Fa \bigl( C_{0,\infty,-1})=0$. That is just the relation
$\tilde f = \tilde f|^\prs_{\rho v_k,s,k}(T+T')$.

Since $\rest_{s,k}$ is an isomorphism of $\Gm$-equivariant sheaves
(Theorem~\ref{thm-resiso}) there is a cocycle\ir{bF}{b_\Fa}
\be\label{bF} b_\Fa \=\rest_{s,k}^{-1} c_\Fa\ee
in $Z^1\bigl( F_\bullet^\Fa;\W{\om^0,0}{\rho v_k,s,k}(\proj\RR)\bigr)$.
It is the zero cocycle only if $\tilde f=0$, and hence $f=0$. So taking
$\bt_{s,k} f = b_\Fa$ gives an injective linear map.
\end{proof}

The boundary germs in $\W{\om^0,0}{\rho v_k,s,k}(\proj\RR)$ are
represented by elements of $\B_{s,k}\bigl( \Om \bigr)$ where
$\Om\subset \proj\CC$ is a neighborhood of $\proj\RR\setminus E$, for a
finite set $E$ of cusps. We define a module of functions on $\uhp$
containing a special choice of these representatives.

By \il{Ervsk}{$\E_{\rho v_k,s,k}$}$\E_{\rho v_k,s,k}(U)$ we denote
$X_\rho \otimes E_{s,k}(U)$ with the action $|_{\rho v_k,k}$ of $\Gm$.

\begin{defn}Let
\il{Gom0}{$\G{\om^0,0}{\rho v_k ,s,k}$}$\G{\om^0,0}{\rho v_k ,s,k}$ be
the space of functions
$f\in \E_{\rho v_k,s,k}\left( \uhp \setminus  E_1\right)$ for a finite
set $E_1\subset X_1^\Fa$ of edges of the Farey tesselation; this set
may depend on $f$. The finitely many connected components $C$ of
$\uhp \setminus E_1$ can be of the following types:
\begin{enumerate}[label=$\mathrm{(\arabic*)}$, ref=$\mathrm{\arabic*}$]
\item\label{eq:type1} The closure $\bar C$ of $C$ in $\proj\CC$ has finite area. In this
case we require that the restriction $f_C$ is in
$\E_{\rho v_k,s,k}(C)= X_\rho \otimes \B_{s,k}(C)$.
\item\label{eq:type2} The closure $\bar C$ of $C$ in $\proj\CC$ contains one or more
intervals $I_j $ in $\proj\RR$. We require that
$f_C\in X_\rho \otimes \B_{s,k}(\Om)$ for an open set
$\Om \subset\proj\CC$ containing $C$ and the intervals $I_j$.
\end{enumerate}

Let $E\subset \proj\QQ$ be the finite set of endpoints of the geodesics
in $E_1$. The functions $\rest_{s,k}\Ph_{s,k}f_C$ determine an element
$\ph \in\V\om{\rho v_k,s,k} \bigl(\proj\RR\setminus E)$. We require
that this element extends continuously to $\proj\RR$.
\end{defn}
\begin{figure}[ht]
\begin{center}
\includegraphics[width=.8\textwidth]{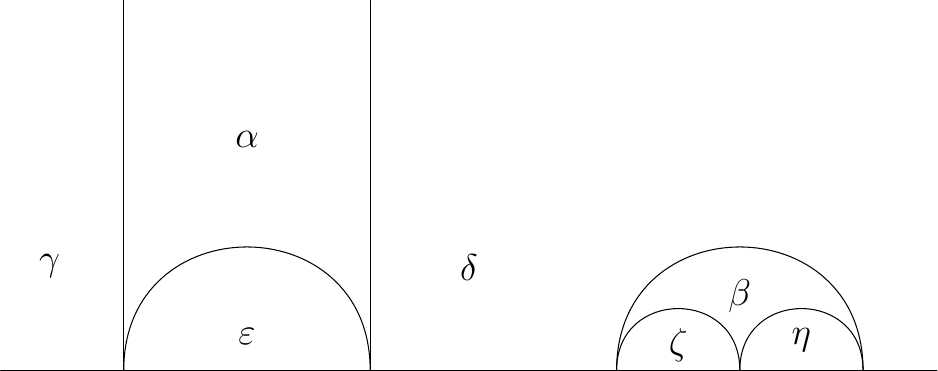}
\end{center}
\caption{Decomposition of $\uhp$ by a finite number of edges in the
Farey tesselation.} \label{fig-Gf}
\end{figure}
In Figure~\ref{fig-Gf} the components indicated by $\al$ and $\bt$ are
of type~\eqref{eq:type1}, the other components are of type~\eqref{eq:type2}. The closure of the
component indicated by $\dt$ contains two intervals in $\proj\RR$.

We note that $\G{\om^0,0}{\rho v_k,s,k}$ is invariant under the action
$|_{\rho v_k,k}$ of $\Gm$. This $\Gm$-module is not equal to the module
$\G{\om^\ast,\mathrm{exc}}{s}$ in \cite[Definition 9.21]{BLZ15}, but we
use it in a similar way.

\begin{figure}[ht]
\begin{center}
\includegraphics[width=.6\textwidth]{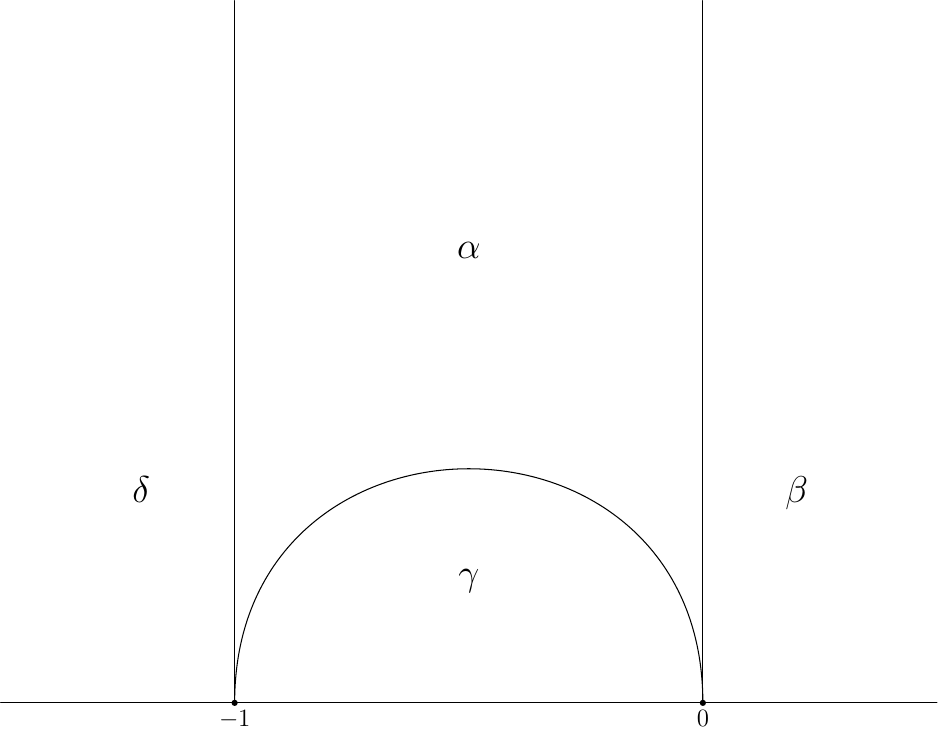}
\end{center}
\caption{The cell $C_{0,\infty,-1}$ in the Farey
tesselation.}\label{fig-cell}
\end{figure}

\begin{lem}\label{lem-cochain}There is a cochain
$\tilde b_\Fa \in C^1\bigl( F^\Fa_
\bullet;\G{\om^0,0}{\rho v_k,s,k}\bigr)$ such that $\tilde b_\Fa(e)$
represents $b_\Fa(e)\in \W{\om^0,0}{ \rho v_k,s,k}(\proj\RR)$ for all
$e\in X_1^\Fa$. Moreover, $
\tilde b_\Fa(e) \in \B_{\rho v_k,s,k}\bigl( \proj\CC\setminus \tilde e\bigr)$
where $\tilde e$ is the closure of the union of $e$ and its complex
conjugate.
\end{lem}
\begin{proof}$b_\Fa(e_{0,\infty})$ is related to $c_\Fa(e_{o,\infty})$
by $\rest_{s,t} b_\Fa(e_{0,\infty} )= c_\Fa(e_{0,\infty})$, and
$c_\Fa(e_{0,\infty})= \tilde f$ is holomorphic on $\CC\setminus i\RR$.
Theorem~\ref{thm-resiso} implies that $b_\Fa(e_{0,\infty})$ has a
representative in $\B_{\rho v_k,s,k}\bigl( \CC \setminus i\RR)$. Take
$\tilde b_\Fa(e_{0,\infty})$ as this representative, and extend the
definition in a $\Gm$-equivariant way. This results in
$\tilde b_\Fa \in C^1\bigl( F^\Fa_\bullet;\G{\om^0,0}{
\rho v_k, s, k}\bigr)$.
\end{proof}

The cochain $\tilde b_\Fa$ represents the cocycle $b_\Fa$. So we have
\[db_\Fa(C_{0,\infty,-1}) \= b_\Fa \bigl( e_{0,\infty}- e_{-1,\infty} -
e_{0,-1}\bigr) \=0\,,\]
which corresponds to the three term relation \eqref{3te}. In
Figure~\ref{fig-cell} this means that $d\tilde b_\Fa(C_{0,\infty,-1}) $
vanishes on the components $\bt$, $\gm$ and $\dt$. On the component
$\al$ the function $d\tilde b_\Fa(C_{0,\infty,-1}) $ can be any
$s(1-s)$-eigenfunction of $\Dt_k$.

\begin{lem}\label{lem-al}Given the cochain $\tilde b_\Fa$ representing
the cocycle $b_\Fa$, there exists a function
$v\in \E_{\rho v_k,s,r}(\uhp)$ satisfying $v|_{\rho v_k , k}\gm =v$ for
all $\gm\in \Gm$. This establishes a linear map\il{alsk}{$\al_{s,k}$}
$\al_{s,k} \colon \FE^\om_{\rho v_k,s,k} \rightarrow \E_{\rho v_k,s,k}(\uhp)^\Gm$.
\end{lem}
\begin{proof}Let $p$ be a positively oriented simple closed path along
edges of the Farey tesselation~$\Fa$, as illustrated in
Figure~\ref{fig-cp}. Evaluating $\tilde b_\Fa(p)$ gives a function on
$\uhp\setminus p$. Since $\tilde b_\Fa$ represents the cocycle $d_\Fa$,
the function $\tilde b_\Fa(p)$ is equal to zero on the components
outside $p$. On the open region $U(p)  \subset \uhp  $ enclosed by $p$
we obtain a function \ $v_p \in \E_{\rho v_k,s,k}\bigl( U(p)\bigr)$,
which depends on the path~$p$.
\begin{figure}[th]
\begin{center}
\includegraphics[width=.9\textwidth]{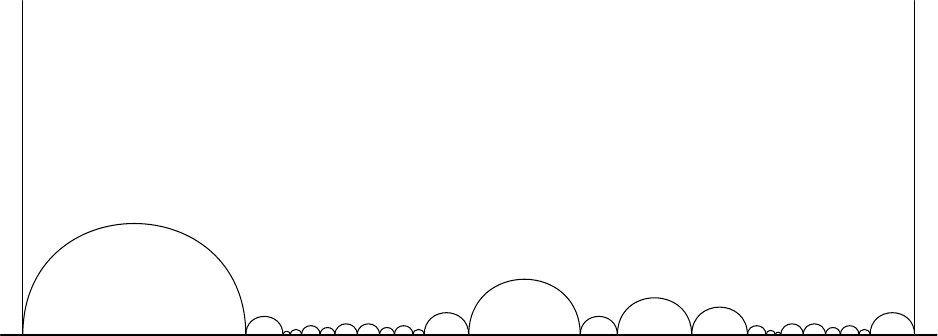}
\end{center}
\caption{Closed path along edges of the Farey tesselation.
(The boundary lines of the wide sector meet each other at $\infty$.)
}\label{fig-cp}
\end{figure}

If $p_1$ and $p_2$ are two paths for which
$U(p_1)\cap U(p_2) \neq \emptyset$ we have $v_{p_1}=v_{p_2}$ on
$U(p_1)\cap U(p_2) $. Indeed, the difference between $p_1$ and $p_2$
can be obtained by adding or subtracting successively a cell
$\gm^{-1} C_{0,\infty,-1}$ to or from the region. This changes only the
value of $\tilde d_\Fa(p_1)$ on $\gm^{-1} C_{0,\infty,-1}$, and
$\gm^{-1} C_{0,\infty,-1} \cap \left(U(p_1)\cap U(p_2)\right) = \emptyset$.

In particular, by making $p$ wider and wider we obtain $v(p)$ on larger
and larger regions. The limit as $p $ tends to $\proj\RR$ gives
$v \in \E_{\rho v_k,s,k}(\uhp)$.

Let $\gm\in \Gm$. For a given relatively compact open region
$V\subset \uhp$ we can take the path $p$ encircling it sufficiently
wide such that $\gm^{-1}p$ encircles $V$ as well. So on the region $V$
we have
\[ v(p) \= \tilde d_\Fa( p ) \= \tilde d_\Fa(\gm^{-1} p ) \= \tilde
d_\Fa(p)|_{\rho v_k,k} \gm \,.\]
So if $z,\gm^{-1} z\in V$ then $v(p)(z) = v|_{\rho v_k,k}\gm(z)$. In the
limit this implies that $v\in \E_{\rho v_k,s,k}(\uhp)^\Gm$.

The map $\bt_{s,k}$ in Lemma~\ref{lem-bt}, followed by
$b_\Fa \mapsto \tilde b_\Fa$ is linear. Also the dependence of $v$ on
$\tilde b_\Fa$ is linear. The composition gives a linear map
\[
\al_{s,k} \colon \FE^\om_{\rho v_k,s,k} \rightarrow \E_{\rho v_k,s,k}(\uhp)^\Gm\,.
\]
\end{proof}

We now have the following situation:
\bad \xymatrix{ \FE^\om_{\rho v_k ,s,k} \ar[r]^(.35){\bt_{s,k}}
\ar[rd]^{\al_{s,k}}
& Z^1\bigl( F_\bullet^\Fa; \W{\om^0,0}{\rho v_k,s,k}(\proj\RR) \bigr)
\ar[d]^{\text{with Lemma~\ref{lem-al}}}
\\
& \E_{\rho v_k,s,k}(\uhp)^\Gm
}
\ead
In Proposition~\ref{prop-pfMcf} we associated a period function to a
Maass cusp form.
\bad \xymatrix{\FE^\om_{\rho v_k ,s,k}\\
&\A^0_k(s,\rho v_k) \ar[lu]_{\pf} }
\ead
The following lemma states that $\al_{s,k}$ is proportional to a left
inverse of~$\pf$.

\begin{lem}\label{lem-osi}If $f\in \FE^\om_{\rho v_k,s,k}$ is the period
function of the Maass cusp form $u\in \A^0_k(s,\rho v_k)$, then
\be \al_{s,k} f = \frac{2\pi i}{b(s,k)}\, u\,,\ee
with the meromorphic factor $b(s,k)$ from~\eqref{bsk}.
\end{lem}
\begin{proof}Now $\tilde f$ in the proof of Lemma~\ref{lem-bt} is given
by
\be c_\Fa(e_{0,\infty} )\= \tilde f \= \int_0^\infty \eta_{s,k}(u) \=
\int_0^\infty \bigl[ u, R_{s,k} \bigr]_k \,, \ee
and hence for each edge $e\in X_1^\Fa$
\be c_\Fa(e) \= \int_e \bigl[u, R_{s,k}\bigr]_k\,.\ee
See Proposition~\ref{prop-P}, \eqref{cparb} and~\eqref{etaks}. With
Proposition~\ref{prop-qR} and the bijectivity of $\rest_{s,k}$ in
Theorem~\ref{thm-resiso} we obtain
\be \bigl[u,R_{s,k}(z,\cdot)\bigr]_k \= \frac1{b(s,k)}\, \bigl[
u,q_{s,k}(z,\cdot) \bigr]_k\,,\ee
and for $e\in X_1^\Fa$
\be \tilde b_\Fa(e) \= \frac1{b(s,k)}\int_e \bigl[ u, q_{s,k}(z,\cdot)
\bigr]_k\,. \ee
The exponential decay of $u$ and its derivatives implies the absolute
convergence of these integrals. We have
\be d\tilde b_\Fa ( C_{0,\infty,-1}) \= \tilde b_F
\bigl(e_{0,\infty}+e_{\infty,-1}
+ e_{-1,0}\bigr)\,.\ee
Now we would like to apply Proposition~\ref{prop-psC}. However the
closed curve $\partial
C_{0,\infty,-1}$ is not in contained in~$\uhp$. We can truncate the cell
$C_{0,\infty,-1}$ at its vertices, and apply Proposition~\ref{prop-psC}
to this approximation of $\partial
C_{0,\infty,1}$. The exponential decay of $u$ and its derivatives
implies that the truncation error goes to zero in the limit. The result
is \be d\tilde b_\Fa( C_{0,\infty,-1}) \= \frac{2\pi i}{b(s,k) } \,u
\ee
on the interior of $C_{0,\infty,-1}$. By analyticity this gives the
lemma.
\end{proof}

\begin{lem}\label{lem-cusp}The function $v=\al_{s,k}f$ associated to a
period function $f\in \FE^\om_{\rho v_k,s,k}$ is in
$\A^0_k(s,\rho v_k)$.
\end{lem}
\begin{proof}
We have still to show that $v$ has exponential decay. The equivariance
of $v$ implies that it suffices to give an estimate of $v(x+iy)$ as
$y\uparrow\infty$ for $x$ in an interval of length at least~$1$.

Let $f\in \FE^\om_{\rho v_k,s_0,k}$ and denote by $\tilde b_\Fa$ the
cochain in Lemma~\ref{lem-cochain} with which we built
$v\in \E_{\rho v_k,s_0,k}(\uhp)^\Gm$ as in the proof of that lemma. Let
$h=\tilde b_\Fa(e_{0,\infty})$. So $\rest_{s,h} h = f$.

Proposition~\ref{prop-FE-eftrof} implies that
$f = \bigl( f|^\prs_{\rho v_k,s,k} T'\bigr)|^\prs_{\rho v_k,s,k}\av+$.
With \eqref{resrel} in Lemma~\ref{lem-Avbg} this implies that
$h= \bigl( h|_{\rho v_k,k}\bigr)|_{\rho v_k}\av+$. With \eqref{avasbg}
this implies that
\be \label{esth}h(z) \= \oh(y^{-s}) \qquad \text{as }y\uparrow\infty\ee
uniform for $x$ in compact sets contained in $(0,\infty)$.
\begin{figure}[ht]
\begin{center}
\includegraphics[width=.8\textwidth]{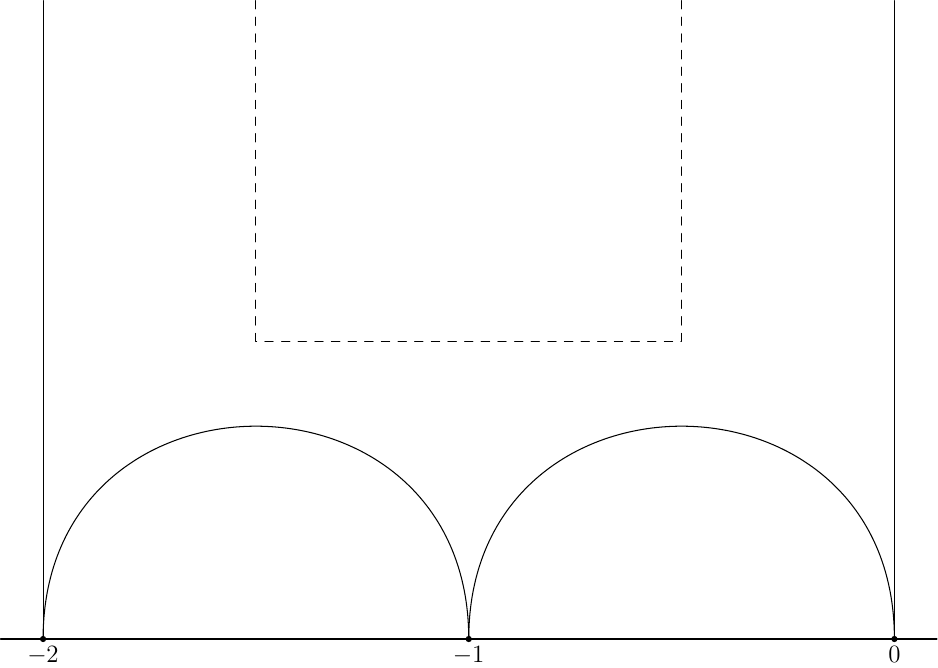}
\end{center}
\caption{Closed path used in the proof of
Lemma~\ref{lem-cusp}.}\label{fig-estdom}
\end{figure}
We use the closed path
\[ p = - T^{-2} e_{0,\infty} - T^{-1}{T'}^{-1} e_{0,\infty} -
{T'}^{-1}e_{0,\infty}
+ e_{0,\infty}\,,\]
sketched in Figure~\ref{fig-estdom}, encircling the union of two cells
in the Farey tesselation. We aim to estimate $v(z)$ for $z$ in the
region bounded by dashed lines. For $z$ inside the path $p$
\bad v(z) & \= - h|_{\rho v_k,k}T^2(z)-h|_{\rho v_k,k}T' T(z)- h|_{\rho
v_k,k}T'(z)+h(z)\\
&\= - \rho(T)^{-2} v_k(T)^{-2} h(z+2)- \ast \,e^{-i k
\arg(z+2)}\,h\Bigl(\frac{z+1}{z+2} \Bigr) \\
&\qquad\hbox{}
- \ast\, e^{-i k \arg(z+1)} \, h\Bigl(\frac z{z+1} \Bigr) + h(z)
\,.
\ead
By $\ast$ we indicate quantities with absolute value~$1$ that do not
depend on~$z$.

For the first and the last term we have the estimate $\oh(y^{-s})$;
see~\eqref{esth}. In the middle terms we have as $y\uparrow\infty$
\begin{align*} h\Bigl( \frac{z+a}{z+a+1}\Bigr) &\=\Ph_{s,k}\Bigl(
\frac{z+a}{z+a+1}\Bigr) ^{-1}\,\oh(1)
\displaybreak[0]\\
& \;\ll\; \frac{y^s}{|z+a+1|^{2s}} \, \Bigl(
\frac{z+a}{z+a+1}+i\Bigr)^{-s-k/2} \, \Bigl( \frac{\bar z+a}{\bar
  z+a+1}+i\Bigr)^{-s+k/2} \displaybreak[0]\\
  & \= \oh\bigl( y^{-s}\bigr)\,.
\end{align*}

The conclusion is that $v(z) = \oh(y^{-s})$ as $y\uparrow\infty$, first
for $\frac12\leq x\leq \frac32$, and then for all $x$ by
$T$-equivariance.

For all components $v_l = \bigl( v, e_l\bigr)_\rho$ this implies that
the exponentially increasing $M$-Whittaker functions in~\eqref{genFt}
do not occur in the Fourier expansion of $v_l$. For $\k_l\in (0,1)$
there are no Fourier terms with order $n$ equal to zero. Then all
Fourier terms are exponentially decreasing, and hence $v_l$ is
exponentially decreasing.

For components $v_l$ with $\k_l=0$ the Fourier term of order zero might
contain a linear combination of $y^s$ and $y^{1-s}$ (or a logarithmic
possibility if $s=\frac12$). For $0<\re s<1$ this is ruled out by the
estimate $\oh(y^{-s})$. (To obtain this we have used an assumption that
holds in this case by Lemma~\ref{lem-f1c}.)
\end{proof}\smallskip

We turn to the injectivity of the map $\al_{s,k}$ from period functions
to Maass cusp forms.

\begin{lem}\label{lem-psplit}For each $f\in \FE^\om_{s,k}$ there is a
holomorphic function $f_\infty$ on
$\CC\setminus \left( i[1,\infty) \cup (-i)[1,\infty) \right)$ such that
\bad f &\= f_\infty - f_\infty|^\prs_{\rho v_k,s,k} S \qquad\text{ on }
\bigl\{z\in \CC\;:\; \re z>0 \bigr\}\,,\\
f_\infty&- f_\infty|^\prs_{\rho v_k,s,k}T \in \V\om{s,k}\bigl( \proj\RR
\bigr)\,.
\ead
\end{lem}
\begin{proof}
Let
$\Om_1 \= \CC\setminus \left( i[1,\infty) \cup (-i)[a,\infty) \right)$
and $\Om_2\= \proj\CC \setminus i[-1,1]$. Then
$\Om_1\cap \Om_2 = \CC\setminus i\RR$ and
$\Om_1\cup\Om_2 = \proj\CC \setminus \{i,-i\}$. We apply \cite[Theorem
1.4.5]{Horm} with $\Om= \proj\CC$ and follow the reasoning in the proof
of \cite[Proposition 13.1]{BLZ15}, obtaining from the holomorphic
function $f$ on $\Om_1\cap \Om_2$, holomorphic functions $A_\infty $ on
$\Om_1$ and $A_0$ on $\Om_2$ such that $A_\infty+A_0 = f$ on
$\Om_1\cap\Om_2$. (H\"ormander requires that we work with open subsets
of $\CC$. That is arranged by a holomorphic transformation of
$\proj\CC$ sending $-i$ to~$\infty$.)

Note that $A_0$ is holomorphic on a neighborhood of $\infty$. Hence
$A_\infty = f-A_0$ is in $\V{\om^0,0}{\rho v_k,s,k} \bigl( \RR\bigr)$,
and analogously
$A_0 \in \V{\om^0,0}{\rho v_k,s,k}\bigl(\proj\RR\setminus \{0\}\bigr)$.
In this way we can conclude that $A_\infty$ and $A_0$ are elements of
$\V{\om^0,0}{\rho v_k,s,k}(\proj \RR)$.

We have
\[ 0 \= f+f|^\prs_{\rho v_k,s,k}S \= \bigl( A_\infty + A_0|^\prs_{\rho
v_k,s,k}S\bigr)
+ \bigl( A_0+A_\infty|^\prs_{\rho v_k,s,k}S\bigr)\,.\]
Considering the singularities of the terms we conclude that
\[ h\= A_\infty + A_0|^\prs_{\rho v_k,s,k}S = - A_0 -
A_\infty|^\prs_{\rho v_k,s,k}S\]
represents an element of $\V\om{\rho v_k,s,k}(\proj\RR)$. We put
\bad f_\infty &\= A_\infty-\frac12 h & &\= \frac12 \bigl( A_\infty -
A_0|^\prs_{\rho v_k,s,k}S \bigr)\,,
\\
f_0 &\= A_0 + \frac12 h &&\= \frac12\bigl( A_0 - A_\infty |^\prs_{\rho
v_k,s,k}\bigr)\,.
\ead
These functions satisfy $f_\infty + f_0 =f$ on $\CC\setminus i \RR$, and
$f_0 |^\prs_{\rho v_k,s,k}S = 
- p_\infty$.

Since $A_\infty \in \V\om{s,k}(\RR)$, we have $p_\infty|^\prs_{
\rho v_k,s,k}T \in \V\om{s,k}(\RR)$ as well. Analogously,
$p_0 \in \V \om{s,k}\bigl(\proj\RR\setminus \{0\}\bigr)$. However,
\[ f_\infty|^\prs_{s,k}(1-T) \= p|^\prs_{s,k}T'-p_0|^\prs_{ \rho
v_k,s,k}(1-T)\,,\]
which show that $f_\infty|^\prs_{s,k}(1-T) \in
\V\om{s,k}\bigl(\proj\RR\setminus\{-1,0\}\bigr)$. Hence
$f_\infty|^\prs_{s,k}(1-T) \in \V\om{s,k}\bigl(\proj\RR\bigr)$.
\end{proof}

\begin{lem}\label{lem-dbcp}Let \ir{Fsk}{\Fbg{s,k}}
\be \label{Fsk}\Fbg{s,k} \= \lim_{\stackrel \Om \rightarrow}
\E_{s,k}(\Om \cap \uhp)
\ee
where $\Om$ runs over the open sets in $\proj\CC$ that contain
$\proj\RR$. Then
\be\label{FEW} \Fbg{s,k} \= \E_{s,k}(\uhp) \oplus \W\om{s,k}\bigl(
\proj\RR\bigr)\,.\ee
\end{lem}
\begin{proof}This result generalizes \cite[(3.3)]{BLZ15}. In
\cite{BLZ15} the decomposition of boundary germs is based on
Proposition~1.1, which we have generalized to weight~$k$ as
Proposition~\ref{prop-psC}. The reasoning leading
to~\cite[(3.3)]{BLZ15} generalizes as well.
\end{proof}

\begin{lem}\label{lem-inj}The map $\al_{s,k}$ in Lemma~\ref{lem-al} is
injective.
\end{lem}
\begin{proof} Any period function $f\in \FE^\om_{\rho v_k,s,k}$
determines a cocycle $\bt_{s,k} p$ on $X_1^\Fa$ with values in
$\W{\om^0,0}{\rho v_k,s,k}(\proj\RR)$ (Lemma~\ref{lem-bt}), determined
by its value on $e_{0,\infty}$. This cocycle is represented by a
$1$-cochain $\tilde b_\Fa$ on $X^\Fa_1$, which is determined by
$h= \tilde b_\Fa(e_{0,\infty}) \in \G{\om^0,0}{\rho v_k,s,k}$.
Lemma~\ref{lem-cochain} implies that $H=\Ph_{s,k}\, h$ extends as a
real-analytic function on $\CC \setminus i\RR$.

The function $h$ is defined (and real-analytic) on
$\uhp\setminus e_{0,\infty}$. This implies that $h_{\rho v_k,k}T$ and
$h_{\rho v_k,k}T'$ are defined (and real-analytic)
on $\uhp \setminus e_{-1,\infty}$ and $\uhp \setminus e_{ 0,-1}$,
respectively.
\begin{center}
\includegraphics[width=.8\textwidth]{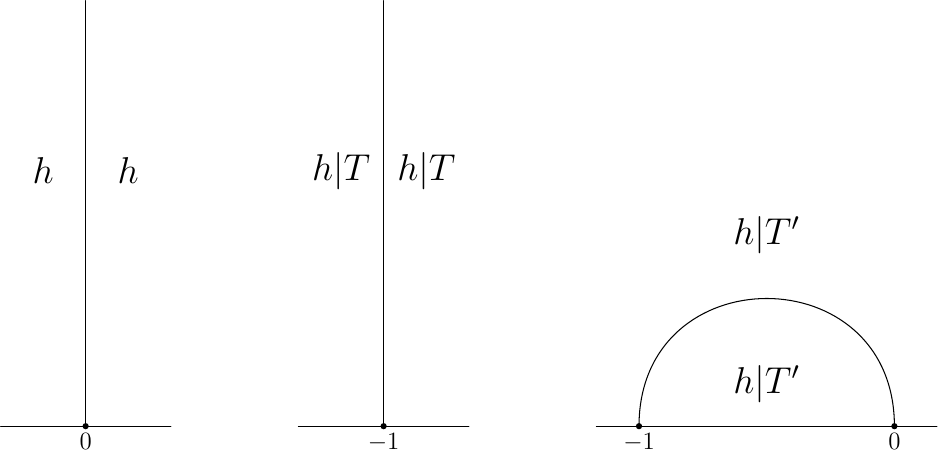}
\end{center}
Hence, the function $l= h-h|_{\rho v_k,k}T -h|_{\rho,v_k,k}T'$ is in
$\E_{s,k}(\uhp\setminus  U )$, where $U$ is the union of the three
geodesic boundary segments of the Farey cell. These geodesics determine
four connected regions in the upper half-plane. The cochain
$\tilde b_\Fa$ with values in $ \G{\om^0,0}{\rho v_k,s,k}$ represents
the cocycle $b_\Fa$ with values in
$\W{\om^0,0}{\rho v_k,s,k}(\proj\RR)$. Since $b_\Fa$ is a cocycle, the
value of
\[ \tilde b_\Fa\bigl( e_{0,\infty}-e_{0,-1} - e_{-1,\infty}\bigr) = l\]
should be zero near $\proj\RR\setminus\{0,\infty,-1\}$. Since $h$ is
real-analytic on $\uhp \setminus e_{0,\infty}$, the function $l$
vanishes outside the triangle with vertices $0$, $\infty$ and $-1$. On
the other hand, inside the triangle, $l$ represents (a multiple of)
$\al_{s,k} f$, by the construction in the proof of Lemma~\ref{lem-al}.

Suppose now that $\al_{s,k} f =0$. Then $l$ is zero on all four
components, and extends as the zero function on $\uhp$. We note that $h\in \E_{\rho v_k,s,k}\bigl(\uhp\setminus e_{0,\infty}\bigr)$,
$h|_{\rho v_k,k}T\in \E_{\rho v_k,s,k}(\uhp\setminus e_{-1,\infty})$, and $h|_{\rho v_k,k}T'\in \E_{\rho v_k,s,k}(\uhp\setminus e_{-1,0})$.
We have $h=h|_{\rho v_k,k}T +h|_{\rho,v_k,k}T'$. The right-hand side of
this equality is in $ \E_{\rho v_k,s,k}(U)$ for some open neighborhood
$U$ of $e_{0,\infty}$ in $\uhp$, hence $h$ is in
$\E_{\rho v_k,s,k}(U)$. This implies that
\be\label{hE} h \in \E_{\rho v_k,s,k}( \uhp)\,. \ee

Theorem~\ref{thm-resiso} states that the restriction $
\rest_{s,k}\colon  \W\om{s,k} \rightarrow \V\om{s,k}$ is bijective. Moreover,
it states that the domain of representatives is preserved.
Lemma~\ref{lem-psplit} splits $f\in \FE^\om_{s,k}$ as
$f=f_\infty-f_\infty|^\prs_{\rho v_k,s,k}S$. This implies that
$h = h_\infty- h_\infty|_{ \rho v_k,k}S$ with
$h_\infty \in \E_{\rho v_k,s,k}\bigl( \uhp\setminus i[1,\infty]\bigr)$
and $h_\infty|_{
\rho v_k,k}S \in \E_{\rho v_k,s,k}\bigl( \uhp\setminus i(0,1]\bigr)$.
From the fact that $h\in \E_{\rho v_k,s,k}(\uhp)$ we conclude that
$h_\infty\in \E_{\rho v_k,s,k}\bigl( \uhp\setminus\{i\}\bigr)$.
Lemma~\ref{lem-psplit} implies that
$h_\infty|_{\rho v_k,k}(1-T)\in \G{\om^0,0}{s,k}(\proj\RR)$ represents
an element of $\W\om{s,k}(\proj\RR)$. The element $h_\infty$ itself
corresponds to the function $f_\infty$, which is holomorphic on a
neighborhood of $\RR$ in~$\CC$. Hence $h_\infty$ represents an element
of $\W\om{s,k}(\RR)$.

Let $h_\infty = \sum_l h_{\infty,l} \, e_l$ as in \eqref{ucomp}, for an
eigenbasis $\{e_l\}$ of $X_\rho$ for $\rho (T)$, as in~\eqref{ucomp}.
The action $h_\infty \mapsto h_\infty|_{\rho v_k,k}$ corresponds to
$h_{\infty,l} \mapsto e^{-2\pi i \k_l}\, h_{\infty,l} |_k T$.

The element $h_{\infty,l}$ represents an element of $\Fbg{s,k} $
in~\eqref{Fsk}. So we can write $h_{\infty,l} = g_l + q_l$ with
$g_l \in \E_{s,k}(\uhp)$ and $q_l \in \G{\om^0,0}{s,k}( \proj\RR)$
representing an element of $\W\om{s,k}(\proj\RR)$. Since the eigenvalue
of $|_{\rho v_k,k}T$ on $e_l$ is $e^{2\pi i \k_l}$ we have
\be h_{\infty,l} \bigm|_k \bigl( 1-e^{-2\pi i \k_l }T\bigr) \= g_l\bigm
|_k \bigl( 1-e^{-2\pi i \k_l }T\bigr) + q_l \bigm|_k \bigl( 1-e^{-2\pi
i \k_l }T\bigr) \,. \ee
This is the situation which is treated in \cite[Lemma 9.23]{BLZ15}. Hence $h_{\infty,l} \in \E_{s,k}\bigl( \uhp\setminus \{i\}\bigr)$.
Moreover, $h_{\infty,l}|_k(1-e^{-2\pi i \k_l} T) $ represents an
element of~$\W\om{s,k}(\proj\RR)$. So the term with
$g_l\in \E_{s,k}(\uhp)$ is zero by the decomposition \eqref{FEW}, and
\be g_l |_k T \= e^{2\pi i \k_l}\, g_l\,.\ee

Combining the component functions to vector-valued functions we obtain
\be h_\infty \= g + q\,,\qquad g|_{\rho v_k,k} T\= g\,,\ee
with $g\in \E_{\rho v_k,s,k}(\uhp)$, and
$q\in \G{\om^0,0}{\rho v_k,s,k}(\proj\RR)$ representing an element of
$\W\om{\rho v_k,s,k}(\proj\RR)$.

Since $h_\infty\in \E_{\rho v_k,s,k}\bigl( \uhp\setminus \{i\})
\cap \G{\om^0,-0}{\rho v_k,s,k}(\RR)$, the function $g$ represents an
element of $\W\om{\rho v_k,s,k}(\RR)$. The invariance of $g$ under
$|_{\rho v_k,k}T$ implies that it has a Fourier expansion of the type
discussed in~\S\ref{sect-Fe}. In this expansion the $W$-Whittaker
functions do not have the right behavior near zero to give a
contribution in $\W\om{\rho v_k,s,k}(\RR)$. In the terms of non-zero
order we are left with multiples of $M$-Whittaker functions, and in the
term of order zero with multiples of $z\mapsto y^s$.

Since $q$ represents an element of $\W\om{\rho v_k,s,k}(\proj\RR)$, this
implies that $g(z) = \oh\bigl ( y^{-s} \bigr)$ as $y\uparrow\infty$.
Each Fourier term inherits this estimate. The function $z\mapsto y^s$
and the $M$-Whittaker functions have larger growth, and hence occur
with coefficient zero. So $g=0$, and $h_\infty = q_\infty$.

We use $h = h|_{\rho v_k,k}(1-S) = q_\infty |_{\rho v_k,k}(1-S)$ to see
that it represents an element of $\W\om {\rho v_k,s,k}(\proj\RR)$. We
combine this with \eqref{hE} to get $h=0$ with Lemma~\ref{lem-dbcp},
and also $f=\rest_{s,k}h=0$\,.\end{proof}

\rmrk{Recapitulation of the proof of Theorem~\ref{thm-al}}
Lemma~\ref{lem-al} gives the central step, in which an invariant
eigenfunction is constructed on the basis of the cochain $\tilde b_\Fa$
representing the cocycle $b_\Fa$. Here it is important that we work
with a cocycle with values in the boundary germs. This allows a
geometrical approach in the upper half-plane.

Before Lemma~\ref{lem-al} we have to construct a cocycle from a given
period function. It is easy to get a cocycle with values in the sheaf
$\V{\om^0,0}{\rho v_k,s,k}$ based on the principal series realized on
the boundary $\proj\RR$ of~$\uhp$. To go over to boundary germs we use
Theorem~\ref{thm-resiso}.

After Lemma~\ref{lem-al} we have to show that the resulting invariant
eigenfunctions have the desired properties. Lemma~\ref{lem-cusp} shows
that the invariant eigenfunctions are indeed Maass cusp form.
Lemma~\ref{lem-osi} shows if we apply the construction to the period
function associated to a Maass cusp form we get back (a non-zero
multiple of) this cusp form. The final steps, in Lemmas
\ref{lem-psplit} and~\ref{lem-inj}, show that a non-zero period
function gives a non-zero Maass cusp form.

%% file: rwt7-JMf.tex


\section{Jacobi Maass forms}\label{sect-JMf} Jacobi Maass form have been
studied by Yang \cite{Yang, Yang2}, and by Pitale \cite{Pit09}.

In this final section we extend the definition of Jacobi Maass forms of
Pitale to real weights, and show that spaces of Jacobi Maass cusp forms
are isomorphic to spaces of vector-valued Maass cusp forms to which we
can apply Theorem~\ref{thm-al}.

\subsection{Jacobi group and its covering group} The \il{Jg}{Jacobi
group}Jacobi group \il{GJ}{$G^J$}$G^J \= \hei\rtimes G$ is the
semidirect product of $G= \SL_2(\RR)$ and the \il{Heig}{Heisenberg
group}Heisenberg group $\hei$. As a topological space,
$\hei\cong \RR^3$:\ir{hei}{\hei}
\be \label{hei}\hei \= \Bigl\{ \hm(x,y,r) \;:\; x,y,r \in
\RR\bigr\}\,;\ee
it has the group operation\il{hm}{$\hm(x,y,r)$}
\be \hm(x_1,y_1,r_1) \,\hm(x_2,y_2,r_2) \= \hm(x_1+x_2,y_1+y_2,r_1+r_2+
x_1 y_2-x_2 y_1)\,.\ee
The semidirect product is given by the following right action of $G$ on
$\hei$.
\be g^{-1}\hm(x,y,r) g \= \hm(a x+c y,b x+d y,r) \qquad\text{ with
}g=\rmatc abcd\,.\ee
We note that $ (ax+cy,bx+dy) = (x,y)\, g$.

One can embed $G^J$ in $\GL_4(\RR)$. Berndt and Schmidt describe $G^J$
by such an embedding, see \cite[\S1.1]{BS}.

The universal covering group of the Jacobi group is obtained as
\[ \tG^J = \hei \rtimes \tG\,, \]
with the universal covering group $\tG$ in~\S\ref{sect-ucg}. The action
of $\tG$ on $\hei$ is given by projection to~$G$ in \eqref{pr}:
\be \tilde g^{-1} h \tilde g \= (\pr \tilde g)^{-1} h (\pr \tilde
g)\qquad h\in \hei,\; \tilde g\in \tG\,.\ee
Since $\hei $ is simply connected, it is its own universal covering
group.

A function $F\in C^\infty(\tG^J)$ has \il{ind}{index}\emph{index}
$m\in \RR_{\neq0}$ and \emph{weight} $k\in \RR$ if
\be \label{wi} F\bigl( \gj \hm(0,0,r) \tkm(\th) \bigr) \= e^{2\pi i m r
+i k \th}\, F (\gj)
\qquad \text{ for }\gj\in \tG^J,\; r,\th \in \RR\,.\ee
It is essential to put $\tkm(\th)$ on the right. The elements
$\hm(0,0,r)$ are central in $\tG^J$ and can be put where it suits us.
The index and the weight are not changed under left translation
$L(\gj g)$ for $\gj\in \tG$.

The Jacobi group acts on $\uhp\times \CC$ by
\bad \rmatc abcd \cdot (\tau,z) &\= \Bigl( \frac {a\tau+b}{c\tau+d},
\frac z{c\tau+d} \Bigr)
&\quad\text{for }&\rmatc abcd\in G\,,\\
\hm(x,y,r) \cdot (\tau,z) &\= ( \tau,z+x\tau+y )&\text{for
}&\hm(x,y,r)\in \hei\,.
\ead
This induces an action of $\tG^J$ on $\uhp\times\CC$ by
\be h \tilde g \cdot (\tau,z) \= h\, \pr\!(\tilde g)\cdot(\tau,z)\,.\ee
In the previous sections we denote elements of the upper half-plane
$\uhp$ by $z$, in accordance we the usual practice in the study of
Maass forms. Here we follow the convention to denote by $\tau$ the
modular variable in $\uhp$, and use $z\in \CC$ as the name of the
elliptic variable.

Like in \eqref{Psk}, there is a map\ir{spsJ}{ \Ps_{k,m}}
\be\label{spsJ} \Ps_{k,m}\colon C^\infty(\uhp\times \CC) \rightarrow
C^\infty(\tG^J)\ee
determined by the following relation between $F$ and $f=\Ps_{k,m} F$:
\be\label{fFJ} f\bigl( \hm(p,q,r)\tppm(\tau)\tkm(\th) \bigr)
\= e^{i k \th
+ 2\pi i m r + 2\pi i m p (p\tau+q)}\,F(\tau,p\tau+q)\,,\ee
with inverse relation
\be \label{fFJi}F(\tau,z) \= e^{-2\pi i m z\,\im(z)/\im(\tau) }\, f
\Bigl( \hm\Bigl( \tfrac{\im z}{\im\tau}, z-\tfrac{\tau \,\im
z}{\im\tau},0\Bigr)
\tppm(\tau) \Bigr)\,. \ee

The right representation of $\tG^J$ by left translation
\il{LJ}{$L$}$L(\gj_1)f \colon  \gj\mapsto f(\gj_1 \gj)$ corresponds
under $\Ps_{k,m}$ to a right representation of $\tG$ on functions on
$\uhp\times\CC$ determined by\ir{|kmtG}{|^J_{k,m}}
\badl{|kmtG} \bigl( F|^J_{k,m} \ell(g) \bigr)(\tau,z) &\= e^{-i k \arg(c
\tau+d)}\,e^{-2\pi i m c z^2 /(c\tau+d)} \\
&\quad\hbox{} \cdot F\bigl( g(\tau,z)\bigr)& \text{for }&g=\rmatc
abcd\in G\,,\\
\bigl( F|^J_{k,m} \z\bigr)(\tau,z) &\= e^{\pi i k n}\,
F(\tau,z)&\text{for }& \z\= \tkm(\pi n)\in \tZ\,,\\
\bigl( F|^J_{k,m} h \bigr) (\tau,z) &\=e^{2\pi i m \bigl(r+ p^2 \tau+ 2
p z+ p q \bigr)}\, F\bigl( h(\tau,z) \bigr)&\text{for }&h=\hm(p,q,r)\,.
\eadl
For $k\in \ZZ$ this is a representation of $\tG^J$ which is trivial on
$\tZ_2$, defined in~\eqref{Z2}. Hence it is a representation of $G^J$.
It is the action used by Pitale, \cite[(4)]{Pit09}. For
$k\in \RR\setminus \ZZ$ the representation $|_{k,m}$ is not trivial on
$\tZ_2$. The operators $|_{k,m} \ell(g)$ for $g\in G$ are similar to
the operators $|_k g$ in \eqref{slash-k}.

\subsection{Discrete subgroup}We use the discrete subgroup
\il{tGmJ}{$\tGm^J$}$\tGm^J= \Ld \rtimes \tGm $ of
$\tG^J = \hei \rtimes \tG$, with the lattice \ir{Ld}{\Ld}
\be\label{Ld} \Ld \= \hei(\ZZ) \= \bigl\{ \hm(\ld,\mu ,\k)\in \hei\;:\;
\ld,\mu,\k\in \ZZ\bigr\}\,,\ee
and the inverse image of the modular group $\tGm= \pr^{-1}\Gm$, as
defined in~\S\ref{sect-ucg}.

Suppose that a function $f$ on~$\tG^J$ has index $m$ and weight $k$ as
indicated in~\eqref{wi}. If $f$ is left-invariant under $\Ld$, then for
$\gj\in \tG^J$:
\[ f(\gj) \= f\bigl( \hm(0,0,1) \gj \bigr) \= f\bigl( \gj \hm(0,0,1)
\bigr) \= e^{2\pi i m } f(\gj)\,.\]
So we need $m$ to be integral for invariance under $\Ld$.

The element $\sm= \widetilde{\rmatc0{-1}10} = \tkm(-\pi/2)$ satisfies
$\sm^4=  \tkm(-2\pi)\neq 1$, although $\rmatc 0{-1}10^4$ is the unit
matrix. We have
\[ f\bigl( \sm^4 \gj \bigr) \= f\bigl( \gj \sm^4\bigr) \= e^{-2\pi i k }
f(\gj)\,.\]
Hence, if $k\in \RR\setminus \ZZ$, then a function cannot be
left-invariant under $\tGm^J$.\

We restrict our attention to left-$\tGm^J$-equivariant functions $f$ on
$\tG^J$ of index $m\in \ZZ_{\geq 1}$ and weight $k\in\RR$ satisfying
\be\label{eqvar} f(\boldsymbol{\gm} \gj) \= \ph(\boldsymbol{\gm})\,
f(\gj) \qquad (\boldsymbol{\gm}\in \tGm^J, \gj \in \tG^J)\ee
for a character $\ph\colon \tGm^J \rightarrow \CC^\ast$ that satisfies
\be \ph(\ld)\=1 \text{ for }\ld\in \Ld\,,\qquad \ph(\sm^4) \= e^{-2\pi i
k}\,.\ee

One such character is \il{chk1}{$\ch_k$}$\ch_k$ as defined in
\eqref{chk} and extended to $\tGm^J$ by taking $\ch_k(\ld)=1$ for
$\ld\in \Ld$. All other such characters are of the form
$\ph = \ph_a \ch_k$ with $a\in \ZZ\bmod 12$\ir{pha}{\ph_a}
\be\label{pha} \ph_a(\tm) \= e^{\pi i a/6}\,,\qquad \ph_a(\sm) \=
e^{-\pi i a/2}\,.\ee
The $\ph_a$ are trivial on $\tZ_2\= \ker\pr$; see~\eqref{Z2}. Hence the
$\ph_a$ correspond to characters of $\Gm=\SL_2(\ZZ)$.

For functions $F$ on $\uhp\times \CC$ we define the action
$|^J_{\ph_a v_k,k,m}$ of $ \Gm^J$ on functions on $\uhp\times \CC$ by
\ir{phvslJ}{|_{\ph_a v_k,k,m}\gm}
\badl{phvslJ} F|_{\ph_av_k,k,m}^J\gm &\= \ph_a(\gm)^{-1} v_k(\gm)^{-1}
\, F|^J_{k,m}\gm& \text{ for } \gm\in \Gm\,,
\\
F|_{\ph_av_k,k,m}^J\ld &\= F|^J_{k,m}\ld&\text{ for } \ld \in \Ld\,.
\eadl
If $F$ corresponds via $\tilde\Ps_{k,m}$ (see \eqref{fFJ}
and~\eqref{fFJi}) to a function $f$ satisfying \eqref{eqvar}, then $F$
is invariant under the action $|_{\ph_a v_k ,k,m}$ of $\Gm^J$. To see
this we use the relation $v_k(\gm) = \ch_k\bigl(\ell(\gm)\bigr)$ and
the fact that $\ph_a$ is a character of $\Gm$.

\subsection{Lie algebra}The group $G^J$ and its covering group $\tG^J$
have the same Lie algebra $\lieJ$. We use the notation of basis
elements of $\lieJ$ as indicated in \cite[\S1.3, \S1.4]{BS}. That are
$Z$, $X_+$ and $X_-$ in the Lie algebra of $G$, already used in
\eqref{opZXX}, and $Z_0$, $Y_+$ and $Y_-$ in the Lie algebra of~$\hei$.
  Each element of the Lie algebra acts on the functions in
  $C^\infty(\tG^J)$ by right differentiation. For any function $f $ of
weight $k$ and index $m$ we have
\be Z_0 f \= 2\pi m\, f\,,\qquad Z f \= k \, f\,.\ee

Under the relation \eqref{fFJ}, the differential operator on $\tG^J$
given by any $\XX\in \lieJ$ commutes with left translations. For given
index and weight it corresponds to a differential operator on
$\uhp\times\CC$ by the relation in~\eqref{fFJi}. \ We use Pitale's
notation \il{kmup}{$\XX^{k,m}$}$\XX^{k,m}$ for this operator. He gives
it explicitly in terms of the coordinates $\tau\in\uhp$ and $z\in\CC$;
see \cite[p~91, 92]{Pit09}. We see for instance that the kernel of
$Y_-^{k,m}$ consists of the functions $F$ on $\uhp\times\CC$ that are
holomorphic in~$z$.

The elements $X_+$ and $X_-$ in $\glie\subset\lieJ$ shift the weight of
functions on $\tG$ by $\pm 2$, respectively. To get weight shifting
operators on $\tG^J$ from~$X_\pm$ we need to calibrate them by adding a
correction term based on the elements $Y_\pm$ in the Lie algebra
of~$\hei$, setting
\be D_\pm \= X_\pm \pm \frac1{4\pi m} Y_\pm^2\,.\ee
These elements are not in the Lie algebra $\lieJ$ itself, but
non-commutative polynomials in Lie algebra elements. Pitale gives the
corresponding weight shifting differential operators $X_\pm^{k,m}$ on
$\uhp\times \CC$.

More complicated is the Casimir operator
\il{Cas}{$\Cas^{k,m}$}$\Cas^{k,m}$, given in \cite[(8)]{Pit09}. It
corresponds to a non-commutative polynomial of degree 3 in elements of
the Lie algebra~$\lieJ$. It has the advantage to commute not only with
$|_{k,m}^J g$ for all $g \in G^J$, but also with $\XX^{k,m}$ for all
elements $\XX$ of the Lie algebra. See also \cite[Proposition
3.1.10]{BS}. Instead of requiring functions to be eigenfunctions of
$\Cas^{k,m}$ we can require that functions are eigenfunctions of
$D^{k+2,m}_- D^{k,m}_+$ with prescribed eigenvalue depending on the
weight. This is similar to the relation~\eqref{DtXpm} for $\SL_2(\RR)$,
which implies for functions of a given weight that eigenfunctions of
$\Dt$ are also eigenfunctions of $X_-X_+$.

\subsection{Jacobi Maass forms} Jacobi Maass forms can be defined as
functions on $\uhp\times \CC$, as Pitale does. One may also view Jacobi
Maass forms as function on the Jacobi group, or on the universal
covering group of the Jacobi group. Both points of view are connected
by the map $\Ps_{k,m}$ in~\eqref{spsJ}. We formulate the definition in
both ways.

First we work on $\tG$:
\begin{defn}\label{defJM}Let $k\in \RR$, $m\in \ZZ_{\geq0}$, $s\in \CC$,
and $a\in \ZZ/ 12\ZZ$. The space
\il{AJ}{$\A^J_{k,m}(s,\ph_a\ch_k), \A^{J,0}_{k,m}(s,\ph_a\ch_k)$
}$\A_{k,m}^J(s,\ph_a \ch_k)$ of \il{JMf}{Jacobi Maass form}\emph{Jacobi
Maass forms} on $\tG^J$ consists of the functions $f\in C^\infty(\tG)$
that satisfy
\begin{enumerate}[label=$\mathrm{(\alph*)}$, ref=$\mathrm{\alph*}$]
\item\label{eq:JFa} $f\bigl( \hm(0,0,r) \gj \tkm(\th) \bigr)
\= e^{2\pi i m r+i k \th}\, f({\gj})$ for $\gj\in\tG^J$, $r,\th\in \RR$.
\item\label{eq:JFb}
$f\bigl(\tilde\gm \gj) \= \ph_a(\tilde\gm)\ch_k(\tilde\gm) \, f(\gj)$
for $\tilde \gm \in \tGm^J$, $\gj\in\tG^J$.
\item\label{eq:JFc} $f$ satisfies the following relations:
\begin{align*} D_- D_+ f &\= \frac{4s^2-(2k+1)^2}{16} f\,,\\
D_+ D_- f &\= \frac{4 s^2-(2k-3)^2}{16}\ f\,, \quad\text{ and }\quad Y_-
f=0\,.\end{align*}
\item\label{eq:JFd} $f\bigl( \tam(t) \gj \bigr) = \oh(t^\al)$ as
$t\uparrow\infty$, for some $\al\in \RR_{>0}$, uniform for $\gj$ in
compact subsets of~$\tG$.
(We recall that $\tam(t) = \tppm(it) $, see \eqref{pr}.)
\end{enumerate}

The subspace $\A^{J,0}_{k,m}(s,\ph_a \ch_k)$ of \il{JMfc}{--- cusp
form}\emph{Jacobi Maass cusp forms} is determined by replacing
\eqref{eq:JFd} by the stronger condition
\begin{enumerate}[label=$\mathrm{(\alph*')}$, ref=$\mathrm{\alph*'}$]
\setcounter{enumi}{3}
\item $f\bigl( \tam(t) \gj\bigr) = \oh(t^{-\al})$ as $t\uparrow\infty$,
for all $\al\in \RR_{>0}$, uniform for $\gj$ in compact subsets
of~$\tG^J$.
\end{enumerate}
\end{defn}

With relation \eqref{fFJi} we obtain the following reformulation:
\begin{defn}\label{defJMHZ}Let $k\in \RR$, $m\in \ZZ_{\geq0}$,
$s\in \CC$, and $a\in \ZZ\bmod 12$. The space
\il{AJ1}{$\A^J_{k,m}(s,\ph_a v_k), \A^{J,0}_{k,m}(s,\ph_a v_k)$
}$\A_{k,m}^J(s,\ph_a v_k)$ of \il{JMf1}{Jacobi Maass form}\emph{Jacobi
Maass forms} on $\uhp\times\CC$ consists of the functions
$F\in C^\infty(\uhp\times\CC)$ that satisfy

\begin{enumerate}[label=$\mathrm{(\Alph*)}$, ref=$\mathrm{\Alph*}$]
\setcounter{enumi}{1}
\item\label{eq:JFl2} $F|^J_{\ph_a v_k,k,m}\gm = F$ for all
$\gm \in \Gm^J$.
\item\label{eq:JFl3} $F$ satisfies the following relations:
\begin{align*}
D_-^{k+2,m} D_+^{k,m} F&\= \frac{4s^2-(2k+1)^2}{16}\, F\,,\\
D_+^{k-2,m} D_-^{k,m} F &\= \frac{4s^2-(2k-3)^2}{16} F\,,\quad\text{and
} Y^{k,m}_- F=0\,.
\end{align*}
\item\label{eq:JFl4}
$F\bigl( \am(y) \cdot (\tau,z) \bigr) \= \oh(y^\al)$ as
$y\uparrow\infty$ for some $\al\in \RR_{>0}$, uniform for $(\tau,z)$ in
compact sets of $\uhp\times \CC$. (Here
$\am(y) = \rmatc {y^{1/2}}00{y^{-1/2}} \in G$.)
\end{enumerate}

The subspace $\A^{J,0}_{k,m}(s,\ph_a v_k)$ of \il{JMfc1}{--- cusp
form}\emph{Jacobi Maass cusp forms} is determined by
replacing~\eqref{eq:JFl4} by the stronger condition
\begin{enumerate}[label=$\mathrm{(\Alph*')}$, ref=$\mathrm{\Alph*'}$]
\setcounter{enumi}{3}
\item $F\bigl( \am(y) \cdot (\tau,z) \bigr) \= \oh(y^{-\al})$ as
$y\uparrow\infty$ for all $\al\in \RR_{>0}$, uniform for $(\tau,z )$ in
compact sets of $\uhp\times \CC$.
\end{enumerate}
\end{defn}

We note that there is no part~(A) in Definition~\ref{defJMHZ}
corresponding to part~\eqref{eq:JFa} in Definition~\ref{defJM}. The
weight and the index are properties of functions on~$\tG^J$, and have
to be fixed in the definition for $\tG^J$. On the other hand, the
weight and the index are not properties of functions on
$\uhp\times \CC$, but parameters in the transformation behavior. We
also note that the character $\ch_k$ of $\tilde \Gm$ in
Definition~\ref{defJM} is replaced in Definition~\ref{defJMHZ} by the
multiplier system $v_k$ on $G$ given by
$v_k(g) = \ch_k\bigl(\ell( g) \bigr)$.

There are a number of differences in comparison with Pitale's
Definition~3.2 in \cite{Pit09}:
\begin{enumerate}[label=$\mathrm{(\arabic*)}$, ref=$\mathrm{\arabic*}$]
\item\label{eq:Pitale1} We allow the weight $k$ to be real, instead of
only integral.
\item\label{eq:Pitale2} Pitale seems to allow the eigenvalue $\ld$ of
$\Cas^{k,m}$ to depend on $F$. In that way, $J^{nh}_{k,m}$ is not a
linear space.
\item\label{eq:Pitale3} Our spaces of Jacobi Maass forms are in
$\hat J_{k,m}^{nh}$ in \cite[(30)]{Pit09}.

Even if we fix the eigenvalue $\ld$, the space of Jacobi Maass forms has
infinite dimension. Pitale does not impose the condition
$Y^{k,m}_- F=0$ in the definition, but imposes it later on.
\item\label{eq:Pitale4} The characterization in \eqref{eq:JFc} says that
Jacobi Maass forms transform under the Lie algebra action in the same
way as a vector $w_0\otimes  v_k$ in the principal series
representation described in \cite[Proposition 3.1.6]{BS}. (There the
weights are integral, but the formulas describe a Lie algebra module if
we let the weight $k$ run through a class in $\RR\bmod 2\ZZ$.)

\item\label{eq:Pitale5} We can show that
$  D_+ D_- f = D_- D_+ f + \bigl( k-\tfrac12\bigr) f$ on functions of
weight $k$ and index $m$ that satisfy $Y_- f=0$. So we can omit the
condition on $D_+ D_- f$ in~c) and~C).

\item\label{eq:Pitale6} We use the condition of quick decay to
characterize Jacobi Maass cusp forms. It takes a consideration of
Fourier expansions to get the formulation used by Pitale.
\end{enumerate}

\subsection{Theta decomposition}According to \cite[Theorem 4.6]{Pit09},
each $F\in \A^J_{k,m}(s,v_0)$, with $k\in \ZZ$ and the trivial
multiplier system $v_0$, is of the form\il{Thmj}{$\Th_{m,j}$}
\badl{Thmj} F(\tau,z) &\= \sum_{j\bmod 2m} \Th_{m,j} (\tau,z)
F_j(\tau)\,,\\
\Th_{m,j} (\tau,z) &\= (\im\tau)^{1/4}\sum_{\al \equiv j/2m\bmod 1}
e^{2\pi i m \tau \al^2}\, e^{4\pi i m z\al}\\
&\= (\im\tau)^{1/4} \sum_{r\equiv j \bmod 2m} e^{\pi i\tau r^2/2m}\,
e^{2\pi i r z}\,
\eadl
with a vector $\bigl(F_j\bigr)_{j\bmod 2m}$ which is a vector-valued
Maass form of weight $k-\frac12$. This opens the way to attach period
functions to Jacobi Maass cusp forms by application of
Theorem~\ref{thm-al}. In this subsection we check that the
decomposition goes through in the case of real weight.

\rmrk{Theta functions on \intitle{\hei}}Let $m\in \ZZ_{\geq 0}$ and
$j\in\ZZ/ 2m$. For each Schwartz function $\ph$ on $\RR$ the theta
function\ir{thHei}{\th^\hei_{m,j}(\ph)}
\be\label{thHei} \th^\hei_{m,j}\bigl(\ph; \hm(p,q,r) \bigr) \=
\sum_{\al\equiv j/2m \,(1)} e^{2\pi i m(r+q(2\al+p))} \ph(p+\al)\ee
is in $C^\infty(\Ld\backslash \hei)$, and the subspace of functions in
$C^\infty(\Ld\backslash \hei)$ that transform according to the
character $\hm(0,0,r) \mapsto e^{2\pi i m r}$ consists of the finite
sum of the form $\sum_{j\bmod 2m}\th^\hei_{m,j}(\ph_j)$ with Schwartz
functions $\ph_j$.

Actually, the map
$\bigl( \ph_j \bigr) \mapsto  \sum_{j\bmod 2m}\th^\hei_{m,j}(\ph_j)$
induces a unitary isomorphism
$L^2(\RR)^{2m} \rightarrow L^2 (\Ld\backslash \hei)_m$, where the
subscript $m$ indicates the subspace given by the character
$\hm(0,0,r) \mapsto e^{2\pi i m r}$.

\rmrk{Theta functions on \intitle{\tG^J}}Let
$f\in C^\infty(\Ld\backslash \tG^J)$. Then there are Schwartz function
$\xi\mapsto \ph_j(\tilde g,\xi)$ parametrized by $\tilde g\in \tG$ such
that
\be \label{gthr} f(h\tilde g) \= \sum_{j\bmod 2m} \th^\hei_{m,j}\bigl(
\ph_j(\tilde g,\cdot);h)\qquad
h\in \hei,\; \tilde g\in \tG\,.\ee

Let us define the following family $\ph^J$ of Schwartz functions on
$\RR$ parametrized by~$\tG$:\ir{phJ}{\ph^J(g,\xi)}
\be\label{phJ} \ph^J\bigl(\tppm(\tau)\tkm(\th),\xi\bigr) \=
\im(\tau)^{1/4} e^{i\th/2}\,e^{2\pi i m \tau \xi^2}\,,\quad
\xi\in\RR\,.\ee
For each $m\in \ZZ_{\neq 0}$ and $j\in \ZZ$ the function in
$C^\infty(\tG^j)$ defined by \ir{thmj}{\th_{m,j}}
\be \label{thmj}\th_{m,j}(h \tilde g) \=
\th^\hei_{m,j}\bigl(\ph^J(\tilde g,\cdot);h)\,,\qquad h\in \hei,\;
\tilde g\in \tG \ee
is left-invariant under the elements of $\Ld$, has index $m$ and weight
$\frac12$. With \eqref{fFJi} the function $\th_{m,j}$ on $\tG^J$
corresponds to the function on $\uhp\times \CC$ that is used
in~\eqref{Thmj}:
\begin{align*}
(\tau,z) &\mapsto e^{-2\pi i m z\,\im(z)/\im(\tau)}\, \th_{m,j}\Bigl(
\hm\Bigl( \frac{\im z}{\im\tau}, z-\frac{\tau\im
z}{\im\tau},0\Bigr)\tppm(\tau) \Bigr)
\displaybreak[0]\\
&\= \sum_{\al \equiv j/2m (1)} e^{4\pi i m \al z} \ e^{-4\pi i m
\al\tau\im(z)/\im(\tau)} \, e^{-2\pi i m \tau \im(z)^2/\im(\tau)^2}\\
&\qquad\hbox{} \cdot
\ph^J \bigl( \tppm(\tau),\al+\im(z)/\im(\tau)\bigr)
\displaybreak[0]\\
&\= \im(\tau)^{1/4}\, \sum_{\al \equiv j/2m (1)} e^{4\pi i m \al z} \
e^{-4\pi i m \al\tau\im(z)/\im(\tau)} \, e^{-2\pi i m \tau
\im(z)^2/\im(\tau)^2} \\
&\qquad\hbox{} \cdot
e^{2\pi i m \tau(\al+\im(z)/\im(\tau))^2}\displaybreak[0]\\
&\= \im(\tau)^{1/4} \sum_{\al\equiv j/2m(1)} e^{4\pi i m \al z}\,
e^{2\pi i m \al^2 \tau} \= \Th_{m,j}(\tau,z)\,.
\end{align*}

We start with the generalization of the theta decomposition \eqref{Thmj}
to real weights, working on the group $\tG^J$ and on the space
$\uhp\times\CC$.

We can check that $\th_{m,j}$ satisfies the conditions \eqref{eq:JFa},
\eqref{eq:JFc} and \eqref{eq:JFd} in Definition~\ref{defJM}, with
spectral parameter $s=1$ or $-1$ and weight~$\frac12$. In particular,
it satisfies
\be\label{thdr} Y_-\, \th_{m,j} \= D_+\, \th_{m,j}\=D_-\,\th_{m,j}
\=0\,.\ee

The behavior of $\Th_{m,j}$ under left-translation is clear for elements
of $\Ld\subset \hei$. It suffices to consider the generators. From
\cite[p.~58, 59]{EZ} we get for the corresponding function on
$\uhp\times \CC$:
\begin{align} \Th_{m,j}(\tau+1,z) &\= e^{\pi i j^2/2m}\,
\Th_{m,j}(\tau,z)
\,,\\
\nonumber
\Th_{m,j}(-1/\tau,z/\tau) &\= \frac{\im(\tau)^{1/4}}{|\tau|^{1/2}}\,
\sqrt{\frac{\tau}{2mi}} e^{2\pi i m z^2/\tau} \, \sum_{j'\bmod 2m}
e^{-2\pi i j j'/2m}\, \frac{ \Th_{m,j'}(\tau,z)}{\im(\tau)^{1/4}}
\displaybreak[0]
\\
&\=(2m)^{-1/2} \,e^{-\pi i/4}\, e^{i \arg(\tau)/2}\, e^{2\pi i m
z^2/\tau}\\
\nonumber
&\qquad\hbox{} \cdot
\sum_{j'\bmod 2m}e^{-\pi i j j'/m}\, \Th_{m,j'}(\tau,z)\,.
\end{align}
So we have $\Th_{m,j}|^J_{1/2,m} \rmatc 1101 = 
e^{\pi i j^2/2m} \Th_{m,j}$ and
\[ \Th_{m,j}|^J_{1/2,m} \rmatr0{-1}10 \= e^{-\pi i/4} \sum_{j'}
\frac{e^{-\pi i j j'/m}}{\sqrt{2m}}\, \Th_{m,j'}\,.\]
For the functions $\th_{m,j}$ on $\tG^J$ this implies the transformation
behavior under left translation by elements of $\tGm^J$. With the row
vector\ir{vth}{\tth_m}
\be \label{vth}\tth_m \= \bigl( \th_{m,1},\ldots , \th_{m,2m} \bigr)\ee
the transformation behavior is determined by
\badl{LM} L(h) \tth _m &\= \tth_m&&\text{for }h\in \Ld\,,\\
L( \tm) \tth_m &\= \tth_m \, M(\tm)\,,\\
L(\sm) \tth_m &\= e^{-\pi i /4 } \tth_m \, M(\sm)\,,
\eadl
where $M(\tm)$ denotes the diagonal matrix with $e^{\pi i j^2/2m}$ at
position $(j,j)$, and $M(\sm)$ the symmetric matrix with
$(2m)^{-1/2} \, e^{-\pi i j j'/m}$ at position $(j,j')$.
\smallskip

We turn to an arbitrary function $f\in C^\infty(\tG^J)$ with weight
$k\in \RR$ and index $m\in \ZZ_{\geq 1}$. We can write $f$ uniquely in
the form
\be \label{fh-rel}f(h \tilde g) \= \tth_m(h \tilde g)\, \vec h_m(\tilde
g)\qquad(h\in \hei,\; \tilde g\in \tG)\,,\ee
with a column vector
$\vec h_m(\tilde g)  = \bigl( h_1(\tilde g), \ldots,h_{2m}(\tilde g)\bigr)$
with $h_j\in C^\infty(\tG)$ of weight~$k-1/2$. We can view $\vec h$ as
a function on $\tG^J$ depending only on the second factor in
$\gj = h \tilde g \in \hei\rtimes \tG$.

The transformation behavior
$L(\tilde \gm) f = \ph_a(\tilde \gm)\ch_k(\tilde \gm) f$ in
Definition~\ref{defJM} takes the form
\be \tth_m(\tilde \gm \gj) \vec h_m(\tilde \gm \gj ) \= \ph_a(\tilde
\gm)\ch_k(\tilde\gm)\, \tth_m( \gj)
\vec h_m(\gj) \qquad (\tilde\gm\in \tGm)\,.\ee
Since the $h_j$ depend only on the factor $\tilde g\in \tG$ in
$\gj =h \tilde g \in \hei\, \tG$, we do not get any condition on
$\vec h_m$ for $\gm \in \Ld$. For $\tm$ and $\sm$ we obtain
\[ M(\tm) \vec h_m(\tm g) \= e^{\pi i(a+k)/6} \,\vec h_m(g)\,,
\qquad e^{-\pi i/4}\, M(\sm) \vec h_m(\sm g) \= e^{-\pi i (a+k)/2}\,
\vec h_m(g)\,. \]
This implies that $\vec h$ has to satisfy the transformation behavior
\be \vec h(\tilde \gm \tilde g) \= \ch_{k-1/2}(\tilde \gm)
\rho_{a,m}(\tilde \gm)\, \vec h_m(\tilde g)\qquad(\tilde\gm\in \tGm, \;
\tilde g\in \tG)\,,\ee
with the representation \il{rhoa}{$\rho_{a,m}$}$\rho_{a,m}$ of $\tGm$
such that $\rho_{a,m}(\tm) $ is $ e^{\pi  i (a/6+1/12)}$ times the
diagonal matrix with entry $e^{-\pi i j^2/2m}$ at position $(j,j)$
(with $1\leq j \leq 2m$), and $\rho_{a,m}(\sm)$ is $e^{-\pi i a/2}$
times the symmetric matrix with $e^{\pi  i j j'/m}/\sqrt{2m}$ at
position $(j,j')$.
\smallskip

We turn to the differential relations in condition~\eqref{eq:JFc} in
Definition~\ref{defJM}. The differentiation by $Y_-$ only involves the
factor $\hei$ of $ \tG^J$, and sends the components $h_j$ of $\vec h_m$
to zero. In view of remark~\eqref{eq:Pitale5} after
Definition~\ref{defJM} we have to look only at the condition
\be \label{dfc}D_- D_+ f = \frac{(s-k-1/2)\,(s+k+1/2)}4 f\,.\ee
We have
\begin{align*}
& D_+ \bigl(\th_{m,j} h_j \bigr) \=\bigl( X_++(4\pi m)^{-1} Y_+^2\bigr)
\bigl( \th_{m,j} h_j\bigr)
\\
& \= (X_+ \th_{m,j}) \, h_j + \th_{m,j} (X_+ h_j)
+ \frac1{4\pi m} \Bigl( (Y_+^2 \th_{m,j})\,h_j +2
(Y_+ \th_{m,j})(Y_+ h_j) + \th_{m,j} Y_+^2 h_j\Bigr)
\displaybreak[0]\\
&\= ( dD_+ \th_{m,j})\, h_j + \th_{m,j}( X_+ h_j) + 0
\\
& \= \th_{m,j} \,(X_+ h_j)\,,
\end{align*}
where we have used that $Y_+ h_j=0$ and, by \eqref{thdr},
$D_+ \th_{m,j}=0$. Proceeding in a similar way, we obtain
\begin{align*}
D_- D_+ (\th_{m,j} h_j) & \= (D_- \th_{m,j}) \, (X_+ h_j)+ \th_{m,j} \,(
X_- X_+ h_j)
\\
&\= \th_{m,j}\,\bigl( X_- X_+ h_j \bigr)\,.
\end{align*}
This means that condition~\eqref{dfc} is equivalent to the condition
\be X_- X_+ h_j \= \frac{s-(k-1/2)-1}2 \, \frac{s+(k-1/2)+1}2 \,
h_j\,.\ee
With $s_1=\frac{s+1}2$ this becomes and $k_1=k-\frac12$ this becomes
\be X_- X_+ h_j \= \Bigl(s_1+\frac{k_1}2\Bigr)\,\Bigl(
s_1-1-\frac{k_1}2\Bigr)\, h_j\,.\ee
In view of \eqref{DtXpm} this is just the differential relation that
Maass forms on $\tG$ of weight $k_1=k-\frac12$ have to satisfy.

\begin{thm}\label{thm-V}Let $m\in \ZZ_{\geq 1}$, $k\in \RR$,
$a\in \ZZ/12$, $s\in \CC$, $\re s\geq 0$, and put $s'=\frac{s+1}2$,
$k'=k-\frac12$. There is a bijective linear map
\be V_{m,k,s} \colon \A^J_{k,m}(s,\ph_a \ch_k) \rightarrow \A_{k'}(
s',\rho_{a,m} v_{k'})\,,\ee
where $\rho_{a,m}$ is the $2m$-dimensional unitary representation of
$\Gm^J = \tGm^J/\tZ_2$ determined by
\bad \label{def:rhoam}
\rho_{a,m}(h) &\= I_{2m} \qquad \text{ for } h \in \Ld\,,\\
\rho_{a,m}(\tm) _{j,j'} &\= \dt_{j,j'}\, e^{\pi i (a/6+1/12-
j^2/2m)}\,,\\
\rho_{a,m}(\sm) _{j,j'}&\= (2m)^{-1/2} \, e^{-\pi i (a/2 + j j'/m)}\,,
\ead
for $j$ and $j'$ running from $1$ to~$2m$. Furthermore,
\[ V_{m,k,s} \,\A^{J,0}_{k,m}(s,\ph_a \ch_k)
\= \A_{k'}^0 (s',\rho_{a,m} v_{k'}\bigr)\,. \]
\end{thm}
\begin{proof}
We have seen already how the equivariance of
$f\in  \A^J_{k,m}(s,\ph_a \ch_k) $ is equivalent to the transformation
behavior of the vector $\vec h$ of functions on $\tG$
in~\eqref{fh-rel}; and also that the relations~\eqref{eq:JFc} in
Definition~\ref{defJM} are equivalent to the differential equations in
Definition~\ref{defMf}. What remains to be done is the relation between
the growth conditions in both definitions.

The theta series
\badl{thser} \th_{m,j} & \bigl( \hm(p,q,r)\tppm(\tau)\tkm(\th) \bigr)
\\
& \= \im(\tau)^{1/4} e^{i\th/2} \sum_{\al\equiv j/2m\bmod 1} e^{2\pi i
m\bigl(r+ q(2\al+p)\bigr)}\, e^{2\pi i m \tau(p+\al)^2}
\eadl
has polynomial growth. Hence polynomial growth of all $h_j$ implies
polynomial growth of $f$, and quick decay of all $h_j$ implies quick
decay of~$f$.

Consider a fixed value of $j$. With $p= -\frac j{2m}$ the theta series
has one term that is a non-zero multiple of $\im\tau^{1/4}$. Hence
polynomial growth of $f$ implies that this $h_j$ has at most polynomial
growth. If $f$ is a Jacobi Maass cusp form, then this $h_j$ has quick
decay. Doing this for all $2m$ values of $j$, we get the desired
equivalence.
\end{proof}

In combination with Theorem~\ref{thm-al} we obtain:
\begin{cor}
\label{cor-pfJM}
Let $m\in \ZZ_{\geq1}$, $s\in \CC$, $k\in \RR$, such that
$0\leq \re s<1$ and $s\not\equiv \pm k \bmod 2$, and put
$s'=\frac{s+1}2$, $k'=k-\frac12$. There is a bijective linear map
\[ \A^{J,0}_{k,m}(s,\ph_a \ch_k)\rightarrow
\FE^\om_{\rho_{a,m}v_{k'},s',k'}\,. \]
\end{cor}

Finally, we give an explicit formulation of the period map in terms of
Jacobi Maass forms as functions on $\uhp\times\CC$.

\begin{prop}\label{prop-isoJMPF}
Let $F\in \A^{J,0}_{k,m}(s,\ph_av_k)$ as in Definition~\ref{defJMHZ},
and put
\be \label{Cjdef} C_{\!j}(\tau) \coloneqq \im(\tau)^{-1/4}\int_{z=0}^{1}
e^{-2\pi i j z} F(\tau,z) \, dz
\qquad (j\in \ZZ,\; \tau\in\uhp)\,.\ee
Then $F_{\!j}(\tau) \coloneqq e^{-\pi i j^2\tau/2m} \,  C_{\!j}(\tau)$
depends only on the class of $j$ in $\ZZ/2m\ZZ$. It is the $j$-th Maass
form in the theta decomposition in~\eqref{Thmj}. The period function
associated to $F$ is given by
\be \left(\int_{\tau=0}^{i\infty} \bigl[ F_{\!j}(\tau),
R_{s',k'}(t,\tau)
\bigr]_k\right)_{1\leq j \leq 2m}\,,\ee
with $s'=\frac{s+1}2$, $k'=k-\frac12$.
\end{prop}
\begin{proof}The theta decomposition \eqref{Thmj} can be formulated in
terms of the Jacobi Maass form $F$ on $\uhp\times \CC$ and the
components of the associated vector-valued Maass form
$(F_j)_{j\bmod 2m}$:
\be F(\tau,z) \= \sum_{c=1}^{2m} \Th_{m,c}(\tau,z) F_{\!c}(\tau)\,.\ee
Expanding the theta functions this becomes
\begin{align*}
F(\tau,z) & \= \sum_{c=1}^{2m} F_c(\tau)\sum_{\substack{j\in\ZZ\\j\equiv
c\bmod 2m}} e^{\pi i j^2 \tau/2m} \im(\tau)^{1/4}e^{2\pi i j z}
\\
& \= \sum_{j\in\ZZ} F_{\!j}(\tau)\,\im(\tau)^{1/4}\,e^{\pi i j^2
\tau/2m}e^{2\pi i j z}\,.
\end{align*}
Here, for $j\in\ZZ$, the map $F_{\!j}$ refers to the unique map
$F_{\!c}$ with $c\equiv j\bmod 2m$ with $c\in[1,2m]$. This formula can
be viewed as a Fourier expansion in~$z$. The Fourier expansion has only
terms that are holomorphic in $z$. This corresponds to $Y^{k,m}_- F=0$;
see \cite[p.~91]{Pit09}.

In \eqref{Cjdef} we defined $\im(\tau)^{1/4} C_{\!j}(\tau) $ as the
Fourier coefficient of order $j$. So
\[ C_j(\tau) = e^{\pi i j^2\tau/2m} \, F_{\!j}(\tau)\,.\]
This implies the proposition.
\end{proof}

%% file: rwt7-ind.tex

\newcommand\ind[2]{\item #1\quad\ #2}
\renewcommand\il[1]{\pageref{i-#1}}
\renewcommand\ir[1]{\pageref{#1}}{}


\section*{Index}
\begin{multicols}{2}
\raggedright
\begin{trivlist}\footnotesize
\ind{analytic boundary germ}{\il{abg}}
\ind{argument convention}{\il{ac}}
\indexspace
\ind{exponential decay}{\il{ed}}
\indexspace
\ind{Farey tesselation}{\il{Fartess}}
\ind{Fourier expansion}{\il{Fe}}
\indexspace
\ind{Green's form}{\il{Gsf}}
\ind{--- in disk model}{\il{Gfdm}}
\ind{growth condition}{\il{grc}}
\indexspace
\ind{Heisenberg group}{\il{Heig}}
\indexspace
\ind{index}{\il{ind}}
\ind{Iwasawa decomposition}{\il{Id}}
\indexspace
\ind{Jacobi Maass form on $\tG$}{\il{JMf}}
\ind{--- on $\uhp\times\CC$}{\il{JMf1}}
\ind{Jacobi Maass cusp form on $\tG$}{\il{JMfc}}
\ind{--- on $\uhp\times\CC$}{\il{JMfc1}}
\ind{Jacobi group}{\il{Jg}}
\indexspace
\ind{Lerch transcendent}{\il{Ltrd}}
\indexspace
\ind{Maass form}{\il{Mf}, \il{MftG}}
\ind{Maass cusp form}{\il{Mcf}, \il{McftG}}
\ind{multiplier system}{\il{msys}}
\indexspace
\ind{one-sided average}{\il{osav}}
\indexspace
\ind{period function}{\il{pdfct}}
\ind{--- associated to Maass cusp form}{\il{pfau}}
\ind{Poisson kernel}{\il{Pk}}
\ind{polar decomposition}{\il{pd}}
\ind{polynomial growth}{\il{pg}}
\ind{principal series representation}{\il{psr}}
\indexspace
\ind{quick decay}{\il{qd}}
\indexspace
\ind{restriction morphism}{\il{res}}
\indexspace
\ind{spectral parameter}{\il{spm}}
\indexspace
\ind{three term relation}{\ir{3te}}
\ind{transfer operator, fast}{\il{ftro}}
\ind{---, slow}{\il{sltro}}
\indexspace
\ind{universal covering group}{\il{ucg}}
\indexspace
\ind{weight}{\il{wt}}
\ind{weight shifting operator}{\il{wsho}}
\end{trivlist}
\end{multicols}


\renewcommand\ind[2]{\item $#1$\quad\ #2}

\section*{List of notations}
\renewcommand\ind[2]{\item $#1$\quad\ #2
}
\begin{multicols}{3}
\raggedright
\begin{trivlist}\footnotesize
\ind{\tA\subset\tG,\; \tA_+ \subset \tA}{\il{tAta+}}
\ind{A_k(s,\rho v_k) \supset \A^0_k(s,\rho v_k)}{\il{A0}, \il{AktG}}
\ind{\A^J_{k,m}(s,\ph_a\ch_k), \A^{J,0}_{k,m}(s,\ph_a\ch_k)}{\il{AJ}}
\ind{\A^J_{k,m}(s,\ph_a v_k), \A^{J,0}_{k,m}(s,\ph_a v_k)}{\il{AJ1}}
\ind{\av+,\; \av-}{\ir{av}}
\ind{\tam(y)}{\il{tam}, \il{ta}}
\indexspace
\ind{\B_{s,k}}{\ir{Bdef}}
\ind{b_\Fa}{\ir{bF}}
\ind{b(s,k)}{\ir{bsk}}
\indexspace
\ind{\Cas^{k,m}}{\il{Cas}}
\ind{\CC'}{\il{C'}}
\ind{C^\om_{\rho}(I)}{\il{Comn}}
\ind{c_\Fa}{\il{cF}}
\ind{c^u_\pb}{\ir{cparb}}
\indexspace
\ind{\E_{s,k}}{\il{Esh}}
\ind{\E_{\rho v_k,s,k}}{\il{Ervsk}}
\ind{e_l\in X_\rho}{\il{el}}
\indexspace
\ind{\Fbg{s,k}}{\ir{Fsk}}
\ind{\Fa}{\il{Far}}
\ind{\FE^\om_{\rho v_k,s,k}}{\il{FEom}}
\ind{F_\bullet^\Fa}{\il{FFar}}
\indexspace
\ind{\G{\om^0,0}{\rho v_k ,s,k}}{\il{Gom0}}
\ind{G=\SL_2(\RR)}{\il{G}}
\ind{G^J}{\il{GJ}}
\ind{\tG}{\il{tG}}
\ind{\glie,\;\glie_c}{\il{glie}}
\indexspace
\ind{\hei}{\ir{hei}}
\ind{\uhp}{\il{uhp}}
\ind{H(s,\z,z)}{\ir{Lerch}}
\ind{\hm(\cdot,\cdot,\cdot)}{\il{hm}}
\indexspace
\ind{I_2}{\il{I2}}
\indexspace
\ind{\tK\subset\tG}{\il{tK}}
\ind{\km(\th)}{\il{km}}
\ind{\tkm(\th)}{\il{tkm}}
\ind{k\text{ weight}}{\il{k}}
\indexspace
\ind{\tro_{\rho v_k,s,k}}{\il{tro}}
\ind{\trof{\rho v_k,s,k}}{\ir{trof}}
\ind{L}{\ir{L}, \il{LJ}}
\ind{g\mapsto \ell(g)}{\ir{tg}}
\indexspace
\ind{n(\rho)}{\il{nrho}}
\indexspace
\ind{P(u)}{\il{Pu}}
\ind{\pf}{\ir{pf}}
\ind{\pr:\tG \rightarrow G}{\ir{pr}}
\ind{\ppm(z),\;\tppm(z)}{\il{tppm}}
\indexspace
\ind{Q_{s,k}}{\ir{Qsk}}
\ind{q_{s,k}(z_1,z_2)}{\ir{qsk}}
\indexspace
\ind{R}{\ir{R}}
\ind{R_{s,k}(t,z)}{\ir{Poisk}}
\ind{\rest_{s,k}}{\il{rest}}
\indexspace
\ind{S=\rmatr0{-1}10}{\il{S}}
\ind{\sm \in \tGm}{\il{sm}}
\ind{s}{\il{s}}
\indexspace
\ind{T=\rmatc1101}{\il{T}}
\ind{T'=\rmatc1011}{\il{Ta}}
\ind{\tm \in \tGm}{\il{tm}}
\indexspace
\ind{\V\om{\rho v_k,s,k}}{\il{Vomsh}}
\ind{\V{\om^0,0}{\rho v_k,s,k}}{\il{Vom0}}
\ind{v_k}{\ir{vk}}
\indexspace
\ind{\W\om{s,k}}{\ir{Womsh}}
\ind{\W\om{\tro v_k,s,k}}{\il{Womrvsk}}
\ind{W_\dt}{\ir{Wdt}}
\indexspace
\ind{X_\rho}{\il{Xrho}}
\ind{X_+, \;X_-\in \glie_c}{\il{Xpm}}
\ind{X_{+,k},\; X_{-,k}}{\il{Xpmk}}
\ind{X_j^\Fa}{\il{XjF}}
\ind{x}{\il{x}}
\indexspace
\ind{y}{\il{y}}
\indexspace
\ind{Z\in \glie_c}{\il{Zl}}
\ind{\tZ}{\ir{tZ}}
\ind{\tZ_2}{\ir{Z2}}
\indexspace
\indexspace
\ind{\al_{s,k}}{\il{alsk}}
\indexspace
\ind{\bt_{s,k}}{\il{bt}}
\indexspace
\ind{\Gm}{\il{Gm}}
\ind{\Gm'}{\il{Gm'}}
\ind{\tGm\subset \tG}{\il{tGm}}
\ind{\tGm^J}{\il{tGmJ}}
\indexspace
\ind{\Dt}{\ir{DtXpm}}
\ind{\Dt_k}{\ir{Dtk}}
\indexspace
\ind{\eta_{s,k}(u)}{\ir{etaks}}
\indexspace
\ind{\Th_{m,j}}{\il{Thmj}}
\ind{\th_{m,j}}{\ir{thmj}}
\ind{\tth_m}{\ir{vth}}
\ind{\th^\hei_{m,j}(\ph)}{\ir{thHei}}
\indexspace
\ind{\k_l}{\il{kpl}}
\indexspace
\ind{\Ld}{\ir{Ld}}
\indexspace
\ind{\rho}{\il{rho}}
\ind{\rho_{a,m}}{\il{rhoa}}
\indexspace
\ind{\Ph_{s,k}}{\il{Phsk}}
\ind{\ph_a}{\ir{pha}}
\ind{\ph^J(g,\xi)}{\ir{phJ}}
\indexspace
\ind{\Ps_k}{\ir{Psk}}
\ind{\Ps_{k,m}}{\ir{spsJ}}
\indexspace
\ind{\ch_k}{\ir{chk}, \il{chk1}}
\indexspace
\indexspace 
\ind{[u_1,u_2]_k}{\ir{Gf}}
\ind{(\cdot,\cdot)_\rho \text{ inner product in }X_\rho}{\il{iprho}}
\ind{(\cdot,\cdot)_c\subset\proj\RR}{\il{ci}}
\ind{|_k}{\ir{slash-k}}
\ind{|_{v_k,k}}{\ir{slash-vk}}
\ind{|_{\rho v_k,k}}{ \ir{|rhockk}}
\ind{|^\prs_{s,k}}{\ir{prs-trf}}
\ind{|^\prs_{\rho v_k,s,k}}{\ir{prsrvk}}
\ind{|^J_{k,m}}{\ir{|kmtG}}
\ind{|^J_{\ph_a,v_k,k,m}}{\ir{phvslJ}}
\ind{\XX^{k,m} \text{ for }\XX\in \lieJ}{\il{kmup}}
\end{trivlist}
\end{multicols}